\documentclass[aos,preprint]{imsart}

\RequirePackage[OT1]{fontenc}
\RequirePackage{amsthm,amsmath}
\RequirePackage{natbib}
\RequirePackage[colorlinks,citecolor=blue,urlcolor=blue]{hyperref}

\usepackage{color}
\usepackage{bm}
\usepackage{booktabs}
\usepackage{amsmath,amssymb}
\usepackage{enumerate}
\usepackage{rotating}
\usepackage{color}
\usepackage{soul}
\usepackage{amsfonts}

\newtheorem{theorem}{Theorem}

\newtheorem{algorithm}{Algorithm}

\newtheorem{lemma}{Lemma}

\newtheorem{remark}{Remark}

\def\BB{\mathbf B}
\def\LL{\mathbf L}
\def\XX{\mathbf X}
\def\II{\mathbf I}

\begin{document}
%\title{Rare-event Analysis for Extremal Eigenvalues of the $\beta$-Laguerre Ensemble}
%\author{{\sc Tiefeng Jiang}\footnote{Supported in part by NSF grants
%DMS-1208982 and DMS-1406279, School of Statistics, University of Minnesota, 224 Church
%Street SE, Minneapolis, MN55455, jiang040@umn.edu},~
%{\sc Kevin Leder}\footnote{
%Supported in part by NSF grants CMMI-1362236 and DMS-1224362,
%Industrial and Systems Engineering, University of Minnesota, 111 Church Street SE, Minneapolis, MN55455, lede0024@umn.edu}
% ~and
%{\sc Gongjun Xu}\footnote{
%School of Statistics, University of Minnesota, 224 Church
%Street SE, Minneapolis, MN55455, xuxxx360@umn.edu.
%\newline{\bf Key Words:} importance sampling, extremal eigenvalues, random matrix, $\beta$-Laguerre ensemble.
%\newline AMS (2010) subject classifications: %%
%	65C05, %%% Random matrices (probabilistic aspects)
%60B20. %%% Random matrices (algebraic aspects)
%}
%}
%\maketitle
\begin{frontmatter}

\title{Rare-event Analysis for Extremal Eigenvalues of white Wishart matrices}
\runtitle{Rare-event Analysis of Wishart matrices}

\begin{aug}
\author{\fnms{Tiefeng} \snm{Jiang}\ead[label=e1]{jiang040@umn.edu}},
\author{\fnms{Kevin} \snm{Leder}\ead[label=e2]{lede0024@umn.edu}},
and
\author{\fnms{Gongjun} \snm{Xu}\ead[label=e3]{xuxxx360@umn.edu}}
%\ead[label=u1,url]{http://www.foo.com}}

%\thankstext{t1}{Some comment}
%\thankstext{t2}{First supporter of the project}
%\thankstext{t3}{Second supporter of the project}
\runauthor{Jiang, Leder and Xu}

\affiliation{ University of Minnesota}

\address{Tiefeng Jiang\\
School of Statistics,\\
 University of Minnesota,\\
 224 Church
Street SE, Minneapolis, MN55455, \\
\printead{e1}}
%\phantom{E-mail: jiang040@umn.edu }\printead*{e2}}

\address{Kevin Leder\\
Industrial and Systems Engineering,\\
 University of Minnesota, \\
 111 Church Street SE, Minneapolis, MN55455\\
\printead{e2}}
\address{Gongjun Xu\\
School of Statistics,\\
 University of Minnesota,\\
 224 Church
Street SE, Minneapolis, MN55455, \\
\printead{e3}}

\end{aug}

\begin{abstract}
In this paper we consider the extreme behavior of the extremal eigenvalues  of white Wishart matrices, which plays an important role in multivariate analysis.
In particular, we focus on the case when the dimension of the feature $p$ is much larger than or comparable to the number of observations $n$,
a common situation in modern data analysis.
We provide asymptotic approximations and bounds for the tail probabilities of the extremal eigenvalues. Moreover, we construct efficient Monte Carlo simulation algorithms to compute the tail probabilities.
Simulation results show that our method has the best performance amongst known approximation approaches, and furthermore provides an efficient and accurate way for evaluating the tail probabilities in practice.
\end{abstract}

%\red{When submitting the paper to a ceratin journal, we will mention in the letter to the editor that some specified details will be cut latter. We have the full details in this version just for the referee's convenience. Or possibly we have two versions: one with full detail posted on the web; another with a abbreviated version for submission}.



\end{frontmatter}

\section{Introduction}
In many modern scientific settings data sets are generated where the dimension of the samples is comparable or even larger than the sample size.
Analysis on such multidimensional data frequently involves estimating rare-event probabilities, such as  small tail probabilities of  test statistics.
For instance, in statistical hypothesis testing, consider multiple comparisons with relatively few signals of interest among a large number of null statistics. In order to control the overall false-positive error rate at a certain level, we may need to evaluate a very small marginal $p$-value for each individual test statistic.

%Such problems require the development of new techniques in multivariate analysis and have been the subject of many recent works.

%An important statistical tool to aid in the study of multidimensional data is  the Wishart matrix.

This paper focus on the tail probabilities of extremal eigenvalues of white Wishart matrices, which play an important role in multivariate statistical analysis and have wide applications in many fields, such as image analysis, signal processing, and functional data analysis.
A white Wishart matrix with parameters $\mathbf{\Sigma}=\mathbf{I}_p$ (the $p\times p$ identity matrix), $n$ and $\beta=1$ is the sample covariance matrix $\XX^*\XX$ where $\XX=(x_{ij})_{n\times p}$ and $x_{ij}$ are i.i.d. $N(0,1)$ random variables. The most natural alternative values of $\beta$ are $\beta=2$ for $x_{ij}$ complex valued and $\beta=4$ for $x_{ij}$ quaternion valued.
Most data analysis in statistics focuses on the case when $\beta = 1$.
In engineering and applied science applications, such as signal processing, oceanography, and atmospheric sciences, it is common to use complex valued variables to study two dimensional signals. In these settings the use of $\beta = 2$ is useful for the understanding of statistical properties of the data set.
%while $\beta=4$ has applications
In physics, for quantum systems with a time reversal symmetry $T$, where either $T^2=1$ or $T^2=-1$, the former leads to symmetric matrices ($\beta=1,2$) and the latter leads to symplectic matrices ($\beta=4$).

The largest eigenvalue of a sample covariance matrix gives useful information for distinguishing a ``signal subspace'' of higher variance from the background noise variables \citep{johnstone2001}. In particular, for $n$ i.i.d. $p$ dimensional Gaussian observations following $N(\mathbf 0,\mathbf \Sigma)$, consider testing the null hypothesis that $\mathbf \Sigma=\mathbf I_p$, where $\mathbf  I_p$ is the identity matrix. Following Roy's union intersection principle \citep{roy1953heuristic},  one can take the largest eigenvalue of the sample covariance matrix as the test statistics and reject the null hypothesis for large values.
 Then the corresponding  $p$-value is the tail probability of the largest eigenvalue under $\mathbf \Sigma=\mathbf I_p$. For example see \cite*{patterson2006population} for applications in SNP (single nucleotide polymorphism) data, \cite*{bianchi2011performance}  for applications in detecting single-source with a sensor array, and \cite*{KDS03} for applications in financial market analysis.
Accurate evaluations of such tail probabilities are needed in performing the corresponding statistical analysis and this motivates our study.

\subsection{Problem setting and related studies}
%Large sample work in multivariate analysis has traditionally assumed that $p$ is fixed and $n$ is large.
%Large sample work in multivariate analysis has traditionally assumed that $p$ is fixed and $n$ is large. In this paper we focus on the case when $p/n\to\gamma\in [1,\infty)$ or $p/n\to\infty$. The motivation for this asymptotic regime is the statistical analysis of high dimensional data, where the dimension $p$ of the covariates is possibly much larger than the sample size $n$.  For instance, in the data of ``1000 Genomes Project'' which is by far the most detailed catalogue of human genetic variation, $n$ is usually at the level of $10^3$ and $p$ is at the level of $10^7$ or $10^8$ \citep{10002010map}.
%In such cases classical statistical procedures and theories based on fixed $p$ and large $n$
%are not directly applicable.
%Motivated by several real world phenomena there have been many works that apply random matrix theory to high dimensional data analysis, see, e.g., \cite{candes2005decoding,donoho2006stable,vershynin2010introduction,cai2011limiting,paul2013random} and many others.
For a white Wishart matrix, it is in fact possible to consider arbitrary values of $\beta>0$.
This more general class of matrices is referred to in the literature  as the $\beta$-Laguerre ensemble.
In this work we primarily focus on the largest  eigenvalues of the $\beta$-Laguerre ensemble in the setting of $p\geq n$, $\mathbf{\Sigma} = \mathbf{I}_p$ and arbitrary $\beta>0$. For this setting the $n$ positive eigenvalues of the $\beta$-Laguerre ensemble  $\lambda_{1},\cdots,\lambda_{n}$ are distributed
with probability density function
\begin{equation}\label{density}
f_{n,p,\beta}(\lambda_{1},\cdots,\lambda_{n})=
c_{n,p,\beta}\prod_{1\leq i<j\leq n}|\lambda_i-\lambda_j|^{\beta}
\cdot\prod_{i=1}^n \lambda_i^{\frac{\beta(p-n+1)}{2}-1}\cdot
e^{-\frac{1}{2} \sum_{i=1}^n \lambda_i},
\end{equation}
where $c_{n,p,\beta}$ is a normalizing constant taking the form of
\begin{equation}\label{cnp}
c_{n,p,\beta}= 2^{-\frac{\beta np}{2}}\prod_{j=1}^n \frac{\Gamma(1+\frac{\beta}{2})}
{\Gamma(1+\frac{\beta}{2}j)\Gamma(\frac{\beta}{2}(p-n+j))}.
\end{equation}
In particular, when $\beta = 1, 2$ and $4$, the function $f_{n,p,\beta}(\lambda_{1},\cdots,\lambda_{n})$ in \eqref{density} is the density function of the $n$ positive eigenvalues of  Wishart matrix $\XX^*\XX,$ where $\XX=(x_{ij})_{n\times p}$ and $x_{ij}$'s are i.i.d. standard  $(\beta=1)$, complex $(\beta=2)$, or quaternion $(\beta=4)$ Gaussian random variables (r.v.'s).
See, for example,  \cite{james1964distributions} and \cite{muirhead2009aspects}
  for the cases of $\beta=1$ and $2$,
and \cite{macdonald1998symmetric} and \cite{edelman2005random} for $\beta=4$. See also \cite*{anderson2010introduction} for further discussion and applications.

Let $\lambda_{(1)}>\cdots>\lambda_{(n)}$ be the order statistics of $\lambda_{1},\cdots,\lambda_{n}$. The joint density function of the order statistics is
\begin{equation}\label{orderdensity}
g_{n,p,\beta}(\lambda_1,\cdots,\lambda_n)
=n! f_{n,p,\beta}(\lambda_{1},\cdots,\lambda_{n}) \times I_{(\lambda_1>\cdots>\lambda_n)},
\end{equation}
where $I_{(\cdot)}$ is the indicator function.
In this paper we  focus on the asymptotic approximation and efficient simulation of tail probabilities
$$P(\lambda_{(1)} > px) \mbox{ as } p\to\infty$$
for any $\beta>0$ and $x>\beta$.
In particular, we consider the high-dimensional settings where $p/n\to\gamma\in [1,\infty)$ or $p/n\to\infty$.
%The current paper deals with the efficient estimation of tail probabilities of $\lambda_{(1)}$.

Large sample properties of the largest eigenvalue have been extensively studied in the literature, most of which focus on the asymptotic distribution of $\lambda_{(1)}$ and its large deviation principle.
For the asymptotic distribution of $\lambda_{(1)},$ \cite{johansson2000shape} and \cite{johnstone2001}  studied the cases when  $p/n\to\gamma\in (0, \infty)$ and  $\beta=2$ and $1$,  and showed that the largest eigenvalue
(with proper recentering and rescaling) follows the Tracy-Widom distribution as appeared
in the study of the Gaussian unitary ensemble.  \cite{el2003largest} extended the asymptotic regime to the case when $p/n\to \infty$.
 For general $\beta>0,$ the limiting distribution of $\lambda_{(1)}$ is obtained by \cite*{ramirez2011beta} for the $\beta$-Laguerre ensemble when $p/n\to\gamma\in [1, \infty).$  Recently, \cite{jiang2013approximation} studied the distribution of $\lambda_{(1)}$ when $p/n^3\to \infty$.
The large deviation principle for $\lambda_{(1)}$ has also been studied in the literature;  see, for example, Chapter 2.6 in Anderson et al (2009).  \cite{maida2007large} investigated the large deviations for $\lambda_{(1)}$ of rank one deformations of Gaussian ensembles when $p/n\to\gamma\in [1, \infty)$, corresponding to the $\beta$-Laguerre ensemble with $\beta = 2$. \cite{jiang2013approximation} studied the case when $p/n\to\infty$ and derived the closed form of the large deviation rate function.

 In practice, however,  to estimate the tail probabilities of $\lambda_{(1)}$, especially when the probabilities are small, i.e., rare events occur,
approximations based on the large sample distribution and large deviation results may not be directly applicable or sufficiently precise.
In particular, to our knowledge efficient estimation methods for the tail probabilities of $\lambda_{(1)}$ as well as sharp asymptotic approximations  are still lacking in the literature.

\subsection{Our contributions}
The current paper deals with the efficient estimation of tail probabilities of $\lambda_{(1)}$.
To do so, we study the extreme behaviors of the largest eigenvalue and describe the conditional distribution of $\lambda_{(1)}$ given the occurrence of the event $\{\lambda_{(1)}>px\}$.
In particular, we use a  so-called ``three-step peeling'' technique to approximate the tail probability  (see the proofs of  Theorem \ref{theorem1} and Lemma \ref{lemma0}) and
give asymptotic approximations of $P(\lambda_{(1)}>px)$,
which provides the necessary technical tools for the development and theoretical analysis of Monte Carlo based computational algorithms.

More importantly, from a computational point of view,
we utilize the technique of  \textit{importance sampling} to develop an efficient Monte Carlo estimator of $P(\lambda_{(1)}>px)$.
Importance sampling is commonly used as a numerical tool for estimating rare event probabilities in a wide variety of stochastic systems \citep[see, e.g.,][]{SIE76,AsmKro06,DupLedWang07,ASMGLY07,BlaGly07,LiuXu12,xu2014rare}.
However, to the authors' best knowledge, this is the first use of this technique for estimating rare event probabilities in the spectrum of random matrices.
In order to implement an importance sampling algorithm,
it is necessary to construct an alternative sampling measure (or change of measure) under which the eigenvalues of the $\beta$-Laguerre ensemble are sampled.
%Note that it is necessary to normalize our estimator with a Radon-Nikodym derivative to ensure an unbiased estimator.
Ideally, one develops a sampling measure so that the event of interest is no longer rare under the sampling measure.
The challenge is of course the construction of an appropriate sampling measure; one common heuristic is to utilize a sampling measure that approximates the conditional distribution of $\lambda_{(1)}$ given $\{\lambda_{(1)}>px\}$.

In this paper, we propose a change of measure denoted by $Q$ that approximates the conditional measure $P(\cdot | \lambda_{(1)}>px)$ in total variation when $p$ is much larger than $n$.
The proposed change of measure is not of a classical exponential-tilting form commonly used in light-tailed stochastic systems \cite[e.g.,][]{SIE76,ASMGLY07} and it has features that are appealing both theoretically and computationally.
Our proposed estimators are {\it asymptotically  efficient} for all $p/n\to\gamma\in[1,\infty]$, that is, the second moments of estimators decay at the same exponential rate as the square of the first moments; see Section \ref{Subsec:Sim} for more details.
%Moreover, the measure $Q$ is computationally tractable in the sense that  the distribution of $\lambda$'s under $Q$ has a closed form representation and the Randon-Nikodym derivative $dQ/dP$ can be easily computed.
Simulation studies in Section \ref{simulation} show that the proposed  method has the best performance amongst existing approximation approaches, especially when estimating probabilities of rare-events.

The proposed method can be easily generalized to the estimation of the smallest eigenvalue  $\lambda_{(n)}$. With completely analogous analysis, we provide approximations of the tail probability of $\lambda_{(n)}$, i.e.,
$$P(\lambda_{(n)} < py) \mbox{ as } p\to\infty$$
for any $\beta>0$ and $0<y<\beta$. Moreover, we construct the corresponding efficient simulation algorithms
as shown in Section \ref{Subsec:Min}.

\medskip
The rest of the paper is organized as follows. In Section \ref{main} we present the main results, including asymptotic approximations of $P(\lambda_{(1)} > px)$ as well as efficient simulation algorithms.
In Section \ref{simulation} we illustrate the theoretical results through a simulation study and a real data example.
Detailed proofs of main theorems and supporting lemmas are presented in Section \ref{proofmain} and the Supplementary Material, respectively.

Throughout this paper, we write:
$a_n=O(b_n)$ if $\limsup_{n\to\infty}|a_n|/|b_n|<\infty$,
$a_n=\Theta(b_n)$ if $0<\liminf_{n\to\infty}|a_n|/|b_n|\leq \limsup_{n\to\infty}|a_n|/|b_n|<\infty$,
$a_n=o(b_n)$ if $\lim_{n\to\infty}|a_n|/|b_n|= 0$,
$a_n\sim b_n$ if $\lim_{n\to\infty}|a_n|/|b_n|= 1$,
$a_n\lesssim b_n$ if  $\limsup_{n\to\infty}|a_n|/|b_n|\leq 1$,
$a_n=O_p(b_n)$ if $a_n=O(b_n)$ in probability,
and $a_n=o_p(b_n)$ if $a_n=o(b_n)$ in probability.

\section{Main results}\label{main}
%The asymptotic study of the tail probability $P(\lambda_{(1)} > px)$ is an important task in multivariate analysis.
%Accurate approximations of $P(\lambda_{(1)} > px)$ have not yet been developed in the literature, especially for the case when $p/n\to\infty$.
We are interested in efficiently estimating $P(\lambda_{(1)} > px)$, which converges to 0 as $p\to \infty$. In Section 2.1, we introduce some commonly used efficiency criteria in the literature; in Sections 2.2-2.4, we present the main asymptotic approximation results and the efficient simulation algorithms.

\subsection{Efficiency criteria in rare-event simulation}
 In the context of rare-event simulations \citep[e.g.,][]{SIE76,ASMGLY07},
 it is necessary to consider the relative computational error with respect to the rare-event probability of interest. 
 In particular, a Monte Carlo estimator $L_p$ is said to be {\it asymptotically efficient} in estimating
the rare-event probability  $P(\lambda_{(1)}> px)$ if $E[L_p] = P(\lambda_{(1)} > px)$ and
\begin{equation}\label{logeff}
\lim_{p\to\infty}\frac{\log E[L_p^2]}{ 2\log P(\lambda_{(1)} > px)} = 1.
\end{equation}
Moreover, $L_p$ is said to be {\it strongly efficient} if $E[L_p] = P(\lambda_{(1)} > px)$ and
\begin{equation}\label{strongeff}
\limsup_{p\to\infty}\frac{ E[L_p^2]}{P(\lambda_{(1)} > px)^2} <\infty.
\end{equation}
There is a rich rare-event simulation literature. An incomplete list of recent works includes \cite{AsmKro06,DupLedWang07,BlaGly07,BL08,BlaGlyLed12,ABL09,LiuXu12,LiuXuTomacs,xu2014rare}.
It is interesting to note that the importance sampling measure we construct in this work has a similar structure to that used in \cite{AsmKro06} where they were studying rare events for sums of i.i.d. heavy-tailed random variables.

\begin{remark}
Suppose we plan to estimate $P(\lambda_{(1)}>px)$ with a given relative accuracy, i.e., to compute an estimator $Z_{p}$ such that
\begin{equation}\label{mse}
P\left(\left|Z_{p}/ {P(\lambda_{(1)}>px)}-1\right| > \varepsilon\right)< \delta
\end{equation}
for some prescribed $\varepsilon, \delta>0$.
For an estimator $L_p$, we can simulate $N$ i.i.d.~copies of $L_{p}$, $\{ L^{(j)}_{p}: j=1,...,N\}$  and  obtain the final estimator $Z_{p} = \frac 1 N \sum_{j=1}^{N} L_{p}^{(j)}$. Then, the estimation error is
$|Z_p - P(\lambda_{(1)}>px)|$.
When  $L_p$ is a strongly efficient estimator as defined in (\ref{strongeff}), the averaged estimator $Z_{p}$
has a relative mean squared error equal to $Var^{1/2} (L_{p})/[N^{1/2}P(\lambda_{(1)}>px)]$.
A simple application of Chebyshev's inequality yields that it suffices to simulate  $N=\Theta(\varepsilon^{-2}\delta ^{-1})$ i.i.d. replicates  of $L_{p}$ to achieve the accuracy in \eqref{mse}.
When $L_p$ is an asymptotically efficient estimator, it suffices to sample $N=\Theta(\varepsilon^{-2}\delta ^{-1}P(\lambda_{(1)}>px)^{-\eta})$, for any $\eta>0$, i.i.d. replicates  of $L_{p}$.
Compared with the crude Monte Carlo simulation, which requires
$N=\Theta(\varepsilon^{-2}\delta ^{-1} \allowbreak P(\lambda_{(1)}>px)^{-1})$ i.i.d. replicates,
the efficient estimators substantially reduce the computational cost.
See Section \ref{simulation} for a simulation study  and further discussion.
\end{remark}
Importance sampling is one of the most widely used methods for variance reduction of Monte Carlo estimators. For ease of notation, we use $P$ to denote the probability measure  of the vector $(\lambda_1,\cdots,\lambda_n)$.
The importance sampling estimator is constructed  based on the following identity:
$$P(\lambda_{(1)} > px) = E \Big[1_{(\lambda_{(1)} > px)}\Big]
=E^Q \Big[1_{(\lambda_{(1)} > px)}\frac{dP}{dQ}\Big],$$
where %$1_{(\cdot)}$ is the indicator function,
$Q$ is a probability measure such that the Radon-Nikodym derivative $dP/dQ$ is well defined on the set $\{\lambda_{(1)} > px\}$, and we use  $E$ and $E^Q$ to denote the expectations under the measures $P$ and $Q$, respectively.
Then, the random variable defined by
\begin{equation}\label{estr}
L_p=\frac{dP}{dQ}1_{(\lambda_{(1)} > px)}
\end{equation}
is an unbiased estimator of $P(\lambda_{(1)} > px)$ under the measure $Q$. Note that when generating the estimator \eqref{estr} we sample $\lambda_{(1)}$ according to the new measure $Q$.

If we choose $Q(\cdot)$ to be $P^*_{px}(\cdot):=P(\cdot | \lambda_{(1)} > px)$, the conditional probability measure given $\lambda_{(1)}>px$, then the corresponding likelihood ratio $dP/dQ$ is exactly $P(\lambda_{(1)}>px)$ on the set $\{\lambda_{(1)} > px\}$ and it has zero variance under $Q$.
However, this change of measure is of no practical use since it needs the value of the target probability  $P(\lambda_{(1)}>px)$. Nonetheless, this conditional measure $P^*_{px}$  provides a guideline for constructing an efficient change of measure. If we can find a measure $Q$ that is a good approximation of $P^*_{px}$, we would expect the corresponding estimator $L_p$ defined in \eqref{estr} to be efficient.

In the following, we design such change of measures for two different  cases: $p/n\to\infty$ in
Section \ref{Subsec:Sim}
and $p/n\to\gamma\in[1,\infty)$ in Section \ref{Subsec:Sim2}.
An analogous analysis of the smallest eigenvalue $\lambda_{(n)}$ is provided in Section \ref{Subsec:Min}.

\subsection{Efficient simulation for $P(\lambda_{(1)} > px)$ when $p/n\to\infty$}\label{Subsec:Sim}

To achieve efficient estimates as defined above, we need to approximate and bound the tail probability $P(\lambda_{(1)} > px)$ as well as the second moment of the estimator.
In Section \ref{Subsec:Tail}, we derive asymptotic approximations and bounds of $P(\lambda_{(1)} > px)$ under different conditions.
We design efficient simulation algorithms in Section \ref{subsubsec:Sim} and show that the estimate is efficient in the sense of \eqref{logeff} and \eqref{strongeff}.

\subsubsection{Tail probability approximation of $\lambda_{(1)}$}\label{Subsec:Tail}
%Recall the notation $\lambda_{(1)}$ in (\ref{orderdensity}).
We have the following  approximations for $P(\lambda_{(1)} > px)$ when $p$ is large.
An exact approximation is given in Theorem \ref{theorem1} when $p/n^{5/3}\to\infty$.
For the general case when $p/n\to\infty$, exact approximations are difficult to obtain and we provide tail approximation bounds, which are good enough to establish the efficiency of the simulation algorithm.
%Importance sampling serves as an appealing alternative in that the design and analysis do not require very sharp approximations of $P(\lambda_{(1)} > px)$.
%In particular, the approximation is sharp when $p/n^{5/3}\rightarrow\infty$.
\begin{theorem}\label{theorem1}
Let $x>\beta$. When $p/n^{5/3}\rightarrow\infty$ as $n\to\infty$,
\begin{eqnarray}\label{gxaaaa}
P(\lambda_{(1)} > px) &\sim&
  \exp(B_{n,p,\beta}(x)),
\end{eqnarray}
where $B_{n,p,\beta}(x)$ is defined by
\begin{align*}
B_{n,p,\beta}(x)=&~
 p\left(\frac{\beta}{2}-\frac{\beta}{2}\log\beta-\frac{x}{2}+\frac{\beta}{2}\log x\right)+\frac{\beta n}{2}\log \frac{p}{n} -\frac{\beta+1}{2}\log p\\
&+\beta n\left(-\frac{\log x}{2}+\log(x-\beta)-\frac{\log\beta}{2}+\frac{1}{2}\right)+\frac{1}{2}\log n-\frac{\beta^3n^{2}}{2(x-\beta)^2p}\\
&-(\beta+1)\log(x-\beta)+\frac{\beta}{2}\log ({2}x)-\log (\pi)+\log\Gamma\big(1+\frac{\beta}{2}\big).
\end{align*}

More generally, if $n\to\infty$ and $p/n\to\infty$ we have
\begin{equation}\label{gxgx2}
\log P(\lambda_{(1)} > px)
=B_{n,p,\beta}(x)+ O(1)\frac{ n^{5/2}}{p^{3/2}}.
\end{equation}
%In particular, when $n^{5/3}/p=O(1),$
%$\log P(\lambda_{(1)} > px)= B_{n,p,\beta}(x)+ O(1).$
\end{theorem}
%The proof of the theorem is given in Section \ref{proofmain}. For general case when $n\to\infty$ and $p/n\to\infty$, we have the following asymptotic bounds, which are helpful to establish the efficiency of the simulation algorithm.
%\begin{theorem}\label{theorem1'}
%Let $x>\beta$.  If  $n\to\infty$ and $p/n\to\infty$ then
%$$\log P(\lambda_{(1)} > px)
%=B_{n,p,\beta}(x)+ O(1)\frac{ n^{5/2}}{p^{3/2}}.$$
%In particular, when $n^{5/3}/p=O(1),$
%$\log P(\lambda_{(1)} > px)= B_{n,p,\beta}(x)+ O(1).$
%\end{theorem}

%\begin{remark}
%The above result generalizes the large deviation result derived in \cite{jiang2013approximation}.
%The dominant term in the expression of $B_{n,p,\beta}(x)$ is $p\left(\frac{\beta}{2}-\frac{\beta}{2}\log\beta-\frac{x}{2}+\frac{\beta}{2}\log x\right)<0$ as $x>\beta$. This implies that
%$$\frac{1}{p}\log P(\lambda_{(1)} > px)
%\sim \frac{\beta}{2}-\frac{\beta}{2}\log\beta-\frac{x}{2}+\frac{\beta}{2}\log x$$ as $p/n\to \infty,$ which is Theorem 2 by Jiang and Li (2014).
%\end{remark}

\begin{remark}\label{remark1}
The tail probability approximation results in Theorem \ref{theorem1} provides technical support for the theoretical analysis of our importance sampling algorithm, where one needs to ensure that the exponential decay rate of the estimator variance matches that of the target tail probability.
Although the term $O(1)$ in \eqref{gxgx2} in the general case is not specified,
the developed approximations are sufficient enough to guarantee the asymptotical efficiency of the proposed Monte Carlo methods.
%This is also true for the results in the following Theorems.
Construction of the importance sampling estimator and the corresponding approximation results for the estimator variance will be provided in Section \ref{subsubsec:Sim}.

When $n$ is fixed with $p\to\infty$, from the proof of Theorem \ref{theorem1}, we have the same approximation results as in Theorem \ref{theorem1}.
In addition,
we believe that it is possible to extend the proof to the case when $p=\Theta(n^{1+\epsilon})$ for any $\epsilon>0$, and  derive
sharper asymptotic approximation results than Theorem \ref{theorem1}. However, this involves the calculation of the expectation of $\exp\{-\sum_{i=1}^n(\lambda_i/p-\beta)^k\}$ for $k\geq 3$ (see the proof of Theorem 1 for more details).
These extensions will be considered in future work.
\end{remark}

\subsubsection{Efficient simulation method}\label{subsubsec:Sim}
We characterize the proposed measure $Q$ in (\ref{estr}) through two ways.
First, we describe the simulation of the eigenvalues from $Q$ {by} following a two-step procedure.
\begin{algorithm} \label{algorithm1}
~ The algorithm goes as follows.
\begin{itemize}
\item[Step 1.] Generate matrix
$\LL_{n-1,p-1,\beta}:=\BB_{n-1,p-1,\beta} \BB_{n-1,p-1,\beta}^\top,$
where $\BB_{n-1,p-1,\beta}$ is a bidiagonal matrix defined by
\begin{eqnarray*}
\BB_{n-1,p-1,\beta}=\left(\begin{array}{ccccc}
\chi_{\beta p-\beta} &\\
\chi_{\beta(n-2)}& \chi_{\beta p-2\beta}&\\~\\
 &\ddots&\ddots\\~\\
&&\chi_{\beta}& \chi_{\beta p-\beta(n-1)}&\\
\end{array}\right)_{(n-1)\times(n-1)}.
\end{eqnarray*}
Here all of the diagonal and sub diagonal elements are mutually independent with the distribution of $\chi_a$, the square-root of the chi-square distribution with degree of freedom $a$.  Calculate the corresponding eigenvalues $(\lambda_{2},\cdots,\lambda_{n})$ of $\LL_{n-1,p-1,\beta}$ and the order statistics $\lambda_{(2)}>\cdots>\lambda_{(n)}$.
\item[Step 2.] Conditional on $(\lambda_{(2)},\cdots,\lambda_{(n)})$, sample $\lambda_{(1)}$ from the exponential distribution with density
\begin{equation}\label{july4}
f(\lambda_{(1)}):=\frac{x-\beta}{2x} e^{-\frac{x-\beta}{2x} (\lambda_{(1)}-px\vee\lambda_{(2)})}\cdot I_{(\lambda_{(1)}>px\vee \lambda_{(2)})},
\end{equation}
where $x>\beta$ is the  threshold value in  $P(\lambda_{(1)}>px)$.
\end{itemize}
\end{algorithm}

Let $Q$ be the measure induced by combining the above two-step sampling procedure on $(\lambda_{(1)},\cdots,\lambda_{(n)})$. It is defined on $[0, \infty)^n.$
We next describe it using the
Radon-Nikodym derivative between $Q$ and the original measure $P$.
From \cite{dumitriu2002matrix}, we know the order statistics of the eigenvalues of $\LL_{n-1,p-1,\beta}$ has density function $$g_{n-1,p-1,\beta}(\lambda_{2},\cdots,\lambda_{n})=
(n-1)! f_{n-1,p-1,\beta}(\lambda_2,\cdots,\lambda_n)\times I_{(\lambda_2>\cdots>\lambda_n)}$$
as defined in \eqref{orderdensity}.
Then, the sampled $\lambda_{(1)},\cdots,\lambda_{(n)}$ under $Q$ has density function
\begin{eqnarray}\label{aa1}
g_{n-1,p-1,\beta}(\lambda_{2},\cdots,\lambda_{n})\cdot
\frac{x-\beta}{2x} e^{-\frac{x-\beta}{2x} (\lambda_{1}-px\vee\lambda_{2})} I_{(\lambda_{1}>px\vee\lambda_{2})}.\end{eqnarray}
The corresponding importance sampling estimator following \eqref{estr} is
\begin{eqnarray*}
L_{p} %=\frac{dP}{dQ}1_{(\lambda_{(1)}>px)}
&=&\frac{g_{n,p,\beta}(\lambda_{(1)},\cdots,\lambda_{(n)})1_{(\lambda_{(1)}>px)}}
{g_{n-1,p-1,\beta}(\lambda_{(2)},\cdots,\lambda_{(n)})\cdot
\frac{x-\beta}{2x} e^{-\frac{x-\beta}{2x} (\lambda_{(1)}-px\vee\lambda_{(2)})} I_{(\lambda_{(1)}>px\vee\lambda_{(2)})}}.
\end{eqnarray*}
The joint density $g_{n,p,\beta}(\lambda_{1},\cdots,\lambda_{n})$ equals
\begin{align}\label{boston1}
& I_{(\lambda_1>\cdots>\lambda_n)} \times n! f_{n,p,\beta}(\lambda_{1},\cdots,\lambda_{n})
\\
=&I_{(\lambda_1>\cdots>\lambda_n)}\times n! c_{n,p,\beta}\prod_{1\leq i<j\leq n}|\lambda_i-\lambda_j|^\beta
\cdot\prod_{i=1}^n \lambda_i^{\frac{\beta(p-n+1)}{2}-1}\cdot
e^{-\frac{1}{2} \sum_{i=1}^n \lambda_i}
\notag\\
=& I_{(\lambda_1>\cdots>\lambda_n)}
 n A_n
  \prod_{i=2}^n(\lambda_1-\lambda_i)^\beta
\cdot \lambda_1^{\frac{\beta(p-n+1)}{2}-1}\cdot
e^{-\frac{1}{2} \lambda_1}
\times g_{n-1,p-1,\beta}(\lambda_2,\cdots,\lambda_n),\notag
\end{align}
 where
\begin{equation}\label{An}
A_n= {c_{n,p,\beta}}/{c_{n-1,p-1,\beta}},
\end{equation}
with $c_{n,p,\beta}$ and $c_{n-1,p-1,\beta}$ defined as in  \eqref{cnp}.
Therefore the importance sampling estimator $L_p$ can be written as
\begin{eqnarray}\label{Lp}
L_{p}
&=&\frac{nA_n  \prod_{i=2}^n(\lambda_{(1)}-\lambda_{(i)})^{\beta}
\cdot \lambda_{(1)}^{\frac{\beta(p-n+1)}{2}-1}\cdot
e^{-\frac{1}{2} \lambda_{(1)}}}
{\frac{x-\beta}{2x} e^{-\frac{x-\beta}{2x} (\lambda_{(1)}-px\vee\lambda_{(2)})}\cdot I_{(\lambda_{(1)}>px\vee\lambda_{(2)})}}1_{(\lambda_{(1)}>px)}.
\end{eqnarray}
Under the measure $Q$, $\lambda_{(1)}>px\vee\lambda_{(2)}$ and therefore $L_p$ is well defined.

The measure $Q$ is constructed such that the behavior of the eigenvalues under $Q$ mimics the tail behavior  given the rare event $\{\lambda_{(1)} > px\}$ under $P$.
According to the proposed simulation procedure, the largest eigenvalue is generated from a truncated exponential distribution at the level about $px$ while the other eigenvalues are generated from the original measure.
We have the following theorem to show the efficiency of the proposed measure.
\begin{theorem}\label{theorem2}
{\normalfont (i)}. If $p/n^{5/3}\rightarrow\infty$, the measure $Q$ approximates $P^*_{px}$, the conditional probability measure given $\lambda_{(1)}>px$, in the total variation sense, i.e.,
$$\lim_{n\rightarrow\infty} \sup_{A\in {\cal F}}|Q(A)-P^*_{px}(A)|=0,$$
where  ${\cal F}$ is the $\sigma$-field ${\cal{B}}([0, \infty)^n)$.
In addition, \begin{eqnarray*}
E^Q\left[L_p^2\right]\sim P(\lambda_{(1)}>px)^2.
\end{eqnarray*}

{\normalfont (ii)}.  If $n^{5/3}/p=O(1)$, we have
\begin{eqnarray*}
E^Q\left[L_p^2\right]=O(1) P(\lambda_{(1)}>px)^2,
\end{eqnarray*}
that is, the importance sampling estimate based on $Q$ is strongly efficient.

{\normalfont (iii)}. More generally, if $p/n\to \infty$, we have
$$\frac{\log E^Q\left[L_p^2\right]}
{2\log P(\lambda_{(1)}>px)}\to 1,$$ that is,  the importance sampling estimate  is asymptotically  efficient.
\end{theorem}

\begin{remark}
Theorem \ref{theorem2} shows that the conditional distribution of
$(\lambda_1,$ $\cdots,$ $\lambda_n)$ given $\lambda_{(1)}>px$ essentially behaves like the proposed measure $Q$.
When $n$ is fixed and $p\to\infty$, a similar argument as in the proof of Theorem \ref{theorem2} gives that $E^Q\left[L_p^2\right]\sim P(\lambda_{(1)}>px)^2$ and the importance sampling estimate is strongly efficient.

It is conceived that the total variation distance between the proposed measure $Q$ and the conditional distribution converges to 0 for the general case when $p=\Theta(n^{1+\epsilon})$, $\epsilon>0$. As the discussion in Remark \ref{remark1}, this needs the calculation of the expectation of $\exp\{-\sum_{i=1}^n(\lambda_i/p-\beta)^k\}$ for $k\geq 3$, which we would like to investigate in the future.
\end{remark}

\begin{remark}
From the proof of Theorem \ref{theorem2}, we can see that the estimator is still asymptotically  efficient if in the second step of Algorithm \ref{algorithm1},  we sample $\lambda_{(1)}$ from an alternative  exponential distribution with density
$$
f(\lambda_{(1)}):=J_{\beta,x} e^{-J_{\beta,x} (\lambda_{(1)}-px\vee\lambda_{(2)})}\cdot I_{(\lambda_{(1)}>px\vee \lambda_{(2)})},$$
where the  rate  $J_{\beta,x} $ is some positive constant smaller than
$ (x-{\beta})/{x}.$ However, as shown in Theorem \ref{theorem2}, when $J_{\beta,x} =(x-\beta)/(2x)$, the change of measure approximates the conditional distribution given $\lambda_{(1)}>px$ in total variation when $p$ is large, and it is conceivable that this rate function yields more efficient results than others.
\end{remark}

\begin{remark}
The above results show that the estimator $L_p$ is asymptotically  efficient.
To estimate $P(\lambda_{(1)}>px)$, we simulate $N$ i.i.d.~copies of $L_{p}$, $\{ L^{(j)}_{p}: j=1,...,N\}$  and the final estimator is $Z_{p} = \frac 1 N \sum_{j=1}^{N} L_{p}^{(j)}$. To achieve the accuracy in \eqref{mse}, by the above theorem, at most
we need $N=\Theta(\varepsilon^{-2}\delta ^{-1})$ if $p/n^{5/3}\to\infty$ or $N=\Theta(\varepsilon^{-2}\delta ^{-1}P(\lambda_{(1)}>px)^{-\eta})$, for any $\eta>0$ and $p/n\to\infty$.
\end{remark}

\subsection{Efficient simulation for $P(\lambda_{(1)} > px)$ when $p/n\to\gamma\in [1,\infty)$}\label{Subsec:Sim2}
When $p/n\to\gamma\in [1,\infty)$,  from a direct application of Theorem 2.6.6   (Anderson et al, 2009),
$\lambda_{(1)}/n$ satisfies the large deviation principle in ${\mathbb R}$ with speed $n$ and good rate function
\begin{equation}\label{LDP}
I_\beta(x) =\left\{
\begin{array}{ll}
-\beta \int_{\mathbb{R}} \log|x-y| \sigma_\beta(dy) + \frac{1}{2}x - \frac{1}{2}\beta(\gamma-1)\log x+\alpha_\beta & \mbox{ if } x\geq  x^*;\\
\infty & \mbox{ if } x< x^*,
\end{array}
\right.
\end{equation}
where
$\alpha_\beta= \frac{\beta}{2}[(\gamma+1)(\log\beta-1)+ \gamma\log \gamma],$  and $\sigma_\beta$ is
the Marchenko-Pastur law \citep{marvcenko1967distribution} corresponding to  the empirical distribution of eigenvalues $(\lambda_1/n,\cdots,\lambda_n/n)$ with $x_*= \beta(\sqrt{\gamma}-1)^2$,   
$x^*=\beta(\sqrt{\gamma}+1)^2$; see also, e.g.,  \cite{hiai1998eigenvalue} and  \cite{dumitriu2003eigenvalue}
for more details.

We now consider the tail probability
 $P(\lambda_{(1)}/p>x) = P(\lambda_{(1)}/n> (p/n)x) $ for $\gamma x> x^*$. From the large deviation result, we know $P(\lambda_{(1)}>px)$ converges to 0 as $n\to\infty$. To construct an efficient estimator, the proposed algorithm in Section \ref{subsubsec:Sim} can not be directly applied and we need to modify the change of measure accordingly.
 
The new algorithm is given as follows: keep {\it Step 1} in the algorithm from Section \ref{subsubsec:Sim}. In {\it Step 2}, we sample $\lambda_{(1)}$ from the exponential distribution with density
$$
{f(\lambda_{(1)})=}J_{\beta,x} e^{-J_{\beta,x}\times (\lambda_{(1)}-px\vee\lambda_{(2)})}\cdot I_{(\lambda_{(1)}>px\vee \lambda_{(2)})},$$
where the rate  $J_{\beta,x} >0$ is chosen such that
\begin{equation}\label{rate}
J_{\beta,x} < 1-2\beta \int \frac{1}{\gamma x-y}\sigma_\beta(dy)-\beta\frac{\gamma-1}{\gamma x}.
\end{equation}

The quantity in the right hand side of \eqref{rate} is the derivative of the rate function $2I_\beta$ at $\gamma x>x^*$.  It is positive on $(x^*,\infty)$ due to the fact that  the rate function $I_\beta(x)$ is a convex function with positive second derivative on set $(x^*,\infty)$ and it achieves the minimum 0 at $x^*$ (Theorem 2.6.6, Anderson et al, 2009). Therefore, the constant $J_{\beta,x} $ is well defined.

{Let $\tilde Q$ be the measure induced by combining the above two-step sampling procedure on $(\lambda_{(1)},\cdots,\lambda_{(n)})$. It is defined on $[0, \infty)^n.$}
{By the same argument as in that between (\ref{july4}) and (\ref{boston1}),  we know the corresponding importance sampling estimate $\tilde L_p:=\frac{dP}{d\tilde Q}$ is given by}
\begin{eqnarray}\label{tildeLp}
\tilde L_{p}
&=&\frac{nA_n  \prod_{i=2}^n(\lambda_{(1)}-\lambda_{(i)})^{\beta}
\cdot \lambda_{(1)}^{\frac{\beta(p-n+1)}{2}-1}\cdot
e^{-\frac{1}{2} \lambda_{(1)}}}
{J_{\beta,x} e^{-J_{\beta,x} (\lambda_{(1)}-px\vee\lambda_{(2)})}\cdot I_{(\lambda_{(1)}>px\vee \lambda_{(2)})}}1_{(\lambda_{(1)}>px)}.
\end{eqnarray}
We have the following efficiency result for $\tilde L_p$:
\begin{theorem}\label{Thmconstant}
If $p/n\to\gamma\in [1,\infty)$ then
 the importance sampling estimate $\tilde L_p$ is asymptotically  efficient.
\end{theorem}

\begin{remark}
To achieve strong efficiency results as in Theorem \ref{theorem2}, we need to derive a more accurate approximation of the tail probability $P(\lambda_{(1)}>px)$ as in Theorem \ref{theorem1}. However, the techniques developed in this paper may not be directly applicable, though some of the derived approximation bounds in the auxiliary lemmas (such as bounds in the proof of Lemma 5) can be generalized to the case of $p/n\to\gamma\in [1,\infty)$.
We leave this as a future work. \end{remark}

\subsection{ Efficient simulation for $\lambda_{(n)}$}\label{Subsec:Min}
Recall that $\lambda_{(n)}$ is the smallest eigenvalue of the $\beta$-Laguerre ensemble defined as in (\ref{orderdensity}).
In this section we focus on the probability
$P(\lambda_{(n)} < py) \mbox{ as } p\to\infty$
for any $\beta>0$ and $0<y<\beta$.
We have the following approximation results similar to Theorem \ref{theorem1}.
Their proofs  follow from analogous arguments as in those for $\lambda_{(1)}$ and therefore are omitted.
\begin{theorem}\label{theoremmin}
For $0<y<\beta$ the following hold as $n\to\infty$.\\
(1). If $p/n\to\infty,$
$\log P(\lambda_{(n)} < py)
=B_{n,p,\beta}(y)+ O(1){ n^{5/2}}{p^{-3/2}},$
where $B_{n,p,\beta}(y)$ is defined as in Theorem \ref{theorem1}.\\
(2). If  $p/n^{5/3}=\Theta(1),$
$\log P(\lambda_{(n)} < py)= B_{n,p,\beta}(y)+ O(1)$.\\
(3). If $p/n^{5/3}\rightarrow\infty$,
$P(\lambda_{(n)} < py) \sim \exp(B_{n,p,\beta}(y)).$
\end{theorem}

Since  $\lambda_{max}/p$ and $\lambda_{min}/p$ are all positive, it can be seen from Theorems \ref{theorem1} and \ref{theoremmin} that the two rate functions on $(0, \infty)$ look ``symmetric" with respect to the line $x=\beta.$
The dominant term in the above expression of $B_{n,p,\beta}(y)$ is $p(\frac{\beta}{2}-\frac{\beta}{2}\log\beta-\frac{y}{2}+\frac{\beta}{2}\log y)$, which is negative if $0<y<\beta$ and thus $\exp(B_{n,p,\beta}(y))$  is no more than 1.
In addition, this gives the same exponential decay rate for $\lambda_{(n)}$ as found in Jiang and Li (2014).

To obtain an efficient Monte Carlo estimator of $P(\lambda_{(n)} < py)$, we propose an importance sampling procedure similar to that for the largest eigenvalue $\lambda_{(1)}$.
\begin{itemize}
\item[Step 1.] Generate matrix
$\LL_{n-1,p-1,\beta}$. Calculate the corresponding eigenvalues $(\lambda_{1},\cdots,\lambda_{n-1})$ of $\LL_{n-1,p-1,\beta}$ and the order statistics $\lambda_{(1)}>\cdots>\lambda_{(n-1)}$.
\item[Step 2.] Conditional on $(\lambda_{(1)},\cdots,\lambda_{(n-1)})$, sample $\lambda_{(n)}$ from the distribution with density
$$
f(\lambda_{(n)}):=\frac{\beta-y}{2y} e^{\frac{\beta-y}{2y} (\lambda_{(n)}-py\wedge\lambda_{(n-1)})}\cdot I_{(\lambda_{(n)}<py\wedge \lambda_{(n-1)})}.$$
\end{itemize}

The importance sampling estimator $L_p$ can be written as
\begin{eqnarray*}
L_{p}
&=&\frac{nA_n  \prod_{1\leq i< n}(\lambda_{(i)}-\lambda_{(n)})^{\beta}
\cdot \lambda_{(n)}^{\frac{\beta(p-n+1)}{2}-1}\cdot
e^{-\frac{1}{2} \lambda_{(n)}}}{\frac{\beta-y}{2y} e^{\frac{\beta-y}{2y} (\lambda_{(n)}-py\wedge\lambda_{(n-1)})}\cdot I_{(\lambda_{(n)}<py\wedge \lambda_{(n-1)})}}I_{(\lambda_{(n)} < py)},
\end{eqnarray*}
The efficiency of the above importance sampling estimator is stated in the next theorem.
\begin{theorem}\label{theorem2min}
Assume $0<y<\beta$ and $n\to\infty$. We have \\
(1). if $p/n^{5/3}\rightarrow\infty$,
$E^Q\left[L_p^2\right]\sim P(\lambda_{(n)} < py)^2$;\\
(2). if $p/n^{5/3}=\Theta(1)$, $E^Q\left[L_p^2\right]=O(1) P(\lambda_{(n)} < py)^2$;\\
(3). if $p/n\to\infty$,  $L_p$ is asymptotically efficient.
\end{theorem}

\section{Numerical Study}\label{simulation}
\subsection{Simulation study}
In order to evaluate the actual performance of our algorithms we conduct a numerical study over different $p$ and $n$ values.
We take $\beta=1$ and choose six combinations of $n$ and $p$: $(n,p)=(10, 10^2),$ $(10, 10^3)$, $(10, 10^4)$,  $(50, 10^2)$, $(50, 10^3)$ and $(50, 10^4)$.  We follow Algorithm \ref{algorithm1} to estimate $P(\lambda_{(1)}>px)$ for different values of $x$'s.
 For $(n,p)=(10,100)$ and $(50,100)$, the algorithm in Section \ref{Subsec:Sim2} gives similar results and therefore are not presented.
 Based on the simulation results,
 we would suggest use Algorithm \ref{algorithm1} in practice for
 $p/n\to \gamma\in[1,\infty]$.

Tables \ref{t1} and \ref{t2} show estimated tail probabilities (column ``Est'') along with the estimated standard deviations $Std(L_p)=\sqrt{Var^Q(L_p)}$ (column ``Std'').
The simulation results are based on $10^4$ independent simulations and it takes just a few seconds in the statistical software ``R'' for each case.
 Note that the standard deviation of the final estimate (in the column ``Est.'') is the reported standard deviation (in the column of ``Std.'') divided by $\sqrt{10^4}=100$.

\begin{table}[h!]{\small
\caption{Estimates of $P(\lambda_{(1)}>px)$ for $n=10$ and $p=100, 1000, 10000$. The standard deviation of the estimate ``Est." is Std./100.}
\begin{center}
$n=10,p=100$
\vspace{0.1cm}

\begin{tabular}{|l|c|c|c|c|c|c|c|}
\hline
$x$&Est.&Std.&Std./Est.&TA  &TW& DMC\\
\hline
 $1.9$ & 1.01e-02 & 6.91e-03 & 0.69 &5.89e-03 &1.34e-02(1.06e-02) &1.00e-02\\
  $2.0$  &  1.71e-03 & 9.95e-04 & 0.58 & 1.07e-03  & 2.20e-03(1.68e-03) &1.74e-03 \\
 $2.1$  &  2.31e-04 & 1.14e-04 & 0.49 & 1.55e-04  & 2.79e-04(2.06e-04) &2.41e-04 \\
  $2.5$  & 1.64e-08 & 5.43e-09 & 0.33 &1.29e-08 &-&-\\
 $3$  &  9.24e-15 & 2.16e-15 & 0.23  & 7.91e-15  &-&-\\
  $4$  &  2.21e-29 & 3.33e-30 & 0.15  & 2.03e-29  &-&-\\
\hline
\end{tabular}

%\label{t10}
\vspace{0.25cm}
$n=10,p=1000$
\vspace{0.1cm}

\begin{tabular}{|l|c|c|c|c|c|c|c|}
\hline
$x$&Est.&Std.&Std./Est.&TA  &TW& DMC\\
\hline
 $1.25$ & 1.16e-02 & 9.08e-03 & 0.78 &6.81e-03 &1.75e-02(1.28e-02) &1.16e-02\\
  $1.28$  &  1.19e-03 & 7.65e-04 & 0.64 & 7.70e-04  & 2.04e-03(1.42e-03) &1.22e-03 \\
 $1.3$  &  2.07e-04 & 1.22e-04 & 0.59 & 1.43e-04  & 3.93e-04(2.66e-04) &2.02e-04 \\
 $1.4$  &  3.88e-09 & 1.50e-09 & 0.39 & 3.11e-09 &-&-\\
  $1.5$  & 3.26e-15 & 9.56e-16 & 0.29 &2.82e-15 &-&-\\
 $2.0$  &  1.88e-59 & 2.59e-60 & 0.14 &1.81e-59  &-&-\\
\hline
\end{tabular}

\vspace{0.25cm}
$n=10,p=10000$
\vspace{0.1cm}

%\end{table}
%
%\begin{table}[h!]
%\caption{Estimates of $P(\lambda_{(1)}>px)$ for $n=10$ and $p=10000$. }
%The standard deviation of the estimate ``Est." is Std./100.}
%\begin{center}
\begin{tabular}{|l|c|c|c|c|c|c|c|}
\hline
$x$&Est.&Std.&Std./Est.&TA  &TW& DMC\\
\hline
 $1.07$ & 4.26e-02 & 4.17e-02 & 0.97   &2.32e-02 & 6.18e-02(4.58e-02) &4.29e-02\\
 $1.08$  &  4.05e-03 & 3.03e-03 & 0.75 & 2.52e-03  &6.98e-03(4.84e-03) &4.02e-03 \\
 $1.09$  &  2.15e-04 & 1.26e-04 & 0.59 & 1.45e-04 & 4.69e-04(3.10e-04) & 2.18e-04\\
  $1.10$  &  6.49e-06 & 3.41e-06 & 0.52 & 4.76e-06 &1.84e-05(1.11e-05)& 5 e-06\\
  $1.15$  &  1.36e-16 & 4.19e-17 & 0.31 & 1.18e-16 &-&-\\
 $1.20$  &  8.15e-32 & 1.81e-32 & 0.22  & 7.49e-32  &-&-\\
\hline
\end{tabular}
 \end{center}
\label{t1}}
\end{table}

\begin{table}[h!]
{\small
\caption{Estimates of $P(\lambda_{(1)}>px)$ for $n=50$ and $p=100, 1000, 10000$. }
%The standard deviation of the estimate ``Est." is Std./100.
\begin{center}
$n=50,p=100$
\vspace{0.1cm}

\begin{tabular}{|l|c|c|c|c|c|c|c|}
\hline
$x$&Est.&Std.&Std./Est.&TA  &TW& DMC\\
\hline
 $3$ &  4.74e-02  & 1.16e-01 & 2.46 & 4.57e-02 &5.07e-02(4.84e-02) &4.77e-02\\
  $3.25$  &  9.15e-04 & 1.44e-03 & 1.58 & 7.22e-04 &8.94e-04(8.36e-04) &9.01e-04 \\
  $3.5$  &  6.02e-06 & 5.56e-06 & 0.92 & 4.47e-06 &2.38e-06(2.03e-06)&-\\
  $4$  &  3.09e-11 & 1.94e-11 & 0.63 & 2.18e-11 &-&-\\
 $5$  &  3.05e-24 & 1.17e-24 & 0.39 &2.27e-24  &-&-\\
\hline
\end{tabular}
% \end{center}
%\label{t30}
%\end{table}
%
%\begin{table}[h!]
%\caption{Estimates of $P(\lambda_{(1)}>px)$ for $n=50$ and $p=1000$. }
%%The standard deviation of the estimate ``Est." is Std./100.
%\begin{center}

\vspace{0.25cm}
$n=50,p=1000$
\vspace{0.1cm}

\begin{tabular}{|l|c|c|c|c|c|c|c|}
\hline
$x$&Est.&Std.&Std./Est.&TA  &TW& DMC\\
\hline
 $1.55$ &  3.51e-03  & 6.57e-03 & 1.87 & 4.43e-03 &4.30e-03(3.60e-03) &3.46e-03\\
  $1.57$  &  5.81e-04 & 8.43e-04 & 1.45 & 6.60e-04 &7.20e-04(5.91e-04) &5.77e-04 \\
  $1.6$  &  2.74e-05 & 3.14e-05 & 1.15 & 2.79e-05 &3.25e-05(2.56e-05)&2.0 e-05\\
  $1.7$  &  8.52e-11 & 6.63e-11 & 0.78 & 7.31e-11 &-&-\\
 $2.0$  &  2.32e-34 & 8.90e-35 & 0.38 &1.91e-34  &-&-\\
\hline
\end{tabular}
% \end{center}
%\label{t3}
%\end{table}
%
%
%\begin{table}[h!]
%\caption{Estimates of $P(\lambda_{(1)}>px)$ for $n=50$ and $p=10000$. }
%%The standard deviation of the estimate ``Est." is Std./100.
%\begin{center}

\vspace{0.25cm}
$n=50,p=10000$
\vspace{0.1cm}

\begin{tabular}{|l|c|c|c|c|c|c|c|}
\hline
$x$&Est.&Std.&Std./Est.&TA  &TW& DMC\\
\hline
 $1.155$ &  1.75e-02& 3.68e-02  & 2.11 & 3.21e-02 &2.25e-02(1.86e-02) &1.77e-02\\
  $1.16$  &  3.99e-03 & 7.37e-03 & 1.84 & 6.09e-03 &5.22e-03(4.20e-03) &3.87e-03 \\
  $1.17$  &  1.06e-04 & 1.50e-04 & 1.41 & 1.33e-04 &1.57e-04(1.20e-04)&1.02e-04\\
 % $x=1.18$  &  1.43e-06 & 1.53e-06 &1.06 & 1.56e-06 &-&-\\
 $1.20$   & 4.22e-11 & 3.62e-11 & 0.86 &3.95e-11  &-&-\\
  $1.25$   &  2.75e-26 & 1.47e-26 & 0.53& 2.36e-26  &-&-\\
\hline
\end{tabular}
 \end{center}
\label{t2}}
\end{table}

To validate our importance sampling results we compute direct Monte Carlo estimates based on $10^6$ independent simulations (column ``DMC''). Note that this validation is not feasible for all probabilities considered. We also present the results from asymptotic approximation methods. The tail probability approximations from Theorem \ref{theorem1} are presented in the column ``TA'' and the approximation results based on the Tracy-Widom distribution are given in the column ``TW''.  The tail probabilities of the Tracy-Widom distribution are calculated using R package {\it ``RMTstat''} \citep*{johnstone2010rmtstat}. In particular,
it is known that when $\beta=1$
$$\frac{\lambda_{(1)}-\mu_{n,p}}{\sigma_{n,p}}$$ converges to the Tracy-Widom law  \citep{johnstone2001,el2003largest}, where
$$\mu_{n,p} = (\sqrt{n}+\sqrt{p-1})^2,\quad \sigma_{n,p}=(\sqrt{n}+\sqrt{p-1})\Big(\frac{1}{\sqrt{n}}+\frac{1}{\sqrt{p-1}}\Big)^{1/3}.$$
A more accurate approximation has been proposed in \cite{johnstone2012fast} and \cite{ma2012accuracy}, where
\begin{eqnarray*}
\mu_{n,p} &=& \Big(\sqrt{n-\frac{1}{2}}+\sqrt{p-\frac{1}{2}}\Big)^2,\\
\sigma_{n,p} &=& \Big(\sqrt{n-\frac{1}{2}}+\sqrt{p-\frac{1}{2}}\Big)\Big(\frac{1}{\sqrt{n-\frac{1}{2}}}+\frac{1}{\sqrt{p-\frac{1}{2}}}\Big)^{1/3}.
\end{eqnarray*}
We report both approximation results in the column ``TW'' with the second in the parentheses.

From Tables \ref{t1} and \ref{t2}, we can see that the proposed importance sampling estimates (``Est'') are consistent with those from direct Monte Carlo simulation (``DMC''). The ratios between the estimated standard deviations of $L_p$ and the estimated tail probabilities (``Std/Est'') stay reasonably small, indicating the efficiency of the algorithm (see equations \eqref{logeff} and  \eqref{strongeff}).
The ratio becomes smaller as $x$ increases. This implies that the algorithm is more efficient for larger $x$'s values.
Moreover, the proposed method provides an efficient way to evaluate the performance of the theoretical approximation methods.
In particular, we can see that the approximations based on the Tracy-Widom distribution (``TW'') overestimate the tail probabilities, especially for larger $x$'s values.
In addition,  larger estimation (relative) errors can be observed for $n=10$ than  for $n=50$.
For the tail approximations (``TA''), we can see they do not give accurate estimates for smaller $x$'s values while the performance gets better as $x$ increases.
Overall, the importance sampling method outperforms results based on the Tracy-Widom distribution  and the tail probability approximations. Lastly it should be noted that the approximations based on Tracy-Widom distribution and the direct Monte Carlo approach are not suitable for estimating the probability of extremely rare events. For these events the only possibility is to use the methods developed in the current work, i.e. importance sampling or the the tail approximation.
R code of the proposed importance sampling algorithm can be found at \url{http://users.stat.umn.edu/~xuxxx360/IS.R}.

%We have also performed the simulation study for $\beta=2$ which give similar conclusions as $\beta=1$ and therefore are omitted due to the space limit.

\begin{table}[h!]{\small
\caption{Estimates of $P(\lambda_{(1)}>px)$ for $n=10$ and $p=100, 1000, 10000$. }
\begin{center}
$n=10,p=100$
\vspace{0.1cm}

\begin{tabular}{|l|c|c|c|c|c|c|c|c|}
\hline
$x$&Est. &TW &  $B_{(5, .5)}$ &  $B_{(10, .5)}$&$B_{(20, .5)}$& $t_{50}$ &  $t_{100}$  \\
\hline
 $1.9$ & 1.01e-2  &1.34(1.06)e-2  & 0.57e-2 & 0.81e-2 & 0.88e-2  & 1.15e-2 & 1.08e-2 \\
  $2.0$  &  1.71e-3 & 2.20(1.68)e-3 & 0.72e-3 & 1.27e-3 & 1.45e-3  & 1.93e-3 & 1.76e-3\\
\hline
\end{tabular}

%\label{t10}
\vspace{0.25cm}
$n=10,p=1000$
\vspace{0.1cm}

\begin{tabular}{|l|c|c|c|c|c|c|c|}
\hline
$x$&Est. &TW &  $B_{(5, .5)}$ &  $B_{(10, .5)}$&$B_{(20, .5)}$& $t_{50}$ &  $t_{100}$  \\
\hline
 $1.25$ & 1.16e-2 &1.75(1.28)e-2  & 0.72e-2& 1.00e-2 & 1.10e-2  & 1.36e-2 & 1.22e-2\\
  $1.28$  &  1.19e-3  & 2.04(1.42)e-3 & 0.55e-3 & 0.94e-3 & 1.13e-3  &1.61e-3 &1.20e-3\\
\hline
\end{tabular}

\vspace{0.25cm}
$n=10,p=10000$
\vspace{0.1cm}

\begin{tabular}{|l|c|c|c|c|c|c|c|c|}
\hline
$x$&Est. &TW &  $B_{(5, .5)}$ &  $B_{(10, .5)}$&$B_{(20, .5)}$& $t_{50}$ &  $t_{100}$  \\
\hline
 $1.07$ & 4.26e-2  & 6.18(4.58)e-2 & 3.04e-2& 3.65e-2 &3.96e-2 &4.85e-2 & 4.49e-2\\
 $1.08$  &  4.05e-3&6.98(4.84)e-3 & 2.16e-3& 3.20e-3 & 3.61e-3 & 4.77e-2 &4.25e-3\\
\hline
\end{tabular}
 \end{center}
\label{t3}}
\end{table}

\subsection{Non-Gaussian Matrices}
We next investigate the behavior of the algorithm and approximations in the non-Gaussian setting.
In particular, we generate the matrix $\XX^*\XX$ where $\XX=(x_{ij})_{n\times p}$ with $n=10,$ $p=100, 1000$ and $10000$, and $x_{ij}$ are i.i.d.  random variables following a standardized Binomial or t-distribution with mean zero and variance one.
Table \ref{t3} presents the simulation results, where  columns ``Est'' is  the importance sampling estimates under the Gaussian assumption as in Table \ref{t1}, ``TW'' is the Tracy-Widom estimates, and the last five columns are direct Monte Carlo results under different standardized distributions with  $10^5$ replications.
Table \ref{t3} shows that the importance sampling estimators are generally  comparable to Tracy-Widom  estimators.
Moreover, as the distribution of $x_{1,1}$ becomes more like the normal distribution
(such as when the number of trials of a Binomial distribution increases or the degrees of freedom of a t-distribution increases),
the importance sampling estimators become more accurate and outperform Tracy-Widom estimators.

\cite{bordenave2014large} studied the large deviations properties of the spectrum of Wigner matrices whose entries were random variables with density proportional to $e^{-|x|^{\alpha}}$ for all $x\in \mathbb{R}$ with parameter $\alpha\in (0,2)$. Interestingly they observed that both the speed and the rate function of the large deviations principle depends on the parameter $\alpha$.
Although the Wigner matrices and the Wishart matrices belong to different ensembles, their large deviation principles are of similar structures; see, e.g., Anderson et al. (2010).  Therefore we do not expect a big universality family for our theoretical approximation results. However, as illustrated  in the above simulation, the importance sampling estimator based on the Gaussian assumption will provide an adequate approximation for many cases.

%C. Bordenave and P. Caputo (2014). A large deviation principle for Wigner matrices without Gaussian tails. The Annals of Probability 42(6), 2454?2496.

\begin{figure}[h]
\includegraphics[width=12cm]
{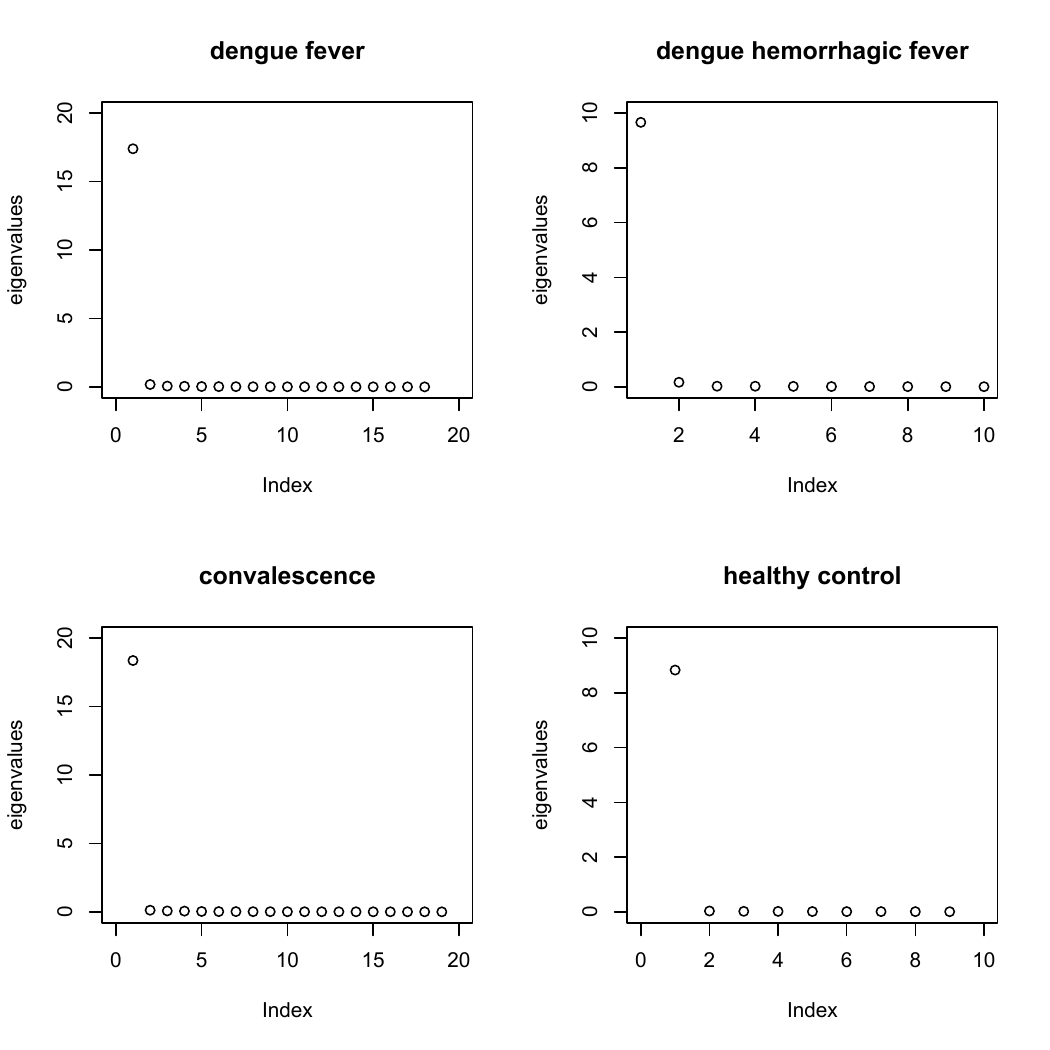}
\caption{Plot of eigenvalues}
\label{figure1}
\end{figure}

\subsection{Dengue Virus Example}
To illustrate the use of our algorithm we consider a real data set of immunity to the Dengue virus (DENV). 
The data set contains the innate immune response to DENV infection in whole blood samples of acutely infected humans in Bangkok, Thailand during the season of 2009 \citep{kwissa2014dengue}.
The data set can be downloaded from
\url{http://www.ncbi.nlm.nih.gov/geo/query/acc.cgi?acc=GSE51808}.
Whole blood samples were analyzed from 18 dengue fever patients and 10 dengue hemorrhagic fever patients hospitalized at the Siriraj Hospital in Bangkok, Thailand, and the samples were obtained between days 2 and 9 after onset of symptoms.
Blood samples from 19 convalescence patients were also obtained  at 4 weeks or later after discharge. In addition, there is a  control group of 9 healthy donors, and their blood was also sampled.

We consider data from  four groups: 18 dengue fever patients; 10 dengue hemorrhagic fever patients; 19 convalescence patients; and 9 healthy control patients. Each individual has $p=54715$ covariates and Robust Multi-array Average (RMA) normalization was performed using Expression Console software \citep{kwissa2014dengue}.
We further standardize each group by their global mean and standard deviation.
For each group we consider the null hypothesis that their covariance matrix is a white Wishart matrix.
We calculate the largest eigenvalues from the four groups' covariance matrices  as our test statistics.
The eigenvalues of the matrix $\XX^*\XX/p$ of each group  is given in the Figure \ref{figure1}.
We compute the corresponding $p$-values of the four test statistics to be  $\ll 10^{-10}$ and therefore we reject all four null hypotheses.

\section{Proof of Theorems}\label{proofmain}
In this section we present the proofs of main theorems. Technical  lemmas and their proofs are provided in the Supplementary Material.
\paragraph{Outline of Proofs}
In our  analysis of the tail probability $P(\lambda_{(1)}>px)$ as well as  the variance of the Monte Carlo estimator,  we frequently use the following factorization of the joint probability density function (pdf) in \eqref{orderdensity}:
$$
g_{n,p,\beta}(\lambda_1,\ldots,\lambda_n)=nA_n\prod_{i=2}^n(\lambda_1-\lambda_i)\lambda_1^{\frac{\beta(p-n+1)}{2}-1}e^{-\lambda_1/2}g_{n-1,p-1,\beta}(\lambda_2,\ldots,\lambda_n),
$$ where $A_n$ is defined as in \eqref{An}. Some key lemmas are developed to show that in the case $p/n\to\infty$ we can find a positive sequence $\delta_n$ converging to zero (in certain rate) such that we can focus on the event
\begin{equation}\label{key1}
\{p(x+\delta_n)>\lambda_{(1)}>px, \lambda_{(2)}<p(\beta+\delta_n),\lambda_{(n)}>p(\beta-\delta_n)\}
\end{equation}
instead of $\{\lambda_{(1)}>px\}$. Roughly speaking we can think of this as the following, if $\lambda_{(1)}>px$, then (i) $\lambda_{(1)}\approx px$ and (ii) the remaining eigenvalues are approximately $p\beta$, see Lemmas \ref{lemma20} and \ref{lemma0} for further details.
With the result in \eqref{key1}, we can approximate $\prod_{i=2}^n(\lambda_1-\lambda_i)$  by
$$
\psi_{n,p,\beta}(\lambda_1,\ldots,\lambda_n)=(px-p\beta)^{\beta(n-1)}e^{(n+o(n))\frac{\lambda_1-px}{px-p\beta} -\sum_{i=2}^n\frac{\lambda_i-p\beta}{px-p\beta}-
\alpha\sum_{i=2}^n\left(\frac{\lambda_i-p\beta}{px-p\beta}\right)^2},
$$ with $\alpha$ being approximately equal to $1/2$. Consider then approximating $g_{n,p,\beta}$ with
$$
nA_n\psi_{n,p,\beta}(\lambda_1,\ldots,\lambda_n)\lambda_1^{\frac{\beta(p-n+1)}{2}-1}e^{-\lambda_1/2}g_{n-1,p-1,\beta}(\lambda_2,\ldots,\lambda_n).
$$
The benefit of working with the approximation in the previous display is that when integrating we can factor our integrand into the product of terms involving $\lambda_1$ and terms involving $\lambda_i$ for $i\in\{2,\ldots,n-1\}$. The term involving $\lambda_1$ are quite simple, basically boiling down to the pdf of a gamma distributed random variable. The term with $\lambda_i$ for $i>1$ requires the approximation of
\begin{align}
\label{eq:IntroExpectedValue}
E\exp\left[-\sum_{i=2}^n\frac{\lambda_i-p\beta}{px-p\beta}-
\alpha\sum_{i=2}^n\left(\frac{\lambda_i-p\beta}{px-p\beta}\right)^2\right].
\end{align}
We achieve this through using a matrix representation of the $\beta$-Laguerre ensemble by \cite{dumitriu2002matrix} to express
$\sum_{i=2}^n\lambda_i$ and $\sum_{i=2}^n\left(\lambda_i-p\beta\right)^2$ as sums of independent random variables.
 We are then able to develop approximations to the expected value in \eqref{eq:IntroExpectedValue}, see Lemma \ref{lemma1'} in the Supplementary Material.

\medskip
\begin{proof}[Proof of Theorem \ref{theorem1}]
We first focus on the case when  $p/n^{5/3}\to\infty$.
Set $a_n=\sqrt{np^{-1}}+p^{-1}$ and $b_n=pn^{-2}$. Choose $\delta_n=\min\{\sqrt{a_n}, \sqrt{a_nb_n}\,\}$. From the assumption $p/n^{5/3}\to\infty$ it is trivial to check  that
\begin{eqnarray}\label{eat_medcine}
\delta_n\to 0,\ \ p\delta_n\to\infty,\ \ \frac{\delta_n^2p}{n}\to\infty, \mbox{ and } \frac{\delta_nn^2}{p}\to 0.
\end{eqnarray}
In the discussion below, whenever we need a restriction about $\delta_n$ we can always get it from the above limits.

To prove \eqref{gxaaaa}, we first show that
\begin{eqnarray}
&& P(\lambda_{(1)}>px)
\gtrsim
 \frac{2x}{x-\beta}nA_n (px-p\beta)^{\beta(n-1)} (px)^{\frac{\beta(p-n+1)}{2}-1}
e^{-\frac{px}{2}-\frac{\beta^3n^{2}}{2(x-\beta)^2p}};\label{couter_clock}\\
& & P(\lambda_{(1)}>px)
\lesssim
 \frac{2x}{x-\beta}nA_n (px-p\beta)^{\beta(n-1)} (px)^{\frac{\beta(p-n+1)}{2}-1}
e^{-\frac{px}{2}-\frac{\beta^3n^{2}}{2(x-\beta)^2p}}\label{nice_sky}.
\end{eqnarray}
We prove them separately.

\medskip
{\it The proof of (\ref{couter_clock})}. By Lemmas \ref{lemma20} and \ref{lemma0} in the Supplementary Material,
\begin{eqnarray*}
P(\lambda_{(1)}> p(x+\delta_n))&=&o(1)P(\lambda_{(1)}> px),\\
P(\lambda_{(1)}>px, \lambda_{(2)}> p(\beta+\delta_n)) &=&o(1)P(\lambda_{(1)}>px),\\
P(\lambda_{(1)}>px, \lambda_{(n)}< p(\beta-\delta_n)) &=&o(1)P(\lambda_{(1)}>px).
\end{eqnarray*}
Therefore, $P(\lambda_{(1)}>px)$ is asymptotically equivalent to the probability of $\{px<\lambda_{(1)}<p(x+\delta_n),~ \lambda_{(2)}< p(\beta+\delta_n),$ and $\lambda_{(n)}> p(\beta-\delta_n)\}$. That is,
\begin{eqnarray*}
 && P(\lambda_{(1)}>px)\\
&\sim & \int_{\lambda_1>\cdots>\lambda_n,~ px<\lambda_{1}<p(x+\delta_n),\atop
~ \lambda_{2}< p(\beta+\delta_n), \lambda_{n}> p(\beta-\delta_n)}
n! f_{n,p,\beta}(\lambda_1,\cdots,\lambda_n) d\lambda_1\cdots d\lambda_n \\
&=& \int_{\lambda_1>\cdots>\lambda_n,~ px<\lambda_{1}<p(x+\delta_n),\atop
~ \lambda_{2}< p(\beta+\delta_n), \lambda_{n}> p(\beta-\delta_n)}
 n A_n \prod_{i=2}^n(\lambda_1-\lambda_i)^\beta
\cdot \lambda_1^{\frac{\beta(p-n+1)}{2}-1}\cdot
e^{-\frac{1}{2} \lambda_1}\\
&&\quad\quad\quad
 \times g_{n-1,p-1,\beta}(\lambda_2,\cdots,\lambda_n)d\lambda_1\cdots d\lambda_n \\
&=& \int_{\lambda_1>\cdots>\lambda_n,~ px<\lambda_{1}<p(x+\delta_n),\atop~ \lambda_{2}< p(\beta+\delta_n), \lambda_{n}> p(\beta-\delta_n)}
 n A_n  (px-p\beta)^{\beta(n-1)}
\prod_{i=2}^n\left(1+\frac{\lambda_1-px}{px-p\beta}-\frac{\lambda_i-p\beta}{px-p\beta}\right)^\beta\\
&&\quad\quad\quad \times
 \lambda_1^{\frac{\beta(p-n+1)}{2}-1}
e^{-\frac{1}{2} \lambda_1}
g_{n-1,p-1,\beta}(\lambda_2,\cdots,\lambda_n)d\lambda_1\cdots d\lambda_n.
\end{eqnarray*}

Take $z$ and $a_2$ in Lemma \ref{lemmabound} such that $$z=\frac{\lambda_1-px}{px-p\beta}-\frac{\lambda_i-p\beta}{px-p\beta}
 \mbox{ and }
 \alpha_2=\frac{1}{2}-\frac{\delta_n}{x-\beta},$$
 then $\alpha_2<\frac{1}{2}$ as $n$ is sufficiently large  (we will have similar situations in the rest of the paper, the same interpretation ``as $n$ is sufficiently large" applies unless otherwise specified). Consequently,
 for $px<\lambda_{1}<p(x+\delta_n)$ and
  $p(\beta-\delta_n)< \lambda_{n}\leq\ldots\leq \lambda_{2}< p(\beta+\delta_n)$
\begin{eqnarray}\label{gx6}
 &&\prod_{i=2}^n\Big(1+\frac{\lambda_1-px}{px-p\beta}-\frac{\lambda_i-p\beta}{px-p\beta}\Big)\notag\\
&=& e^{\sum_{i=2}^n\log\left(1+\frac{\lambda_1-px}{px-p\beta}-\frac{\lambda_i-p\beta}{px-p\beta}\right)}\notag\\
&\leq&
e^{\sum_{i=2}^n\left(\frac{\lambda_1-px}{px-p\beta}-\frac{\lambda_i-p\beta}{px-p\beta}\right)
-\alpha_2\sum_{i=2}^n\left(\frac{\lambda_1-px}{px-p\beta}-\frac{\lambda_i-p\beta}{px-p\beta}\right)^2}
\notag\\
&\leq&
e^{(n+o(n))\frac{\lambda_1-px}{px-p\beta} -\sum_{i=2}^n\frac{\lambda_i-p\beta}{px-p\beta}-
\alpha_2\sum_{i=2}^n\left(\frac{\lambda_i-p\beta}{px-p\beta}\right)^2},
\end{eqnarray}
where in the last step we used $e^{-(n-1)\alpha_2\left(\frac{\lambda_1-px}{px-p\beta}\right)^2}\leq 1$ and
$\alpha_2\frac{\lambda_1-px}{px-p\beta}\sum_{i=2}^n\frac{\lambda_i-p\beta}{px-p\beta}=o(n)\frac{\lambda_1-px}{px-p\beta}$ since $\frac{\lambda_i-p\beta}{px-p\beta}\leq O(\delta_n)=o(1)$ uniformly for all $2\leq i \leq n$.
Then we have the following upper bound:
\begin{eqnarray}
 && P(\lambda_{(1)}>px) \nonumber\\
&\lesssim &
nA_n (px-p\beta)^{\beta(n-1)}
\int_{px}^{p(x+\delta_n)}
\lambda_1^{\frac{\beta (p-n+1)}{2}-1}
e^{(\beta n+o(n))\frac{\lambda_1-px}{px-p\beta} -\frac{\lambda_1}{2}}d\lambda_1\nonumber \\
&&\times
\int_{\lambda_2>\cdots>\lambda_n\atop  \lambda_{2}< p(\beta+\delta_n), \lambda_{n}> p(\beta-\delta_n)}
e^{-\beta\sum_{i=2}^n\frac{\lambda_i-p\beta}{p(x-\beta)}
-\beta\alpha_2\sum_{i=2}^n\left(\frac{\lambda_i-p\beta}{px-p\beta}\right)^2}\notag \\
&&\quad\times g_{n-1,p-1,\beta}(\lambda_2,\cdots,\lambda_n)d\lambda_2\cdots d\lambda_n \nonumber\\
&\leq&
nA_n (px-p\beta)^{\beta(n-1)}
\int_{px}^{p(x+\delta_n)}
\lambda_1^{\frac{\beta (p-n+1)}{2}-1}
e^{(\beta n+o(n))\frac{\lambda_1-px}{px-p\beta} -\frac{\lambda_1}{2}}d\lambda_1 \nonumber\\
&&\times E\Big[e^{-\beta\sum_{i=2}^n\frac{(\lambda_i-\beta p)}{p(x-\beta)}-\beta\alpha_2\sum_{i=2}^n\frac{(\lambda_i-\beta p)^2}{p^2(x-\beta)^2}}\Big]. \label{left_sun}
\end{eqnarray}
Trivially, $\frac{n^2}{p}-\frac{(n-1)^2}{p-1}=o(1)$ since $n/p\to 0$. We then have from Lemma \ref{lemma1'} that
\begin{eqnarray}\label{siren_dark}
E\Big[e^{-\beta\sum_{i=2}^n\frac{(\lambda_i-\beta p)}{p(x-\beta)}-\beta\alpha_2\sum_{i=2}^n\frac{(\lambda_i-\beta p)^2}{p^2(x-\beta)^2}}\Big]=
 e^{-\frac{\alpha_2\beta^3}{(x-\beta)^2}[\frac{n^2}{p}-\frac{O(1)n^{3}}{3p^2}]}.
\end{eqnarray}
 This implies that:
\begin{eqnarray}\label{upbound}
&& P(\lambda_{(1)}>px)\notag\\
&\lesssim& nA_n (px-p\beta)^{\beta(n-1)}
e^{-\frac{\alpha_2\beta^3}{(x-\beta)^2}[\frac{n^2}{p}-\frac{O(1)n^{3}}{3p^2}]}
\notag\\
&&\times
\int_{px}^{p(x+\delta_n)}
\lambda_1^{\frac{\beta (p-n+1)}{2}-1}
e^{(\beta n+o(n))\frac{\lambda_1-px}{px-p\beta}-\frac{\lambda_1}{2}}d\lambda_1\notag\\
&=& nA_n (px-p\beta)^{\beta(n-1)}
e^{-\frac{\alpha_2\beta^3}{(x-\beta)^2}[\frac{n^2}{p}-\frac{O(1)n^{3}}{3p^2}]}
\notag\\
&&\times
\int_{0}^{p\delta_n}
(\lambda_1+px)^{ \frac{\beta(p-n+1)}{2}-1}
e^{(\beta n+o(n))\frac{\lambda_1}{px-p\beta}-\frac{\lambda_1+px}{2}}d\lambda_1\notag\\
& \leq &
nA_n (px-p\beta)^{\beta(n-1)}
e^{-\frac{px}{2}-\frac{\alpha_2\beta^3}{(x-\beta)^2}[\frac{n^2}{p}-\frac{O(1)n^{3}}{3p^2}]}
 (px)^{\frac{\beta(p-n+1)}{2}-1}\notag\\
&&\times \int_{0}^{p\delta_n}
e^{ \{\frac{\beta(p-n+1)}{2}-1\}\frac{\lambda_1}{px}+(\beta n+o(n))\frac{\lambda_1}{px-p\beta}  -\frac{\lambda_1}{2}}
d\lambda_1 \notag\\
&\sim& \frac{2x}{x-\beta}nA_n (px-p\beta)^{\beta(n-1)} (px)^{\frac{\beta(p-n+1)}{2}-1}
e^{-\frac{px}{2}-\frac{\alpha_2\beta^3}{(x-\beta)^2}[\frac{n^2}{p}-\frac{O(1)n^{3}}{3p^2}]},
\end{eqnarray}
where in the second step we changed the variable $\lambda_1$ to $\lambda_1+px$; in the third step we used $(\lambda_1+px)\leq(px)\exp\{\lambda_1/(px)\}$; the last step follows from the fact that   $\{\frac{\beta(p-n+1)}{2}-1\}\frac{\lambda_1}{px}+
(\beta n+o(n))\frac{\lambda_1}{px-p\beta} -\frac{\lambda_1}{2}\sim \frac{\beta-x}{2x}\lambda_1$ by using the fatcs $p\delta_n\to\infty$ and $n/p\to 0$.

Finally, noticing that $\frac{\alpha_2\beta^3}{(x-\beta)^2}[\frac{n^2}{p}-\frac{O(1)n^{3}}{3p^2}]-\frac{\beta^3n^{2}}{2(x-\beta)^2p}\to 0$ due to the fact $\frac{\delta_n n^2}{p}\to 0$ and $\frac{n^{5/3}}{p}\to 0$, we obtain (\ref{couter_clock}).

\medskip
{\it The proof of (\ref{nice_sky})}. By the same argument as in the above derivation,
  take $\alpha_1=1/2+\delta_n/(x-\beta)$ in Lemma \ref{lemmabound} to have
\begin{eqnarray*}
\sum_{i=2}^n\log\left(1+\frac{\lambda_1-px}{px-p\beta}-\frac{\lambda_i-p\beta}{px-p\beta}\right)
&\geq&\sum_{i=2}^n \log\left(1-\frac{\lambda_i-p\beta}{px-p\beta}\right)\\
&\geq&
-\sum_{i=2}^n\frac{\lambda_i-p\beta}{px-p\beta}
-\alpha_1\sum_{i=2}^n\frac{(\lambda_i-p\beta)^2}{(px-p\beta)^2}
\end{eqnarray*}
under the restriction $px<\lambda_{1}<p(x+\delta_n),~ \lambda_{2}< p(\beta+\delta_n),$ and $\lambda_{n}> p(\beta-\delta_n)$. Therefore,
\begin{eqnarray}\label{lowerbound_0}
&& P(\lambda_{(1)}>px)\notag\\
&\geq  & \int_{\lambda_1>\cdots>\lambda_n,~ \lambda_{1}>px,\atop \lambda_{2}< p(\beta+\delta_n), \lambda_{n}> p(\beta-\delta_n)} n! f_{n,p,\beta}(\lambda_1,\cdots,\lambda_n) d\lambda_1\cdots d\lambda_n \notag\\
&\geq & nA_n (px-p\beta)^{\beta(n-1)}\notag\\
&&\times
\int_{\lambda_2>\cdots>\lambda_n\atop  \lambda_{2}< p(\beta+\delta_n), \lambda_{n}> p(\beta-\delta_n)}\int_{px}^{p(x+\delta_n)}
 e^{-\beta\sum_{i=2}^n\frac{\lambda_i-p\beta}{p(x-\beta)}
-\beta\alpha_1\sum_{i=2}^n\frac{(\lambda_i-p\beta)^2}{p^2(x-\beta)^2}}\notag\\
&&\quad\quad\times
\lambda_1^{\frac{\beta (p-n+1)}{2}-1}e^{-\frac{\lambda_1}{2}}
\times g_{n-1,p-1,\beta}(\lambda_2,\cdots,\lambda_n)d\lambda_1d\lambda_2\cdots d\lambda_n\notag\\
&= & nA_n (px-p\beta)^{\beta(n-1)}\times \int_{px}^{p(x+\delta_n)}
\lambda_1^{\frac{\beta (p-n+1)}{2}-1}e^{-\frac{\lambda_1}{2}}d\lambda_1\notag\\
&&\times \int_{\lambda_2>\cdots>\lambda_n,\atop  \lambda_{2}< p(\beta+\delta_n), \lambda_{n}> p(\beta-\delta_n)}e^{-\beta\sum_{i=2}^n\frac{\lambda_i-p\beta}{p(x-\beta)}
-\beta\alpha_1\sum_{i=2}^n\frac{(\lambda_i-p\beta)^2}{p^2(x-\beta)^2}}
\notag\\ &&\quad\quad\times
g_{n-1,p-1,\beta}(\lambda_2,\cdots,\lambda_n)d\lambda_2\cdots d\lambda_n.
\end{eqnarray}
By Lemmas \ref{lemma1'} and  \ref{lemma4}, we have the second integral in \eqref{lowerbound_0} is $ e^{-\frac{\alpha_1\beta^3}{(x-\beta)^2}[\frac{n^2}{p}-\frac{O(1)n^{3}}{3p^2}]}$ since $\frac{\delta_n n^2}{p}\to 0$ and  $p\delta_n\to\infty$.
It follows that
\begin{eqnarray}
 && P(\lambda_{(1)}>px)\notag\\
&&\gtrsim
nA_n (px-p\beta)^{\beta(n-1)}
e^{-\frac{\alpha_1\beta^3}{(x-\beta)^2}[\frac{n^2}{p}-\frac{O(1)n^{3}}{3p^2}]}
\int_{px}^{p(x+\delta_n)}
\lambda_1^{ \frac{\beta(p-n+1)}{2}-1}e^{-\frac{\lambda_1}{2}}d\lambda_1
\notag\\
&&\sim
 \frac{2x}{x-\beta}nA_n (px-p\beta)^{\beta(n-1)} (px)^{\frac{\beta(p-n+1)}{2}-1}
e^{-\frac{px}{2}-\frac{\alpha_1\beta^3}{(x-\beta)^2}[\frac{n^2}{p}-\frac{O(1)n^{3}}{3p^2}]},
\label{lowerbound}
\end{eqnarray}
where the last step follows the same argument as in \eqref{upbound} due to the fact $p\delta_n\to\infty$. This yields (\ref{nice_sky}) by noticing that $\frac{\alpha_1\beta^3}{(x-\beta)^2}[\frac{n^2}{p}-\frac{O(1)n^{3}}{3p^2}]-\frac{\beta^3n^{2}}{2(x-\beta)^2p}\to 0$.

\bigskip
Next we prove the result \eqref{gxaaaa}. By the above derivations,
\begin{eqnarray}\label{appoxeq}
&& \log P(\lambda_{(1)} > px)\notag\\
&&=~ \log n+\log A_n +\log \frac{2x}{x-\beta}+{\beta(n-1)}\log(px-p\beta)\notag\\
&&~ +\left[\frac{\beta (p-n+1)}{2}-1\right]\log(px)-\frac{px}{2}-\frac{\beta^3n^{2}}{2(x-\beta)^2p}+o(1).
\end{eqnarray}
From the Stirling formula 
$$\log\Gamma(z)=z\log z-z-(\log z)/2+\log \sqrt{2\pi}+o(1)$$ for large $|z|$; we know
\begin{align}\label{logA}
&~ \log A_n\notag\\
=&~
\log\frac{2^{-\frac{\beta(n+p-1)}{2}}\Gamma(1+\frac{\beta}{2})}{\Gamma(1+\frac{\beta n}{2})\Gamma(\frac{\beta p}{2})}
=\log\frac{2^{-\frac{\beta(n+p-1)}{2}}\Gamma(1+\frac{\beta}{2})}{\frac{\beta n}{2}\Gamma(\frac{\beta n}{2})\Gamma(\frac{\beta p}{2})}\notag\\
=&~
-\frac{\beta}{2}(n+p-1)\log 2+\log\Gamma\big(1+\frac{\beta}{2}\big)-\frac{\beta n}{2}\log\frac{\beta n}{2} -\frac{\beta p}{2}\log\frac{\beta p}{2} \notag\\
&~+\frac{\beta n}{2}+\frac{\beta p}{2}+\frac{1}{2}\log \frac{\beta p}{2}-\frac{1}{2}\log \frac{\beta n}{2}
-2\log \sqrt{2\pi}
+o(1)\notag\\
=&~ -\frac{\beta p}{2}\log p -\frac{\beta p}{2}\left(\log {\beta}-1\right)
-\frac{\beta n}{2}\log n-\frac{\beta n}{2}\left(\log{\beta}-1\right)\notag\\
&~+\frac{1}{2}\log \frac{\beta p}{2}-\frac{1}{2}\log \frac{\beta n}{2}
+{\frac{\beta}{2}\log 2} -2\log \sqrt{2\pi}+\log\Gamma\big(1+\frac{\beta}{2}\big)+o(1).\notag\\&
\end{align}
Therefore, plugging in the above expansion of $\log A_n$  into equation \eqref{appoxeq}, we obtain
\begin{eqnarray*}
&&\log P(\lambda_{(1)} > px) \\
&=&\log n -\frac{\beta p}{2}\log p -\frac{\beta p}{2}\left(\log{\beta}-1\right)
-\frac{\beta n}{2}\log n-\frac{\beta n}{2}\left(\log{\beta}-1\right)\\
&&+\frac{1}{2}\log \frac{\beta p}{2}-\frac{1}{2}\log \frac{\beta n}{2}{+\frac{\beta}{2}\log 2}-2\log \sqrt{2\pi}+\log\Gamma\big(1+\frac{\beta}{2}\big)\\
 &&+\log \frac{2x}{x-\beta}+{\beta(n-1)}\log(px-p\beta) \\
&&+\left[\frac{\beta (p-n+1)}{2}-1\right]\log(px)-\frac{px}{2}-\frac{\beta^3n^{2}}{2(x-\beta)^2p}+o(1)\\
&=&  p\left(\frac{\beta}{2}-\frac{\beta}{2}\log\beta-\frac{x}{2}+\frac{\beta}{2}\log x\right)+\frac{\beta n}{2}\log \frac{p}{n} -\frac{\beta+1}{2}\log p\\
&&+\beta n\left(-\frac{\log x}{2}+\log(x-\beta)-\frac{\log\beta}{2}+\frac{1}{2}\right)+\frac{1}{2}\log n-\frac{\beta^3n^{2}}{2(x-\beta)^2p}\\
&&-(\beta+1)\log(x-\beta)+\frac{\beta}{2}\log ({2}x)-\log (\pi)+\log\Gamma\big(1+\frac{\beta}{2}\big)+o(1).
\end{eqnarray*}
This completes the proof of \eqref{gxaaaa}.
%\end{proof}

\medskip
Next we prove the result \eqref{gxgx2} when $p/n\to\infty$.
For any $l_n>0$ such that $l_n\to\infty$ and $\sqrt{np^{-1}}l_n\to 0$, take $\delta_n=\sqrt{np^{-1}}l_n$. Then $\delta_n\to 0$ and
$\frac{\delta_n^2}{np^{-1}}\to\infty.$ Reviewing the proof of  the asymptotic upper bound \eqref{upbound} and the lower bound \eqref{lowerbound}, we only use the three conditions: $\delta_n\to 0$,
$\frac{\delta_n^2}{np^{-1}}\to\infty$ and $p/n\to\infty$.
Consequently,
\begin{eqnarray*}
&& B_{n,p,\beta}(x)+ \big(\frac{1}{2}-\alpha_1\big)\frac{\beta^3n^{2}}{(x-\beta)^2p}
 +\alpha_1\frac{\beta^3O(1)n^{3}}{3(x-\beta)^2p^2}\\
&\leq& \log P(\lambda_{(1)} > px)\\
&\leq& B_{n,p,\beta}(x)+\big(\frac{1}{2}-\alpha_2\big)\frac{\beta^3n^{2}}{(x-\beta)^2p}
+\alpha_2\frac{\beta^3O(1)n^{3}}{3(x-\beta)^2p^2},
\end{eqnarray*}
where $B_{n,p,\beta}(x)$ is defined as in Theorem \ref{theorem1}, $\alpha_1=\frac{1}{2}+\frac{\delta_n}{x-\beta}$ and $\alpha_2=\frac{1}{2}-\frac{\delta_n}{x-\beta}$.
Replace $\delta_n$ with $\sqrt{np^{-1}}l_n$ and we have
$$B_{n,p,\beta}(x)- \Theta(1)\frac{l_n n^{5/2}}{p^{3/2}}
\leq \log P(\lambda_{(1)} > px)
\leq B_{n,p,\beta}(x)+ \Theta(1)\frac{l_n n^{5/2}}{p^{3/2}},$$
 or equivalently,
\begin{eqnarray*}
\big|\log P(\lambda_{(1)} > px) - B_{n,p,\beta}(x)\big|\cdot \frac{p^{3/2}}{n^{5/2}} \leq \Theta(l_n)
\end{eqnarray*}
for any $l_n$ satisfying $l_n\to\infty$ and  $l_n=o(\frac{p}{n})$. Observe that the left hand side of the above does not depend on $l_n$, we conclude
\begin{eqnarray*}
\limsup_{n\to\infty}\big|\log P(\lambda_{(1)} > px) - B_{n,p,\beta}(x)\big|\cdot \frac{p^{3/2}}{n^{5/2}} < \infty
\end{eqnarray*}
by using a trivial argument of contradiction.
This completes the proof.
\end{proof}

\bigskip

\begin{proof}[Proof of Theorem \ref{theorem2}] 

First consider (i). Since $p/n^{5/3}\to\infty$, we are able to pick $\delta_n>0$ satisfying $\delta_n\to 0$,
${\delta_n^2}{n^{-1}p}\to\infty$ and $\delta_nn^2/p \to 0$. To show $E^Q\left[L_p^2\right]=E^Q\left[L_p^2;\, \lambda_{(1)}>px\right]\sim P(\lambda_{(1)}>px)^2$, by Lemma \ref{lemma3} in the Supplementary Material, it suffices to show that
\begin{eqnarray}\label{Liuhe_county}
&& E^Q\left[L_p^2;\, \lambda_{(1)}>px,p(\beta+\delta_n)>\lambda_{(2)}\cdots>\lambda_{(n)}>p(\beta-\delta_n) \right]\\
&\sim&
 P(\lambda_{(1)}>px)^2.\notag
\end{eqnarray}
Following (\ref{aa1}) and \eqref{Lp},
\begin{align}
 & \mbox{LHS of } (\ref{Liuhe_county})\notag \\
 = &~ E^Q\Big[\Big(\frac{nA_n\times
\prod_{i=2}^n(\lambda_1-\lambda_i)^{\beta}
\cdot \lambda_1^{\frac{\beta(p-n+1)}{2}-1}\cdot
e^{-\frac{1}{2} \lambda_1}}
{\frac{x-\beta}{2x} e^{-\frac{x-\beta}{2x} (\lambda_1-px)}\cdot I_{(\lambda_1>px)}}
\Big)^2;  \nonumber\\
& \ \ \ \ \ \ \ \ \ \ \ \ \ \ \ \ \ \ \ \ \ \ \ \ \ \ \ \
 \lambda_1>px,p(\beta+\delta_n)>\lambda_2\cdots>\lambda_n>p(\beta-\delta_n)\Big]\notag\\
=&~ 2x(x-\beta)^{-1}n^2 A_n^2 e^{-\frac{x-\beta}{2x} px}\notag\\
&~ \times \int_{\lambda_1>px,\atop p(\beta+\delta_n)>\lambda_2\cdots>\lambda_n>p(\beta-\delta_n)}
\prod_{i=2}^n(\lambda_1-\lambda_i)^{2\beta}
\cdot \lambda_1^{\beta(p-n+1)-2}\cdot
e^{-(1-\frac{x-\beta}{2x})\lambda_1} d\lambda_1\notag\\
&\quad\quad\times g_{n-1,p-1,\beta}(\lambda_2,\cdots,\lambda_n)d\lambda_2\cdots d\lambda_n\notag\\
=&~
 2x(x-\beta)^{-1}n^2 A_n^2 e^{-\frac{x-\beta}{2x} px} (px-p\beta)^{2\beta(n-1)}\notag\\
&~ \times \int_{\lambda_2>\cdots>\lambda_n\atop  \lambda_{2}< p(\beta+\delta_n), \lambda_{n}> p(\beta-\delta_n)} \int_{px}^\infty
 \lambda_1^{\beta(p-n+1)-2}\cdot
e^{-(1-\frac{x-\beta}{2x})\lambda_1} d\lambda_1\notag\\
&\quad\quad\times \prod_{i=2}^n\left(1+\frac{\lambda_1-px}{px-p\beta}-\frac{\lambda_i-p\beta}{px-p\beta}\right)^{2\beta}\cdot g_{n-1,p-1,\beta}(\lambda_2,\cdots,\lambda_n)d\lambda_2\cdots d\lambda_n.\notag\\&
\label{eq:sq}
\end{align}
Using the upper bound as in \eqref{gx6} and part of the arguments in (\ref{left_sun}) and (\ref{siren_dark}), we have
\begin{align}
& \mbox{Display } \eqref{eq:sq}\notag\\
\lesssim& ~
 2x(x-\beta)^{-1}n^2 A_n^2 e^{-\frac{x-\beta}{2x} px} (px-p\beta)^{2\beta(n-1)}\notag\\
&~ \times
\int_{px}^\infty
e^{(2\beta n+o(n))\frac{\lambda_1-px}{px-p\beta} -
2\beta(n-1)\alpha_2\left(\frac{\lambda_1-px}{px-p\beta}\right)^2}
\cdot \lambda_1^{\beta(p-n+1)-2}\cdot
e^{-(1-\frac{x-\beta}{2x})\lambda_1} d\lambda_1\notag\\
&~\times\int_{\lambda_2>\cdots>\lambda_n\atop  \lambda_{2}< p(\beta+\delta_n), \lambda_{n}> p(\beta-\delta_n)}
 e^{-2\beta\sum_{i=2}^n\frac{\lambda_i-p\beta}{p(x-\beta)}
-2\beta\alpha_2\sum_{i=2}^n\left(\frac{\lambda_i-p\beta}{px-p\beta}\right)^2}\notag\\
&\times
g_{n-1,p-1,\beta}(\lambda_2,\cdots,\lambda_n)d\lambda_2\cdots d\lambda_n\notag\\
\sim&~
 2x(x-\beta)^{-1}n^2 A_n^2
 e^{-\frac{x-\beta}{2x} px-2\alpha_2\frac{\beta^3}{(x-\beta)^2}[\frac{n^2}{p}-\frac{O(1)n^3}{3p^2}]}
 (px-p\beta)^{2\beta(n-1)}\notag\\
&~ \times  \int_{px}^\infty e^{(2\beta n+o(n))\frac{\lambda_1-px}{px-p\beta} -
2\beta(n-1)\alpha_2\left(\frac{\lambda_1-px}{px-p\beta}\right)^2 }
\cdot \lambda_1^{\beta(p-n+1)-2}\cdot
e^{-(1-\frac{x-\beta}{2x})\lambda_1} d\lambda_1\notag\\
=&~
2x(x-\beta)^{-1}n^2 A_n^2 e^{-\frac{x-\beta}{2x} px
-2\alpha_2\frac{\beta^3}{(x-\beta)^2}[\frac{n^2}{p}-\frac{O(1)n^3}{3p^2}]}
(px-p\beta)^{2\beta(n-1)}\notag\\
&~ \times  \int_{0}^\infty e^{\frac{(2\beta n+o(n))\lambda_1}{px-p\beta} -2\beta(n-1)\alpha_2\left(\frac{\lambda_1}{px-p\beta}\right)^2
-(1-\frac{x-\beta}{2x})(\lambda_1+px)}
\cdot (\lambda_1+px)^{\beta(p-n+1)-2}d\lambda_1\notag\\
\lesssim &~
(2x)^2(x-\beta)^{-2}n^2 A_n^2 e^{-px-2\alpha_2\frac{\beta^3}{(x-\beta)^2}[\frac{n^2}{p}-\frac{O(1)n^3}{3p^2}]} (px-p\beta)^{2\beta(n-1)}
(px)^{\beta(p-n+1)-2},\notag\\&
\label{eq:sq1}
\end{align}
where $\alpha_2=\frac{1}{2}-\frac{\delta_n}{x-\beta}$; in the third step we changed variable $\lambda_1$ to $\lambda_1+px$ for the integral; the last step follows from the inequality that
\begin{eqnarray*}
&& \int_{0}^\infty e^{\frac{(2\beta n+o(n))\lambda_1}{px-p\beta} -2\beta(n-1)\alpha_2\left(\frac{\lambda_1}{px-p\beta}\right)^2
-(1-\frac{x-\beta}{2x})(\lambda_1+px)}
\cdot (\lambda_1+px)^{\beta(p-n+1)-2}d\lambda_1\\
 &\leq&
(px)^{\beta(p-n+1)-2}e^{-(1-\frac{x-\beta}{2x})px}
\int_{0}^\infty e^{\frac{(2\beta n+o(n))\lambda_1}{px-p\beta} -(1-\frac{x-\beta}{2x})\lambda_1}
e^{[\beta(p-n+1)-2]\frac{\lambda_1}{px}}d\lambda_1\\
&\sim&
(px)^{\beta(p-n+1)-2}e^{-(1-\frac{x-\beta}{2x})px}
\frac{2x}{x-\beta},
\end{eqnarray*}
where in the first step we used $(\lambda+px)^{{\beta(p-n+1)-2}}\leq (px)^{\beta(p-n+1)-2}\exp(\lambda_1/(px))$ and the second step we used $\frac{(2\beta n+o(n))\lambda_1}{px-p\beta} -(1-\frac{x-\beta}{2x})\lambda_1+[\beta(p-n+1)-2]\frac{\lambda_1}{px}\sim -\frac{x-\beta}{2x}\lambda_1.$ Easily, $\alpha_2[\frac{n^2}{p}-\frac{O(1)n^3}{3p^2}]-\frac{n^2}{2p} \to 0$.
Based on  {\eqref{eq:sq1}} and Theorem \ref{theorem1}, we know that
$$\frac{E^Q\big[(\frac{dP}{dQ})^2; \lambda_{(1)}>px\big]}{P(\lambda_{(1)}>px)^{2}}
\lesssim 1$$
provided $p/n^{5/3}\to\infty$.
Since
\begin{equation}\label{climb_mountain}
E^Q\Big[(\frac{dP}{dQ})^2; \lambda_{(1)}>px\Big]
\geq \Big\{E^Q\Big[(\frac{dP}{dQ}); \lambda_{(1)}>px\Big]\Big\}^2
=P(\lambda_{(1)}>px)^{2}
\end{equation}
by H\"older's inequality, we have
\begin{eqnarray}\label{array}
\frac{E^Q\big[(\frac{dP}{dQ})^2; \lambda_{(1)}>px\big]}{P(\lambda_{(1)}>px)^{2}}
\sim 1.
\end{eqnarray}
So the second statement of the theorem is obtained. 

Now we prove the first one.
Recall $P^*_{px}=P(\cdot|\lambda_{(1)}>px)$  as described in the paragraph following (\ref{estr}). It is easy to check that
$$\frac{dP^*_{px}}{dQ}=\frac{dP}{dQ}\cdot\frac{I(\lambda_{(1)}>px)}{P(\lambda_{(1)}>px)}$$ a.s.  with respect to $Q$ defined on ${\cal{B}}([0, \infty)^n)$.
For any $A\subset {\cal{B}}([0, \infty)^n)$,
\begin{eqnarray*}
|Q(A)-P_{px}^*(A)|&=&\Big|\int_A\Big(\frac{dP^*_{px}}{dQ}-1\Big)\,dQ\Big|\\
&\leq & \biggr\{E^Q\Big(\frac{dP}{dQ}\cdot\frac{I(\lambda_{(1)}>px)}{P(\lambda_{(1)}>px)}-1\Big)^2
\biggr\}^{1/2}\\
& = & \Big(\frac{E^Q\big[(\frac{dP}{dQ})^2; \lambda_{(1)}>px\big]}{P(\lambda_{(1)}>px)^{2}}-1\Big)^{1/2} \to 0
\end{eqnarray*}
by H\"{o}lder's inequality and (\ref{array}).
Thus,  $$\lim_{n\rightarrow\infty} \sup_{A\in {\cal F}}|Q(A)-P^*_{px}(A)|=0$$ as $p/n^{5/3}\to\infty$. This gives the first conclusion of part (i).

We next prove the conclusion in (ii) and (iii). For general $p/n\to\infty$, by Lemma \ref{lemma3}, we have
\begin{eqnarray*}
& &  E^Q\left[L_p^2;\, \lambda_{(1)}>px\right] \nonumber\\
 & \sim &
 E^Q\left[L_p^2;\, \lambda_{(1)}>px,p(\beta+\delta_n)>\lambda_{(2)}\cdots>\lambda_{(n)}>p(\beta-\delta_n) \right] \\
 &&+ o(1) P(\lambda_{(1)}>px)^2.\ \
\end{eqnarray*}
Note that $P(\lambda_{(1)}>px)^2 \leq E^Q\left[L_p^2;\, \lambda_{(1)}>px\right]$. Then we have
\begin{eqnarray*}
  E^Q\left[L_p^2;\, \lambda_{(1)}>px\right]
  \sim
E^Q\left[L_p^2;\, \lambda_{(1)}>px,p(\beta+\delta_n)>\lambda_{(2)}\cdots>\lambda_{(n)}>p(\beta-\delta_n) \right] .
\end{eqnarray*}
Then following exactly the same argument as in \eqref{eq:sq} and \eqref{eq:sq1}, which requires the assumption that $p/n\to\infty$ only, we have
\begin{eqnarray*}
& &   E^Q\left[L_p^2;\, \lambda_{(1)}>px\right]\\
& \lesssim & (2x)^2(x-\beta)^{-2}n^2 A_n^2 e^{-px-2\alpha_2\frac{\beta^3}{(x-\beta)^2}[\frac{n^2}{p}-\frac{O(1)n^3}{3p^2}]} \\
&& \times ~(px-p\beta)^{2\beta(n-1)}
(px)^{\beta(p-n+1)-2}.
\end{eqnarray*}
Similar to the proof of Theorem \ref{theorem1}, this implies that
\begin{equation}\label{ax}
	E^Q\left[L_p^2;\, \lambda_{(1)}>px\right]\leq \exp\Big\{2B_{n,p,\beta}(x)+ O(1)\frac{ n^{5/2}}{p^{3/2}}\Big\}.
\end{equation}
Recall that $L_p=\frac{dP}{dQ}1\{\lambda_{(1)} > px\}$. Then, the  ratio $${E^Q\big[L_p^2; \lambda_{(1)}>px\big]}/{P(\lambda_{(1)}>px)^{2}}=O(1)$$  provided $n^{5/3}/p=O(1)$; \eqref{ax} together with (\ref{climb_mountain}) further imply that
$$\lim_{n\to\infty}\frac{\log E^Q\big[(\frac{dP}{dQ})^2; \lambda_{(1)}>px\big]}{2\log P(\lambda_{(1)}>px) }
 = 1 $$ as $p/n\to\infty$.  The proof is complete.
  \end{proof}

\appendix
\section*{Acknowledgment}
 The authors thank the editor, an associate editor, and an anonymous reviewer for many helpful and constructive comments.

\section*{Supplementary Material}
%slink[url]{http://www.e-publications.org/ims/support/dowload/imsart-ims.zip}
The online Supplementary Material contains proofs of technical lemmas and Theorem 3.
%\end{supplement}

\bibliographystyle{imsart-nameyear}
\bibliography{minimalBib}

\newpage

\title{Supplementary Material: Rare-event Analysis for Extremal Eigenvalues of white Wishart matrices}
\begin{aug}
\author{Tiefeng Jiang, Kevin Leder, and Gongjun Xu}
\affiliation{University of Minnesota}
\end{aug}

\appendix

\bigskip
This supplementary material contains proofs of technical lemmas in Appendix \ref{prooflemma} and Theorem 3 in Appendix \ref{aaax}.

\section{Auxiliary Lemmas}\label{prooflemma}
In this section we will use $\chi^2_k$ to refer to chi-squared random variables with $k$ degrees of freedom. The parameter $k$ does not have to be an integer.

\begin{lemma}\label{lemmabound}
For $\alpha_1>1/2$ and $0<\alpha_2<1/2$, we have
\begin{eqnarray*}
&&  \log(1-z)\geq -z-\alpha_1z^2, \quad \mbox{ if } z\in \Big(-1,1-\frac{1}{2\alpha_1}\Big);
\\
&&  \log(1+z)\leq z-\alpha_2z^2, \quad \mbox{ if } z\in \Big(-1,\frac{1}{2\alpha_2}-1\Big).
\end{eqnarray*}
\end{lemma}
The proof of Lemma \ref{lemmabound} follows from basic calculation and is therefore omitted.

\begin{lemma}\label{Starbucks} The following are true for chi-square distributions.

(a) $E(\chi^2_k-k)^4=O(k^2)$ as $k\to\infty$.

(b) $Var(\chi^2_k\chi^2_l)\leq 4 kl(k+l)$ for any $k>0$ and $l>0$ such that $k+l>2$, where $\chi^2_k$ and $\chi^2_l$ are independent.

(c) $Var((X+c)^2)\leq 2 Var(X^2) + 8c^2Var(X)$ for any random variable $X$ and constant $c>0$.
\end{lemma}
\begin{proof}[Proof] (a)
If $k$ is an integer, we know $E(\chi^2_k-k)^4=O(k^2)$ as $k\to\infty$ since $\chi^2_k$ is a sum of independent random variables with distribution $N(0,1)^2$ (see, e.g., p. 368 in Chapter 10 of \cite{ChowTeicher98}). In general, write $k=[k]+\{k\}$ where $[k]$ is the integer part. Then, by the additive property of the chi-square distribution, $\chi^2_k$ has the same distribution as that of $\chi^2_{[k]}+ \chi^2_{\{k\}}$ where the two random variables are independent. It follows that
\begin{eqnarray*}
E(\chi^2_k-k)^4 &\leq & 8 E(\chi^2_{[k]}-[k])^4 + 8E(\chi^2_{\{k\}}-\{k\})^4\\
& \leq & O(k^2) + 64E(\chi^2_{\{k\}})^4+ 64\{k\}^4=O(k^2)
\end{eqnarray*}
as $k\to\infty$ since $E(\chi^2_{k})^4$ is increasing in $k$ due to the representation of $\chi^2_k$ as a sum of $k$ $N(0,1)^2$.

(b) Easily since for $k+l>2$,
\begin{eqnarray*}
Var(\chi^2_k\chi^2_l) & = & E(\chi^2_k\chi^2_l)^2 - (E\chi^2_kE\chi^2_l)^2\\
& = & E(\chi^2_k)^2E(\chi^2_l)^2-k^2l^2\\
& = & (k^2+2k)(l^2+2l)-k^2l^2 \leq 4 kl(k+l).
\end{eqnarray*}

(c) Evidently,
\begin{eqnarray*}
Var((X+c)^2)&=& E[(X^2-EX^2) + 2c (X-EX)]^2\\
& \leq & 2 Var(X^2) + 8c^2Var(X).
\end{eqnarray*}
\end{proof}

\begin{lemma}\label{lemma1'} Let $(\lambda_1, \cdots, \lambda_n)$ have density $f_{n, p, \beta}(\lambda_1, \cdots, \lambda_n)$ as in (\ref{density}). Then,
$$
\sum_{i=1}^n(\lambda_i-p\beta)
= O_p(n^{1/2}p^{1/2}),\quad
\sum_{i=1}^n(\lambda_i-p\beta)^2
=\beta^2n^{2}p+O_p(np)
$$
provided $p/n\to\infty$. In addition, for any $\alpha,\gamma>0$ and $p/n\to\infty$, we have
$$
E\Big[e^{-\gamma\sum_{i=1}^n\frac{(\lambda_i-\beta p)}{p(x-\beta)}-\alpha\gamma\sum_{i=1}^n\frac{(\lambda_i-\beta p)^2}{p^2(x-\beta)^2}}\Big]
=
e^{-\alpha\gamma\frac{\beta^2}{(x-\beta)^2}[\frac{n^{2}}{p}-\frac{O(1)n^{3}}{3p^2}]}.$$
\end{lemma}
\begin{proof}[Proof]
From \cite{dumitriu2002matrix}, the eigenvalues
 $(\lambda_1,\cdots,\lambda_n)$ with density function \eqref{density}
 has the same distribution as the eigenvalues of the matrix
$\LL_{n,p,\beta}=\BB_{n,p,\beta} \BB_{n,p,\beta}^\top,$ where $\BB_{n,p,\beta}$ is a bidiagonal matrix defined as
\begin{eqnarray*}
\BB_{n,p,\beta}=\left(\begin{array}{ccccc}
\chi_{\beta p} &\\
\chi_{\beta(n-1)}& \chi_{\beta p-\beta}&\\~\\
 &\ddots&\ddots\\~\\
&&\chi_{\beta}& \chi_{\beta p-\beta(n-1)}&\\
\end{array}\right)_{n\times n}.
\end{eqnarray*}
Here all of the diagonal and sub-diagonal elements are mutually independent and the notation $\chi_{a}$ stands for the square root of  $\chi^2_a$, the chi-square distribution of degree $a$.
This gives
\begin{eqnarray*}
\LL_{n,p,\beta}=\left(\begin{array}{rrrrr}
\chi_{\beta p}^2 &\chi_{\beta(n-1)}\chi_{\beta p}\\~\\
\chi_{\beta(n-1)}\chi_{\beta p}& \chi^2_{\beta (p-1)}+\chi_{\beta(n-1)}^2&\chi_{\beta(n-2)}\chi_{\beta (p-1)}\\~\\
\ddots&\ddots&\ddots\\~\\
&\chi_{\beta+1}\chi_{\beta (p-(n-3))}& \chi_{\beta (p-(n-2))}^2+\chi_{2\beta}^2&\chi_{\beta}\chi_{\beta (p-(n-2))}&\\~\\
&&\chi_{\beta}\chi_{\beta (p-(n-2))}\quad & \chi_{\beta (p-(n-1))}^2+\chi_{\beta}^2&\\
\end{array}\right)_{n\times n}.
\end{eqnarray*}
Then, we know that
\begin{eqnarray*}
&&\sum_{i=1}^n\lambda_i\sim^d Tr(\LL_{n,p,\beta}),\quad
\sum_{i=1}^n\lambda_i^2\sim^d Tr(\LL_{n,p,\beta}\LL_{n,p,\beta}),\\
&&\sum_{i=1}^n(\lambda_i-\beta p)^2\sim^d Tr((\LL_{n,p,\beta}-\beta p \II_n)(\LL_{n,p,\beta}-\beta p \II_n)),
\end{eqnarray*} where $\II_n$ is the $n\times n$ identity matrix and $X\sim^d Y$ denotes that $X$ and $Y$ have the same distribution.
This implies
\begin{eqnarray}
\sum_{i=1}^n\lambda_i
&\sim^d&
\chi_{\beta p}^2+\sum_{i=1}^{n-1}[\chi_{\beta(p-i)}^2 + \chi_{\beta(n-i)}^2]
\sim ^d \chi_{\beta pn}^2
, \nonumber\\
\sum_{i=1}^n(\lambda_i-\beta p)^2
&\sim^d&
(\chi_{\beta p}^2-\beta p)^2+\sum_{i=1}^{n-1}[\chi_{\beta(p-i)}^2 + \chi_{\beta(n-i)}^2-\beta p]^2\notag\\
&&+2\sum_{i=1}^{n-1}\chi_{\beta(p+1-i)}^2\chi_{\beta(n-i)}^2.
\label{variance_changchun}
\end{eqnarray}
With the result for the distribution of $\sum_{i=1}^n\lambda_i$ we can apply Chebyshev's inequality to  see that
 $$\sum_{i=1}^n(\lambda_i-p\beta)
= O_p(n^{1/2}p^{1/2}).$$
In addition,
by using independence and the facts that $E\chi^2_a=a$ and $Var(\chi^2_a)=2a$ for all $a>0$, we have
\begin{eqnarray}
&&E\Big[\sum_{i=1}^n(\lambda_i-\beta p)^2\Big] \nonumber\\
&=&
E\Big[(\chi_{\beta p}^2-\beta p)^2+\sum_{i=1}^{n-1}[\chi_{\beta(p-i)}^2 + \chi_{\beta(n-i)}^2-\beta p]^2
+2\sum_{i=1}^{n-1}\chi_{\beta(p+1-i)}^2\chi_{\beta(n-i)}^2\Big]\nonumber \\
&=&
2\beta p+\sum_{i=1}^{n-1}[2\beta(p+n-2i)+\beta^2(n-2i)^2]
+2\sum_{i=1}^{n-1}\beta^2(p+1-i)(n-i) \nonumber.
\end{eqnarray}
A direct calculation gives that
\begin{eqnarray*}
\sum_{i=1}^{n-1}[2\beta(p+n-2i)+\beta^2(n-2i)^2] &=& 2\beta p(n-1)-\beta^2n^2(n-1)+4\beta^2\sum_{i=1}^{n-1}i^2,\\
2\sum_{i=1}^{n-1}\beta^2(p+1-i)(n-i)&=&\beta^2n(n-1)(p-n+1)+2\beta^2\sum_{i=1}^{n-1}i^2.
\end{eqnarray*}
Since $\sum_{i=1}^{n-1}i^2=(n-1)n(2n-1)/6,$ it follows that
\begin{eqnarray}
E\Big[\sum_{i=1}^n(\lambda_i-\beta p)^2\Big]
=
\beta^2n^2p+O(pn). \label{hot_cold}
\end{eqnarray}

Now we bound $Var(\sum_{i=1}^n(\lambda_i-\beta p)^2)$ as in (\ref{variance_changchun}). Trivially,
\begin{eqnarray*}
Var(X+Y+Z) \leq 3 (Var(X) + Var(Y) + Var(Z))
\end{eqnarray*}
for any random variables $X, Y, Z$. By (\ref{variance_changchun}),
\begin{eqnarray*}
 Var(\sum_{i=1}^n(\lambda_i-\beta p)^2)
& \leq &  3E(\chi_{\beta p}^2-\beta p)^4
 + 3\sum_{i=1}^{n-1}Var\big([\chi_{\beta(p-i)}^2 + \chi_{\beta(n-i)}^2-\beta p]^2\big)
 \\
&&+ 12 \sum_{i=1}^{n-1}Var\big(\chi_{\beta(p+1-i)}^2\chi_{\beta(n-i)}^2\big).
\end{eqnarray*}
From Lemma \ref{Starbucks}, $E(\chi_{\beta p}^2-\beta p)^4= O(p^2)$ and
\begin{eqnarray*}
 \sum_{i=1}^{n-1}Var\big(\chi_{\beta(p+1-i)}^2\chi_{\beta(n-i)}^2\big)
& \leq & 4 \beta^2  \sum_{i=1}^{n-1} (p+1-i)(n-i)(p+n+1-2i)\\
& \leq & 4\beta^2p^2n^2.
\end{eqnarray*}
Further, \begin{eqnarray*}
&&Var\big([\chi_{\beta(p-i)}^2 + \chi_{\beta(n-i)}^2-\beta p]^2\big) \\
&=& Var\big([(\chi_{\beta(p-i)}^2 -\beta(p-i))+( \chi_{\beta(n-i)}^2-\beta(n-i))+\beta (n-2i)]^2\big) \\
& \leq & 16\, E[(\chi_{\beta(p-i)}^2 -\beta(p-i))^4] + 16\,E[(\chi_{\beta(n-i)}^2-\beta(n-i))^4] \\
&&+ 8\beta^2(n-2i)^2(p+n-2i)\\
& \leq & O(p^2) + O(n^2) + O(pn^2)= O(p^2+pn^2)
\end{eqnarray*}
uniformly for all $1\leq i \leq n$.  Here the first inequality follows from an application of Lemma \ref{Starbucks}(c) and the second inequality follows from $E(\chi^2(k)-k)^4=O(k^2)$ by Lemma \ref{Starbucks}(a).
This says that
\begin{eqnarray*}
3\sum_{i=1}^{n-1}Var\big([\chi_{\beta(p-i)}^2 + \chi_{\beta(n-i)}^2-\beta p]^2\big) =O(np^2+pn^3).
\end{eqnarray*}
Combine the above to have
$
Var(\sum_{i=1}^n(\lambda_i-\beta p)^2)=O(n^2p^2),
$
which together with (\ref{hot_cold}) implies
$$\sum_{i=1}^n(\lambda_i-p\beta)^2=\beta^2n^{2}p+O_p(np).$$

Next we show that for $\alpha, \beta, \gamma>0$ and $p/n\to\infty$,
$$E\Big[e^{-\gamma\sum_{i=1}^n\frac{(\lambda_i-\beta p)}{p(x-\beta)}-\alpha\gamma\sum_{i=1}^n\frac{(\lambda_i-\beta p)^2}{p^2(x-\beta)^2}}\Big]=
 e^{-\alpha\gamma\frac{\beta^2}{(x-\beta)^2}[\frac{n^{2}}{p}-\frac{O(1)n^{3}}{3p^2}]}.$$
We have the following lower bound:
\begin{eqnarray*}
&&\sum_{i=1}^n\frac{(\lambda_i-\beta p)}{p(x-\beta)}+\alpha\sum_{i=1}^n\frac{(\lambda_i-\beta p)^2}{p^2(x-\beta)^2} \nonumber\\
&\sim^d&
\frac{1}{{p(x-\beta)}}
\Big\{ \chi_{\beta p}^2+\sum_{i=1}^{n-1}[\chi_{\beta(p-i)}^2 + \chi_{\beta(n-i)}^2]
-\beta pn\Big\}
\nonumber \\
&&+\frac{\alpha}{{p^2(x-\beta)^2}}\Big\{(\chi_{\beta p}^2-\beta p)^2+\sum_{i=1}^{n-1}[\chi_{\beta(p-i)}^2 + \chi_{\beta(n-i)}^2-\beta p]^2\notag\\
&&\quad\quad+2\sum_{i=1}^{n-1}\chi_{\beta(p+1-i)}^2\chi_{\beta(n-i)}^2\Big\} \nonumber\\
&\geq&
\frac{1}{{p(x-\beta)}}\Big\{\sum_{i=1}^{n-1}[\chi_{\beta(p+1-i)}^2 + \chi_{\beta(n-i)}^2] +\chi_{\beta(p-n+1)}^2-\beta pn\Big\} \nonumber\\
&&+\frac{2\alpha}{{p^2(x-\beta)^2}}\Big\{\sum_{i=1}^{n-1}\chi_{\beta(p+1-i)}^2\chi_{\beta(n-i)}^2\Big\}
\end{eqnarray*}
since $\chi_{\beta p}^2+\sum_{i=1}^{n-1}[\chi_{\beta(p-i)}^2 + \chi_{\beta(n-i)}^2]= \sum_{i=1}^{n-1}[\chi_{\beta(p+1-i)}^2 + \chi_{\beta(n-i)}^2] +\chi_{\beta(p-n+1)}^2$. For any $0\leq i\leq n-1$,  we have
\begin{eqnarray*}
&&E\left[e^{-\frac{\gamma}{{p(x-\beta)}} [\chi_{\beta(p+1-i)}^2 + \chi_{\beta(n-i)}^2]
 -\frac{2\alpha\gamma}{p^2(x-\beta)^2}\chi_{\beta(p+1-i)}^2\chi_{\beta(n-i)}^2}\right]\\
&=&\frac{2^{-\frac{\beta(p+1-i)+\beta(n-i)}{2}}}{\Gamma(\frac{\beta(p+1-i)}{2})
\Gamma(\frac{\beta(n-i)}{2})}
\int_0^\infty\int_0^\infty y^{\frac{\beta(p+1-i)}{2}-1}z^{\frac{\beta(n-i)}{2}-1}
e^{-\frac{\gamma(y+z)}{{p(x-\beta)}}
 -\frac{2\alpha\gamma yz}{p^2(x-\beta)^2} -\frac{y+z}{2}}dydz\\
 &=&\frac{2^{-\frac{\beta(n-i)}{2}}}{\Gamma(\frac{\beta(n-i)}{2})}
\int_0^\infty
z^{\frac{\beta(n-i)}{2}-1}
e^{-\frac{\gamma z}{{p(x-\beta)}} -\frac{z}{2}}\\
&&\times
\frac{2^{-\frac{\beta(p+1-i)}{2}}}{\Gamma(\frac{\beta(p+1-i)}{2})}
 \int_0^\infty y^{\frac{\beta(p+1-i)}{2}-1}
e^{-\frac{\gamma y}{{p(x-\beta)}}
 -\frac{2\alpha\gamma yz}{p^2(x-\beta)^2} -\frac{y}{2}}dydz\\
  &=&\frac{2^{-\frac{\beta(n-i)}{2}}}{\Gamma(\frac{\beta(n-i)}{2})}
\int_0^\infty
z^{\frac{\beta(n-i)}{2}-1}
e^{-\frac{\gamma z}{{p(x-\beta)}} -\frac{z}{2}}\times
\Big(1+ \frac{2\gamma}{p(x-\beta)}+\frac{4\alpha\gamma z}{p^2(x-\beta)^2}\Big)^{-\frac{\beta(p+1-i)}{2}}
dz.
\end{eqnarray*}
For large constant $M>[\beta]+1$, the previous display is bounded by
\begin{eqnarray*}
  &\leq&\frac{2^{-\frac{\beta(n-i)}{2}}}{\Gamma(\frac{\beta(n-i)}{2})}
\int_0^{Mn}
z^{\frac{\beta(n-i)}{2}-1}
e^{-\frac{\gamma z}{{p(x-\beta)}} -\frac{z}{2}}\times
\Big(1+ \frac{2\gamma}{p(x-\beta)}+\frac{4\alpha\gamma z}{p^2(x-\beta)^2}\Big)^{-\frac{\beta(p+1-i)}{2}}
dz \\
&&+P(\chi_{\beta (n-i)}^2>Mn)\Big(1+ \frac{2\gamma}{p(x-\beta)}+\frac{4\alpha\gamma Mn}{p^2(x-\beta)^2}\Big)^{-\frac{\beta(p+1-i)}{2}}
\\
  &\leq&\frac{2^{-\frac{\beta(n-i)}{2}}}{\Gamma(\frac{\beta(n-i)}{2})}
\int_0^{\infty}
z^{\frac{\beta(n-i)}{2}-1}
e^{-\frac{\gamma z}{{p(x-\beta)}} -\frac{z}{2}-\frac{\beta(p+1-i)}{2}[\frac{2\gamma}{p(x-\beta)}+\frac{4\alpha\gamma z}{p^2(x-\beta)^2}+O(\frac{1}{p^2})]}dz \\
&&+P(\chi_{\beta (n-i)}^2>Mn)e^{-\frac{\beta(p+1-i)}{2}[\frac{2\gamma}{p(x-\beta)}+\frac{4\alpha\gamma Mn}{p^2(x-\beta)^2}+O(\frac{1}{p^2})]}\\
&\lesssim&\Big(1+\frac{2\gamma }{{p(x-\beta)}}+\frac{4\alpha\gamma\beta(p+1-i)}{p^2(x-\beta)^2}\Big)^{-\frac{\beta(n-i)}{2}}e^{-\frac{\gamma\beta(p+1-i)}{p(x-\beta)}}\\
&&+e^{-Cn}\times e^{-\frac{\beta(p+1-i)}{2}[\frac{2\gamma}{p(x-\beta)}+\frac{4\alpha\gamma Mn}{p^2(x-\beta)^2}+O(\frac{1}{p^2})]}\\
&\sim&e^{-(\frac{2\gamma }{{p(x-\beta)}}+\frac{4\alpha\gamma\beta(p+1-i)}{p^2(x-\beta)^2}){\frac{\beta(n-i)}{2}}}e^{-\frac{\gamma\beta(p+1-i)}{p(x-\beta)}}
=e^{-\frac{\gamma\beta(p+n-2i+1) }{{p(x-\beta)}}-\frac{2\alpha\gamma\beta^2(p+1-i)(n-i)}{p^2(x-\beta)^2}},
\end{eqnarray*}
where we use the fact $\frac{2\gamma}{p(x-\beta)}+\frac{4\alpha\gamma z}{p^2(x-\beta)^2}
=O(\frac{1}{p})$
 uniformly for all $0\leq z \leq Mn$ in the second inequality due to the fact $n=o(p)$; the Chernoff bound for $\sup_{1\leq i \leq n-1}P(\chi_{\beta (n-i)}^2>Mn)\leq P(\chi_{([\beta]+1) n}^2>Mn) \leq e^{-Cn}$ with some constant $C=C_{\beta,M}>0$ (see, e.g., p. 31 from \cite{dembo2009large}) is used in the third inequality.
  Similarly, we have
\begin{eqnarray*}
E\left[e^{-\frac{\gamma}{{p(x-\beta)}}\chi_{\beta(p-n+1)}^2 }\right]
= \Big(1+\frac{2\gamma}{{p(x-\beta)}}\Big)^{-\frac{\beta(p-n+1)}{2}}\sim e^{-\frac{\gamma\beta}{x-\beta}}.
\end{eqnarray*}
Since for different $i$, the $\chi$-distributed random variables are independent,
 from the above result  we know
\begin{eqnarray*}
&&E\Big[e^{-\gamma\sum_{i=1}^n\frac{(\lambda_i-\beta p)}{p(x-\beta)}-\gamma\alpha\sum_{i=1}^n\frac{(\lambda_i-\beta p)^2}{p^2(x-\beta)^2}}\Big]\\
&\leq&E\left[e^{-\sum_{i=1}^{n-1}\left\{\frac{\gamma}{{p(x-\beta)}} [\chi_{\beta(p+1-i)}^2 + \chi_{\beta(n-i)}^2]
+\frac{2\alpha\gamma}{p^2(x-\beta)^2}\chi_{\beta(p+1-i)}^2\chi_{\beta(n-i)}^2\right\}
-\frac{\gamma}{{p(x-\beta)}}\chi_{\beta(p-n+1)}^2 +\frac{\gamma\beta n}{x-\beta}}\right]\\
 &\lesssim&e^{-\sum_{i=1}^{n-1}\{\frac{\gamma\beta(p+n-2i+1) }{{p(x-\beta)}}+\frac{2\alpha\gamma\beta^2(p+1-i)(n-i)}{p^2(x-\beta)^2}\}}e^{-\frac{\gamma\beta}{x-\beta}}e^{\frac{\gamma\beta n}{x-\beta}}\\
 &\sim&e^{-\sum_{i=1}^{n-1}\frac{2\alpha\gamma\beta^2(p+1-i)(n-i)}{p^2(x-\beta)^2}}\\
 &\sim& e^{-\alpha\gamma\frac{\beta^2}{(x-\beta)^2}[\frac{n^{2}}{p}-\frac{n^{3}}{3p^2}]},
\end{eqnarray*}
where $-\sum_{i=1}^{n-1}\frac{\gamma\beta(p+n-2i+1) }{p(x-\beta)}-\frac{\gamma\beta}{x-\beta}+ \frac{\gamma\beta n}{x-\beta}=-\frac{r\beta(n-1)}{p(x-\beta)}=o(1)$ is used in the third step. This implies that
$$\sup_{n\geq 2} E\Big[e^{-\gamma\sum_{i=1}^n\frac{(\lambda_i-\beta p)}{p(x-\beta)}-\alpha\gamma\sum_{i=1}^n
\frac{(\lambda_i-\beta p)^2}{p^2(x-\beta)^2}+\alpha\gamma\frac{\beta^2}{(x-\beta)^2}
[\frac{n^{2}}{p}-\frac{n^{3}}{3p^2}]}
\Big]< \infty.$$
Since the above result holds for any $\gamma$,  take $\gamma_K =K\gamma$ with $K>1$ to have
$$\sup_{n\geq 2} E\Big[e^{-\gamma_K\sum_{i=1}^n\frac{(\lambda_i-\beta p)}{p(x-\beta)}-\alpha\gamma_K\sum_{i=1}^n\frac{(\lambda_i-\beta p)^2}{p^2(x-\beta)^2}+\alpha\gamma_K\frac{\beta^2}{(x-\beta)^2}[\frac{n^{2}}{p}-\frac{n^{3}}{3p^2}]}\Big]< \infty.$$
This implies the uniform integrability of $e^{-\gamma\sum_{i=1}^n\frac{(\lambda_i-\beta p)}{p(x-\beta)}-\alpha\gamma\sum_{i=1}^n\frac{(\lambda_i-\beta p)^2}{p^2(x-\beta)^2}+\alpha\gamma\frac{\beta^2}{(x-\beta)^2}[\frac{n^{2}}{p}-\frac{n^{3}}{3p^2}]}$. 
Further,  by the previous results that $\sum_{i=1}^n(\lambda_i-p\beta)
= O_p(n^{1/2}p^{1/2})$ and $\sum_{i=1}^n(\lambda_i-p\beta)^2=\beta^2n^{2}p+O_p(np)$,
 we have for $p/n\to\infty$,
  $$E\Big[e^{-\gamma\sum_{i=1}^n\frac{(\lambda_i-\beta p)}{p(x-\beta)}-\alpha\gamma\sum_{i=1}^n\frac{(\lambda_i-\beta p)^2}{p^2(x-\beta)^2}}\Big]= e^{-\alpha\gamma\frac{\beta^2}{(x-\beta)^2}[\frac{n^{2}}{p}-\frac{O(1)n^{3}}{3p^2}]}.$$
This completes the proof.
\end{proof}

\begin{lemma}\label{lemma1}
Consider the order statistics $\lambda_{(1)}>\cdots>\lambda_{(n)}$ as defined in (\ref{orderdensity}).
For $p/n\rightarrow\infty$ and $\delta_n>0$ such that $\delta_n\to 0$ and $\frac{\delta_n^2}{np^{-1}\log (n^{-1}p)}\to\infty$ as $n\to\infty$, we have
 $$\log P(\lambda_{(1)}> p(\beta+\delta_n))\lesssim -\frac{p\delta_n^2}{4\beta}, \quad \log P(\lambda_{(n)}< p(\beta-\delta_n))\lesssim-\frac{p\delta_n^2}{4\beta}.$$
\end{lemma}
\begin{proof}[Proof] Recall $A_n$ as in (\ref{An}).
We have
 \begin{eqnarray*}
&& P(\lambda_{(1)}> p(\beta+\delta_n))\\
&= & \int_{\lambda_1>\cdots>\lambda_n,~ \lambda_{1}>p(\beta+\delta_n)}  n A_n g_{n-1,p-1,\beta}(\lambda_2,\cdots,\lambda_n)  \prod_{i=2}^n(\lambda_1-\lambda_i)^\beta
\cdot \lambda_1^{\frac{\beta(p-n+1)}{2}-1}\\
&&\times
e^{-\frac{1}{2} \lambda_1}d\lambda_1\cdots d\lambda_n \\
&\leq& n A_n \int_{p(\beta+\delta_n)}^\infty \lambda_1^{\beta \frac{p+n-1}{2}-1}e^{-\frac{\lambda_1}{2}}d\lambda_1 \int_{\lambda_2>\cdots>\lambda_n}   g_{n-1,p-1,\beta}(\lambda_2,\cdots,\lambda_n)  d\lambda_2\cdots d\lambda_n \\
&= & n A_n \int_{p(\beta+\delta_n)}^\infty \lambda_1^{\beta \frac{p+n-1}{2}-1}e^{-\frac{\lambda_1}{2}}d\lambda_1.
\end{eqnarray*}
For any $x>\beta$, we have  from (5.6) in Jiang and Li (2014) that  for large $p$,
\begin{equation}\label{gx1}
\int_{px}^\infty \lambda_1^{\beta \frac{p+n-1}{2}-1}e^{-\frac{\lambda_1}{2}}d\lambda_1\leq
 \frac{2}{px-p\beta-\beta n+\beta-2} (px)^{\frac{\beta (p+n-1)}{2}}e^{-\frac{px}{2}}.
 \end{equation}
In addition, we know from \eqref{logA}
\begin{eqnarray*}
\log A_n
= -\frac{\beta p}{2}\log p -\frac{\beta p}{2}\left(\log {\beta}-1\right)
-\frac{\beta n}{2}\log n-\frac{\beta n}{2}\left(\log{\beta}-1\right)+O(\log p).
\end{eqnarray*}
The above results then imply that
 \begin{align}\label{gx0}
&\log P(\lambda_{(1)}> p(\beta+\delta_n))\notag\\
\leq &~ \log n+\log A_n+\log \int_{p(\beta+\delta_n)}^\infty \lambda_1^{\beta \frac{p+n-1}{2}-1}e^{-\frac{\lambda_1}{2}}d\lambda_1\notag\\
\leq&~ \log n -\frac{\beta p}{2}\log p -\frac{\beta p}{2}\left(\log{\beta}-1\right)
-\frac{\beta n}{2}\log n-\frac{\beta n}{2}\left(\log {\beta}-1\right)+O(\log p)\notag\\
&~-\log (p\delta_n)+\frac{\beta(p+n-1)}{2}(\log p+\log (\beta+\delta_n))-\frac{p\beta +p\delta_n}{2} +O(1)
\notag\\
= &~ -(1+o(1))\frac{p\delta_n^2}{4\beta} +\frac{\beta n}{2}\log \frac{p}{n} +O(\log p+n),\notag\\&
\end{align}
where in the last step we used  $\log (\beta+\delta_n)=\log \beta +\frac{\delta_n}{\beta}-\frac{\delta_n^2}{2\beta^2}(1 +o(1))$.
Therefore, since $p/n\rightarrow\infty$ and $\frac{\delta_n^2}{np^{-1}\log (n^{-1}p)}\to\infty$, we have
$$ \log P(\lambda_{(1)}> p(\beta+\delta_n)) \lesssim -\frac{p\delta_n^2}{4\beta}.$$
Note that this implies $P(\lambda_{(1)}> p(\beta+\delta_n))=o(1).$

\medskip
Now we study $\lambda_{(n)}$. For some big $M>0$,
 \begin{eqnarray}\label{gx4}
&&  P(\lambda_{(n)}< p(\beta-\delta_n))\notag\\
&\leq&
 P(\lambda_{(n)}< p(\beta-\delta_n), \lambda_{(1)}<Mp) +P(\lambda_{(1)}>Mp).
 \end{eqnarray}
For the first term on the right hand side of \eqref{gx4}, we have from (5.24) in Jiang and Li (2014) that
 \begin{eqnarray}
&&\log P(\lambda_{(n)}< p(\beta-\delta_n), \lambda_{(1)}<Mp) \nonumber\\
&\leq & \log n+\log A_n+\log \int_0^{p(\beta-\delta_n)} (Mp)^{\beta(n-1)}\lambda_n^{\beta \frac{p-n+1}{2}-1}e^{-\frac{\lambda_n}{2}}d\lambda_n \nonumber\\
&\leq & \log n -\frac{\beta p}{2}\log p -\frac{\beta p}{2}\left(\log {\beta}-1\right)
-\frac{\beta n}{2}\log n-\frac{\beta n}{2}\left(\log {\beta}-1\right) \nonumber\\
&&+\beta(n-1)\log p-\frac{p\beta -p\delta_n}{2} +\frac{\beta(p-n+1)}{2}(\log p+\log (\beta-\delta_n))\nonumber\\
&&-\log (p\delta_n)+O(\log p+n)\nonumber\\
&= & -(1+o(1))\frac{p\delta_n^2}{4\beta} +\frac{\beta n}{2}\log \frac{p}{n} +O(\log p+n),\label{go_come}
\end{eqnarray}
where in the second step we used the approximation in \eqref{logA}; in the last step the equality $\log (\beta-\delta_n)=\log \beta -\frac{\delta_n}{\beta}-\frac{\delta_n^2}{2\beta^2}(1 +o(1))$ and the inequality
 (5.7) in Jiang and Li (2014) that for $y<\beta$,
\begin{equation}\label{gx2}
\int_{0}^{py} \lambda_n^{\frac{\beta (p-n+1)}{2}-1}e^{-\frac{\lambda_n}{2}}d\lambda_n\leq
 \frac{2}{p(\beta-y)-\beta n} (px)^{\frac{\beta (p-n+1)}{2}}e^{-\frac{px}{2}}
 \end{equation}
 are used.
Consider the second term in \eqref{gx4}. A similar argument as in \eqref{gx0} gives that for $x>\beta$
 \begin{eqnarray}
&&\log P(\lambda_{(1)}> px)\notag\\
&\leq & \log n+\log A_n+\log \int_{px}^\infty \lambda_1^{\beta \frac{p+n-1}{2}-1}e^{-\frac{\lambda_1}{2}}d\lambda_1\notag\\
&\leq & \log n -\frac{\beta p}{2}\log p -\frac{\beta p}{2}\left(\log{\beta}-1\right)
-\frac{\beta n}{2}\log n-\frac{\beta n}{2}\left(\log {\beta}-1\right)\notag\\
&&+O(\log p)-\log (px-p\beta)+\frac{\beta(p+n-1)}{2}\log (px)-\frac{px}{2} +O(1)
\notag\\
&= & p\left(\frac{\beta}{2}-\frac{\beta}{2}\log\beta-\frac{x}{2}+\frac{\beta}{2}\log x\right)
+O(n\log\frac{p}{n}).\label{gx3}
\end{eqnarray}
Then we have $\log P(\lambda_{(1)}>Mp)\leq -C p +O(n\log\frac{p}{n})$ with a constant $C=C_M>0$.
Combining this, (\ref{gx4}) and (\ref{go_come}), we get the desired result for $\lambda_{(n)}$.
\end{proof}

\begin{lemma}\label{lemma2}
For the order statistics $\lambda_{(1)}>\cdots>\lambda_{(n)}$ as defined in (\ref{orderdensity}),
if $p/n\rightarrow\infty$ and $x>\beta>y>0$, we have
\begin{eqnarray*}
\log P(\lambda_{(1)}> px) &=& p\left(\frac{\beta}{2}-\frac{\beta}{2}\log\beta-\frac{x}{2}+\frac{\beta}{2}\log x\right)
+O(n\log\frac{p}{n}),\\
\log P(\lambda_{(n)}< py) &=& p\left(\frac{\beta}{2}-\frac{\beta}{2}\log\beta-\frac{y}{2}+\frac{\beta}{2}\log y\right)
+O(n\log\frac{p}{n}).
\end{eqnarray*}
\end{lemma}
\begin{proof}[Proof] Recall $A_n$ as in (\ref{An}).
Take $\delta_n>0$ such that $\delta_n\to 0$ and $\frac{\delta_n^2}{np^{-1}\log (n^{-1}p)}\to\infty$ as $n\to\infty$.
We have
\begin{eqnarray*}
P(\lambda_{(1)}> px)
&\geq& P(\lambda_{(1)}>px, \lambda_{(2)}< p(\beta+\delta_n))\\
&= & \int_{\lambda_1>\cdots>\lambda_n,\atop \lambda_{1}>px,~ \lambda_{2}< p(\beta+\delta_n)} n A_n g_{n-1,p-1,\beta}(\lambda_2,\cdots,\lambda_n) \\
&&\times \prod_{i=2}^n(\lambda_1-\lambda_i)^\beta
\cdot \lambda_1^{\frac{\beta(p-n+1)}{2}-1}\cdot
e^{-\frac{1}{2} \lambda_1}d\lambda_1\cdots d\lambda_n \\
&\geq & nA_n \int_{px}^\infty (px-p\beta-p\delta_n)^{\beta(n-1)}\lambda_1^{\frac{\beta (p-n+1)}{2}-1}e^{-\frac{\lambda_1}{2}}d\lambda_1\\
&&
\times P( \lambda_{n-1,(1)}< p(\beta+\delta_n))\\
&\sim & nA_n \int_{px}^\infty (px-p\beta-p\delta_n)^{\beta(n-1)}\lambda_1^{\frac{\beta (p-n+1)}{2}-1}e^{-\frac{\lambda_1}{2}}d\lambda_1,
\end{eqnarray*}
where in the third step we used $\lambda_{(1)}-\lambda_{(i)}\geq \lambda_{(1)}-\lambda_{(2)}\geq px-p(\beta+\delta_n)$, $i=2,\cdots,n$,  the last step follows from Lemma \ref{lemma1} that $P( \lambda_{n-1,(1)}< p(\beta+\delta_n))\sim 1$,
and the notation  $ \lambda_{n-1,(1)}$ is a shorthand of  $\lambda_{(1)}$ with $\lambda_{(1)}>\cdots > \lambda_{(n-1)}$ having the density $g_{n-1, p-1, \beta}(\lambda_{(1)}, \cdots, \lambda_{(n-1)})$.
Applying the inequality  $\int_{px}^\infty \lambda_1^{\frac{\beta (p-n+1)}{2}-1}e^{-\frac{\lambda_1}{2}}d\lambda_1 \geq 2(px)^{\frac{\beta (p-n+1)}{2}-1}e^{-\frac{px}{2}}$ to the previous display and then using the result in \eqref{logA} we see that
\begin{eqnarray}\label{gx5}
&& \log P(\lambda_{(1)}> px) \notag\\
&\geq & \log n+\log A_n +{\beta(n-1)}\log  (px-p\beta)\notag\\
&& +[{\frac{\beta (p-n+1)}{2}-1}]\log (px) -\frac{px}{2} +o(n)
\notag\\
&=& \log n+\log A_n+[{\frac{\beta (p+n-1)}{2}-1}]\log (px) \notag\\
&&-\frac{px}{2}+{\beta(n-1)}\log  (1-\beta/x)  +o(n)\notag\\
&=&  p\left(\frac{\beta}{2}-\frac{\beta}{2}\log\beta-\frac{x}{2}+\frac{\beta}{2}\log x\right)
+\frac{\beta n}{2}\log \frac{p}{n} -\frac{\beta+1}{2}\log p
\notag\\
&& +\beta n\left(-\frac{\log x}{2}+\log(x-\beta)-\frac{{\log}\beta}{2}+\frac{1}{2}\right)+{o(n)}.
\end{eqnarray}
Evidently, $\log p=o(n\log \frac{p}{n}).$ Therefore, we obtain the corresponding lower bound. From \eqref{gx3} we know
 $$\log P(\lambda_{(1)}> px)\leq
 p\left(\frac{\beta}{2}-\frac{\beta}{2}\log\beta-\frac{x}{2}+\frac{\beta}{2}\log x\right) +O(n\log\frac{p}{n}).$$
Consequently,
  $$\log P(\lambda_{(1)}> px)=
 p\left(\frac{\beta}{2}-\frac{\beta}{2}\log\beta-\frac{x}{2}+\frac{\beta}{2}\log x\right) +O(n\log\frac{p}{n}).$$

Similarly, we have for $y<r<\beta$ and $\delta_n=n/p$,
\begin{eqnarray*}
&& P(\lambda_{(n)}< py)\\
&\geq& P(p(y-\delta_n)<\lambda_{(n)}<py, \lambda_{(n-1)}> pr)\\
&=&
\int_{\lambda_1>\cdots>\lambda_n,\atop p(y-\delta_n)<\lambda_{n}<py, \lambda_{n-1}> pr}
n A_n g_{n-1,p-1,\beta}(\lambda_1,\cdots,\lambda_{n-1})\\
&&\times
 \prod_{i=1}^{n-1}(\lambda_i-\lambda_n)^\beta
 \lambda_n^{\frac{\beta(p-n+1)}{2}-1}
e^{-\frac{1}{2} \lambda_n}d\lambda_1\cdots d\lambda_n \\
&\geq & nA_n \int_{p(y-\delta_n)}^{py} (pr-py)^{\beta(n-1)}\lambda_n^{\frac{\beta (p-n+1)}{2}-1}e^{-\frac{\lambda_n}{2}}d\lambda_n
\times P( \lambda_{n-1,(n-1)}> pr)\\
&\sim & nA_n \int_{p(y-\delta_n)}^{py} (pr-py)^{\beta(n-1)}\lambda_n^{\frac{\beta (p-n+1)}{2}-1}e^{-\frac{\lambda_n}{2}}d\lambda_n,
\end{eqnarray*}
where  $\lambda_{n-1,(n-1)}$ is a shorthand of  $\lambda_{(n-1)}$ with $\lambda_{(1)}>\cdots > \lambda_{(n-1)}$ having the density $g_{n-1, p-1, \beta}(\lambda_{1}, \cdots, \lambda_{n-1})$.
In the last step we used the approximation that $P( \lambda_{n-1,(n-1)}> pr)\sim 1$ (Theorem 3, Jiang and Li (2014)).
Together with approximation \eqref{logA}
 and inequality that
$\int_{p(y-\delta_n)}^{py}\lambda_n^{\frac{\beta (p-n+1)}{2}-1}e^{-\frac{\lambda_n}{2}}d\lambda_n\geq
p\delta_n [p(y-\delta_n)]^{\frac{\beta (p-n+1)}{2}-1}e^{-\frac{py}{2}}$,
 this implies that
  \begin{eqnarray}
 && \log P(\lambda_{(n)}< py) \nonumber\\
 &\geq& \log n+\log A_n+\beta(n-1)\log (pr-py)
 \notag\\
&& +[\frac{\beta (p-n+1)}{2}-1]\log (p(y-\delta_n))
 -\frac{py}{2}+ \log (p\delta_n)\nonumber\\
  &=& \log n+\log A_n+\beta(n-1)\log p
  \nonumber\\
  &&+[\frac{\beta (p-n+1)}{2}-1]\log (py)
 -\frac{py}{2}+O(n)\nonumber\\
 &=& \log n+\log A_n+[\frac{\beta (p+n-1)}{2}-1]\log (py)
 -\frac{py}{2}+O(n)\nonumber\\
 &=& p\left(\frac{\beta}{2}-\frac{\beta}{2}\log\beta-\frac{y}{2}+\frac{\beta}{2}\log y\right)
+O(n\log\frac{p}{n}), \label{Sky_cold}
\end{eqnarray}
where in the second step, we used $\log (p(y-\delta_n)) =\log (py)-O(\delta_n)$; the calculation for the last step is similar as that of inequality \eqref{gx5}.
  This gives the corresponding lower bound.

Now let us look at the upper bound of $\lambda_{(n)}$. For some big constant $M$,
 $ P(\lambda_{(n)}< py)\leq P(\lambda_{(n)}< py, \lambda_{(1)}<Mp)+P(\lambda_{(1)}>Mp)$.
 A similar argument as in \eqref{go_come} in the proof of Lemma \ref{lemma1} gives   that
 \begin{eqnarray}
&&
\log P(\lambda_{(n)}< py, \lambda_{(1)}<Mp) \nonumber\\
&\leq &
\log n+\log A_n+\log \int_0^{py} (Mp)^{\beta(n-1)}\lambda_n^{\beta \frac{p-n+1}{2}-1}e^{-\frac{\lambda_n}{2}}d\lambda_n \nonumber\\
&\leq &
 \log n -\frac{\beta p}{2}\log p -\frac{\beta p}{2}\left(\log {\beta}-1\right)
-\frac{\beta n}{2}\log n-\frac{\beta n}{2}\left(\log {\beta}-1\right) \nonumber\\
&&
+\beta(n-1)\log p-\frac{py}{2} +\frac{\beta(p-n+1)}{2}\log (py)
-\log (p(\beta-y))\notag\\
&&+O(\log p+n)\nonumber\\
&= &
p\left(\frac{\beta}{2}-\frac{\beta}{2}\log\beta-\frac{y}{2}+\frac{\beta}{2}\log y\right)+O(n\log\frac{p}{n}).
\end{eqnarray}
From \eqref{gx3} we have that
$\log P(\lambda_{(1)}>Mp)
\lesssim
p\left(\frac{\beta}{2}-\frac{\beta}{2}\log\beta-\frac{M}{2}+\frac{\beta}{2}\log M\right)$.
Then,
$$\log P(\lambda_{(n)}< py)\leq p\left(\frac{\beta}{2}-\frac{\beta}{2}\log\beta-\frac{y}{2}+\frac{\beta}{2}\log y\right)+O(n\log\frac{p}{n}).$$
 This and (\ref{Sky_cold}) yield the desired approximation for $\log P(\lambda_{(n)}< py)$.
\end{proof}

\begin{lemma}\label{lemma20}
Let $\lambda_{(1)}>\cdots>\lambda_{(n)}$ be defined as in (\ref{orderdensity}). Assume $n/p\to 0$.
 For $x>\beta$ and  $\delta_n >0$ with $\delta_n\to0$ and  $\frac{\delta_n^2}{p^{-1}n}\to\infty$ we have
$$P(\lambda_{(1)}> p(x+\delta_n))=o(1)P(\lambda_{(1)}> px).$$
\end{lemma}

\noindent\begin{proof}[Proof]
Following a similar argument as in the proof of Lemma \ref{lemma1}, we have
 \begin{eqnarray*}
&&\log P(\lambda_{(1)}> p(x+\delta_n))\\
&\leq & \log n+\log A_n+\log \int_{p(x+\delta_n)}^\infty \lambda_1^{\beta \frac{p+n-1}{2}-1}e^{-\frac{\lambda_1}{2}}d\lambda_1\\
&\leq & \log n -\frac{\beta p}{2}\log p -\frac{\beta p}{2}\left(\log{\beta}-1\right)
-\frac{\beta n}{2}\log n-\frac{\beta n}{2}\left(\log {\beta}-1\right)+O(\log p)\\
&&-\log [p(x+\delta_n)-p\beta)]+\frac{\beta(p+n-1)}{2}(\log p+\log (x+\delta_n))-\frac{px +p\delta_n}{2} +O(1)
\\
&=&p\left(\frac{\beta}{2}-\frac{\beta}{2}\log\beta-\frac{x}{2}+\frac{\beta}{2}\log x\right)
+O(n\log\frac{p}{n})-(1+o(1))\frac{x-\beta}{2x}p\delta_n,
\end{eqnarray*}
where we used the approximations that $\log [p(x+\delta_n)-p\beta)]=\log (px-p\beta)+O(\delta_n)$
and $\log(x+\delta_n)=\log x+(1+o(1)){\delta_n}/{x}.$
By Lemma \ref{lemma2},
\begin{eqnarray*}
& & \log P(\lambda_{(1)}> p(x+\delta_n))-\log P(\lambda_{(1)}> px)\\
& = &  O(n\log\frac{p}{n})-(1+o(1))\frac{x-\beta}{2x}p\delta_n=-(1+o(1))\frac{x-\beta}{2x}p\delta_n
\end{eqnarray*}
since  $n\log\frac{p}{n}=o(p\delta_n)$ from the given condition. The proof is then complete from the fact $p\delta_n\to +\infty$.
\end{proof}

\begin{lemma}\label{lemma0}
Let $\lambda_{(1)}>\cdots>\lambda_{(n)}$ be defined as in (\ref{orderdensity}) and $x>\beta.$ Assume $p/n\rightarrow\infty$. For $\delta_n>0$ such that $\delta_n\to 0$ and $\frac{\delta_n^2}{np^{-1}}\to\infty$ as $n\to\infty$,
\begin{eqnarray}
&& P(\lambda_{(1)}>px, \lambda_{(2)}> p(\beta+\delta_n)) =o(1)P(\lambda_{(1)}>px);\label{song1}\\
&& P(\lambda_{(1)}>px, \lambda_{(n)}< p(\beta-\delta_n)) =o(1)P(\lambda_{(1)}>px).\label{song2}
\end{eqnarray}
\end{lemma}

\noindent\begin{proof}[Proof] We prove this lemma in three steps.

{\it Step 1}. We first show that (\ref{song1}) and (\ref{song2}) hold under a less restrictive condition on $\delta_n$, that is, $\delta_n>0$ such that $\delta_n\to 0$ and
$\frac{(\delta_n)^2}{np^{-1}\log (n^{-1}p)}\to\infty$ as $n\to\infty$.
In fact, recalling $A_n$ in (\ref{An}), we have
\begin{eqnarray*}
&&
P(\lambda_{(1)}>px, \lambda_{(2)}> p(\beta+\delta_n))\\
&= &
 \int_{\lambda_1>\cdots>\lambda_n,\atop~ \lambda_{1}>px,~ \lambda_{2}> p(\beta+\delta_n)}
  n A_n g_{n-1,p-1,\beta}(\lambda_2,\cdots,\lambda_n)
  \\
  &&\times
  \prod_{i=2}^n(\lambda_1-\lambda_i)^\beta
\cdot \lambda_1^{\frac{\beta(p-n+1)}{2}-1}\cdot
e^{-\frac{1}{2} \lambda_1}d\lambda_1\cdots d\lambda_n \\
&\leq & nA_n \int_{px}^\infty \lambda_1^{\beta \frac{p+n-1}{2}-1}e^{-\frac{\lambda_1}{2}}d\lambda_1
\times P( \lambda_{n-1,(1)}> p(\beta+\delta_n)),
\end{eqnarray*}
where, as seen before, $ \lambda_{n-1,(1)}$ is equal to  $\lambda_{(1)}$ with $\lambda_{(1)}>\cdots > \lambda_{(n-1)}$ having the density $g_{n-1, p-1, \beta}(\lambda_{(1)}, \cdots, \lambda_{(n-1)})$. From  Lemma \ref{lemma1}, we know $\log P( \lambda_{n-1,(1)}> p(\beta+\delta_n))\lesssim -\frac{p\delta_n^2}{4\beta} $. It follows that
\begin{eqnarray}\label{gx9}
&&
\log P(\lambda_{(1)}>px, \lambda_{(2)}> p(\beta+\delta_n))\notag\\
&\leq &
\log n +\log A_n +\log \int_{px}^\infty \lambda_1^{\beta \frac{p+n-1}{2}-1}e^{-\frac{\lambda_1}{2}}d\lambda_1 -(1+o(1))\frac{p\delta_n^2}{4\beta}  \notag\\
&=&
p\left(\frac{\beta}{2}-\frac{\beta}{2}\log\beta-\frac{x}{2}
+\frac{\beta}{2}\log x\right)+O(n\log\frac{p}{n})
-(1+o(1))\frac{p\delta_n^2}{4\beta},
\end{eqnarray}
where the second step follows exactly from the derivation of \eqref{gx3}. By Lemma \ref{lemma2},
\begin{eqnarray}\label{gx8}
& &  \log P(\lambda_{(1)}>px, \lambda_{(2)}> p(\beta+\delta_n)) - \log P(\lambda_{(1)}>px)\notag\\
 & \leq & O(n\log\frac{p}{n})
-(1+o(1))\frac{p\delta_n^2}{4\beta} \to -\infty
\end{eqnarray}
since $n\log\frac{p}{n}=o(p\delta_n^2)$ under the assumption that $\delta_n\to 0$ and
$\frac{\delta_n^2}{np^{-1}\log (n^{-1}p)}\to\infty$. The assertion (\ref{song1}) follows.

A similar argument gives that
\begin{eqnarray*}
&&
P(\lambda_{(1)}>px, \lambda_{(n)}< p(\beta-\delta_n))\\
&= &
 \int_{\lambda_1>\cdots>\lambda_n,\atop~ \lambda_{1}>px,~  \lambda_{(n)}< p(\beta-\delta_n)}
  n A_n g_{n-1,p-1,\beta}(\lambda_2,\cdots,\lambda_n)
  \\
  &&\times
   \prod_{i=2}^n(\lambda_1-\lambda_i)^\beta
\cdot \lambda_1^{\frac{\beta(p-n+1)}{2}-1}\cdot
e^{-\frac{1}{2} \lambda_1}d\lambda_1\cdots d\lambda_n \\
&\leq & nA_n \int_{px}^\infty \lambda_1^{\beta \frac{p+n-1}{2}-1}e^{-\frac{\lambda_1}{2}}d\lambda_1
\times P( \lambda_{n-1,(n-1)}< p(\beta-\delta_n)),
\end{eqnarray*}
where, $\lambda_{n-1,(n-1)}$ denotes  $\lambda_{(n-1)}$ in the order statistics $\lambda_{(1)}>\cdots > \lambda_{(n-1)}$ having the density $g_{n-1, p-1, \beta}(\lambda_{(1)}, \cdots, \lambda_{(n-1)})$. By  Lemma \ref{lemma1},  $\log P( \lambda_{n-1,(n-1)}< p(\beta-\delta_n))\lesssim -\frac{p\delta_n^2}{4\beta} $. Then, from \eqref{gx9},
\begin{eqnarray*}
&&
\log P(\lambda_{(1)}>px,  \lambda_{(n)}< p(\beta-\delta_n))\\
&\leq &
\log n +\log A_n +\log \int_{px}^\infty \lambda_1^{\beta \frac{p+n-1}{2}-1}e^{-\frac{\lambda_1}{2}}d\lambda_1 -(1+o(1))\frac{p\delta_n^2}{4\beta}  \\
&=&
p\left(\frac{\beta}{2}-\frac{\beta}{2}\log\beta-\frac{x}{2}
+\frac{\beta}{2}\log x\right)+O(n\log\frac{p}{n})
-(1+o(1))\frac{p\delta_n^2}{4\beta}.
\end{eqnarray*}
Using Lemma \ref{lemma2} in a similar way to \eqref{gx8}, we obtain
\begin{eqnarray*}
& &   \log P(\lambda_{(1)}>px, \lambda_{(n)}< p(\beta-\delta_n)) - \log P(\lambda_{(1)}>px)\\
& \leq &  O(n\log\frac{p}{n})
-(1+o(1))\frac{p\delta_n^2}{4\beta} \to -\infty.
\end{eqnarray*}
The statement (\ref{song2}) is concluded.

{\it Step 2}. We prove (\ref{song1})  under the given condition on $\delta_n>0$, that is, $\delta_n\to 0$ and
$\frac{(\delta_n)^2}{np^{-1}}\to\infty$ as $n\to\infty$. Set $\delta_n'=\max\{2\delta_n,\sqrt{np^{-1}}\log (n^{-1}p)\}$. Then
\begin{eqnarray}\label{trumpet}
2\delta_n < \delta'_n \to 0\ \  \ \mbox{and}\ \ \ \frac{\delta_n'^2}{np^{-1}\log (n^{-1}p)}\to\infty.
\end{eqnarray}
Based on the result in {\it Step} 1  and Lemma \ref{lemma20},
we only need to show that
\begin{eqnarray}\label{noodle}
&& P(px<\lambda_{(1)}<p(x+\delta_n), p(\beta+\delta_n)<\lambda_{(2)}< p(\beta+\delta'_n), \lambda_{(n)}>p(\beta-\delta_n')) \nonumber\\
&&\quad\quad= o(1)P(\lambda_{(1)}> px).
\end{eqnarray}
Let $\beta_n=\beta+\delta_n$. We have
\begin{eqnarray*}
 &&
 P(px<\lambda_{(1)}<p(x+\delta_n), p\beta_n<\lambda_{(2)}< p(\beta+\delta'_n), \lambda_{(n)}>p(\beta-\delta'_n))\\
&=&
\int_{\lambda_1>\cdots>\lambda_n,~ px<\lambda_{1}<p(x+\delta_n),\atop
~p\beta_n< \lambda_{2}< p(\beta+\delta'_n),  \lambda_{n}> p(\beta-\delta'_n)} n! f_{n,p,\beta}(\lambda_1,\cdots,\lambda_n) d\lambda_1\cdots d\lambda_n \\
&=& \int_{\lambda_1>\cdots>\lambda_n,~ px<\lambda_{1}<p(x+\delta_n),\atop
~p\beta_n< \lambda_{2}< p(\beta+\delta'_n),  \lambda_{n}> p(\beta-\delta'_n)}
 n A_n \prod_{i=2}^n(\lambda_1-\lambda_i)^\beta
\cdot \lambda_1^{\frac{\beta(p-n+1)}{2}-1}
e^{-\frac{1}{2} \lambda_1}\\
&&\quad\times g_{n-1,p-1,\beta}(\lambda_2,\cdots,\lambda_n) d\lambda_1\cdots d\lambda_n.
\end{eqnarray*}
Note that
$$(\lambda_1-\lambda_i) = (px-p\beta)\Big\{1+\frac{\lambda_1-px}{px-p\beta}-\frac{\lambda_i-p\beta}{px-p\beta}\Big\}.$$
By taking $z$ in Lemma \ref{lemmabound} as $\frac{\lambda_1-px}{px-p\beta}-\frac{\lambda_i-p\beta}{px-p\beta}$ and
 $0<\alpha_2<1/2$, we have the inequality as shown in \eqref{gx6}.
Then we have
$ \prod_{i=2}^n\left(\lambda_1-\lambda_i\right)^\beta
\leq
 (px-p\beta)^{\beta(n-1)}
\exp\big\{\beta(n+o(n))\frac{\lambda_1-px}{px-p\beta}
-\beta\sum_{i=2}^n\frac{\lambda_i-p\beta}{px-p\beta}
-\beta\alpha_2\sum_{i=2}^n(\frac{\lambda_i-p\beta}{px-p\beta})^2
\big\}.
$
 This gives the following upper bound:
  \begin{eqnarray}\label{july1}
&&P(px<\lambda_{(1)}<p(x+\delta_n), p\beta_n<\lambda_{(2)}< p(\beta+\delta'_n), \lambda_{(n)}>p(\beta-\delta'_n))
\notag\\
&\leq&
nA_n (px-p\beta )^{\beta(n-1)}
\int_{px}^{p(x+\delta_n)}
e^{(\beta n+o(n))\frac{\lambda_1-px}{px-p\beta}}
\lambda_1^{\frac{\beta (p-n+1)}{2}-1}e^{-\frac{\lambda_1}{2}}d\lambda_1
\notag\\
&&
\times \int_{\lambda_2>\cdots>\lambda_n\atop
p\beta_n<\lambda_{2}< p(\beta+\delta'_n),  \lambda_{n}> p(\beta-\delta'_n)}
e^{-\beta\sum_{i=2}^n\frac{\lambda_i-p\beta}{p(x-\beta)}
-\beta\alpha_2\sum_{i=2}^n\left(\frac{\lambda_i-p\beta}{px-p\beta}\right)^2}
\notag\\
&&\times g_{n-1,p-1,\beta}(\lambda_2,\cdots,\lambda_n)d\lambda_2\cdots d\lambda_n
\notag\\
&\leq &
nA_n (px-p\beta)^{\beta(n-1)}
\int_{px}^{p(x+\delta_n)}
e^{(\beta n+o(n))\frac{\lambda_1-px}{px-p\beta}}
\lambda_1^{\frac{\beta (p-n+1)}{2}-1}e^{-\frac{\lambda_1}{2}}d\lambda_1
\notag\\
&&\times
\int_{\lambda_2>\cdots>\lambda_n\atop
p\beta_n<\lambda_{2}< p(\beta+\delta'_n),  \lambda_{n}> p(\beta-\delta'_n)}
e^{-\beta\sum_{i=3}^n\frac{\lambda_i-p\beta}{p(x-\beta)}
-\beta\alpha_2\sum_{i=3}^n\left(\frac{\lambda_i-p\beta}{px-p\beta}\right)^2}
\notag \\
&&\quad
\times
(n-1)A_{n-1} \prod_{i=3}^n (\lambda_2-\lambda_i)^\beta
\lambda_2^{\frac{\beta(p-n+1)}{2}-1}
e^{-\frac{\lambda_2}{2}}\notag\\
&&\quad\times
 g_{n-2,p-2,\beta}(\lambda_3,\cdots,\lambda_n)d\lambda_2\cdots d\lambda_n,
\end{eqnarray}
where $A_{n-1}=\frac{c_{n-1,p-1,\beta}}{c_{n-2,p-2,\beta}}$ and  in the last step we used
$e^{-\beta\frac{\lambda_2-p\beta}{p(x-\beta)}
-\beta\alpha_2\left(\frac{\lambda_2-p\beta}{px-p\beta}\right)^2}\leq 1$ since $\lambda_2>p\beta$.
Note that for $\lambda_i<\lambda_2$, $i=3,\cdots,n$, we have
$$
(\lambda_2-\lambda_i)
=
(p\beta_n-p\beta)
\left(1+\frac{\lambda_2-p\beta_n}{p\beta_n-p\beta}-\frac{\lambda_i-p\beta}{p\beta_n-p\beta}\right)
\leq
(p\delta_n) e^{\frac{\lambda_2-p\beta_n}{p\delta_n}-\frac{\lambda_i-p\beta}{p\delta_n}}.
$$
which implies that
$\prod_{i=3}^n (\lambda_2-\lambda_i)^\beta
\leq
(p\delta_n)^{\beta(n-2)} \exp\{\beta(n-2)\frac{\lambda_2-p\beta_n}{p\delta_n}
-\beta\sum_{i=3}^n\frac{\lambda_i-p\beta}{p\delta_n}\}$. Therefore, we have
\begin{eqnarray}
&&P(px<\lambda_{(1)}<p(x+\delta_n), p\beta_n<\lambda_{(2)}< p(\beta+\delta'_n), \lambda_{(n)}>p(\beta-\delta'_n))\notag\\
&\quad\leq&
nA_n (px-p\beta)^{\beta(n-1)}
\int_{px}^{p(x+\delta_n)}
\lambda_1^{\frac{\beta (p-n+1)}{2}-1}
e^{(\beta n+o(n))\frac{\lambda_1-px}{px-p\beta}-\frac{\lambda_1}{2}}d\lambda_1
\label{int1}\\
&&
\times
(n-1)A_{n-1} (p\delta_n)^{\beta(n-2)} \int_{p\beta_n}^{p(\beta+\delta'_n)}
\lambda_2^{\frac{\beta(p-n+1)}{2}-1}
e^{\beta(n-2)\frac{\lambda_2-p\beta_n}{p\delta_n} -\frac{\lambda_2}{2}} d\lambda_2
\label{int2}\\
&&
\times
\int_{\lambda_3>\cdots>\lambda_n\atop
\lambda_{3}< p(\beta+\delta'_n),  \lambda_{n}> p(\beta-\delta'_n)}
e^{-(1+\frac{x-\beta}{\delta_n})\beta\sum_{i=3}^n\frac{\lambda_i-p\beta}{p(x-\beta)}
-\beta\alpha_2\sum_{i=3}^n\left(\frac{\lambda_i-p\beta}{px-p\beta}\right)^2}
\label{int3}\\
&&
\quad\times g_{n-2,p-2,\beta}(\lambda_3,\cdots,\lambda_n)d\lambda_3\cdots d\lambda_n.\notag
\end{eqnarray}
Now we analyze the three terms \eqref{int1}, \eqref{int1} and \eqref{int3} one by one.

{\it The estimate of \eqref{int1}}. Note that
\begin{align}
 \eqref{int1}
=&~
nA_n (px-p\beta)^{\beta(n-1)}
\int_{0}^{p\delta_n}
(\lambda_1+px)^{ \frac{\beta(p-n+1)}{2}-1}
e^{(\beta n+o(n))\frac{\lambda_1}{px-p\beta}-\frac{\lambda_1+px}{2}}d\lambda_1\notag\\
\leq&~
nA_n (px-p\beta)^{\beta(n-1)}  e^{-\frac{px}{2}} (px)^{\frac{\beta(p-n+1)}{2}-1}
\notag\\
& \int_{0}^{p\delta_n}
e^{\{\frac{\beta(p-n+1)}{2}-1\}\frac{\lambda_1}{px}+
(\beta n+o(n))\frac{\lambda_1}{px-p\beta} -\frac{\lambda_1}{2}}
d\lambda_1 \notag\\
\sim&~ \frac{2x}{x-\beta}nA_n (px-p\beta)^{\beta(n-1)} (px)^{\frac{\beta(p-n+1)}{2}-1} e^{-\frac{px}{2}},\notag\\&
 \label{flower_rose}
\end{align}
where in the first step we changed the variable $\lambda_1$ to $\lambda_1+px$ for the integral; in the second step we used $(\lambda_1+px)\leq(px)\exp\{\lambda_1/px\}$; and the last step follows from the fact that $p\delta_n\to\infty$, $\{\frac{\beta(p-n+1)}{2}-1\}\frac{\lambda_1}{px}+
(\beta n+o(n))\frac{\lambda_1}{px-p\beta} -\frac{\lambda_1}{2}\sim \frac{\beta-x}{2x}\lambda_1$.
Following the first inequality in \eqref{gx5}, where the quantity on the right-hand-side is log of $\eqref{flower_rose}$ up to a $o(n)$ term, we have
\begin{eqnarray}\label{estimate1}
\eqref{int1}\leq P(\lambda_{(1)}>px)e^{o(n)}.
\end{eqnarray}

{\it The estimate of \eqref{int2}}. Observe that the term \eqref{int2} is equal to (using change of variable from $\lambda_2$ to $\lambda_2+p\beta_n$ for the integral):
\begin{align}\label{eq:2int}
&
(n-1)A_{n-1} (p\delta_n)^{\beta(n-2)}
\int_{0}^{p(\delta'_n-\delta_n)}
(\lambda_2+p\beta_n)^{\frac{\beta(p-n+1)}{2}-1}
e^{\beta(n-2)\frac{\lambda_2}{p\delta_n} -\frac{\lambda_2+p\beta_n}{2}} d\lambda_2\notag\\
\leq&
(n-1)A_{n-1} (p\delta_n)^{\beta(n-2)}
(p\beta_n)^{\frac{\beta(p-n+1)}{2}-1}\notag\\
&\times
\int_{0}^{p(\delta'_n-\delta_n)}
e^{(\frac{\beta(p-n+1)}{2}-1)\frac{\lambda_2}{p\beta_n}
+\beta(n-2)\frac{\lambda_2}{p\delta_n} -\frac{\lambda_2+p\beta_n}{2}} d\lambda_2\notag\\
\lesssim &
(n-1)A_{n-1} (p\delta_n)^{\beta(n-2)}
 (p\beta_n)^{\frac{\beta(p-n+1)}{2}-1}e^{-\frac{p\beta_n}{2}}
\int_{0}^{p(\delta'_n-\delta_n)}
e^{-\frac{\delta_n\lambda_2}{2\beta_n}
+\beta(n-2)\frac{\lambda_2}{p\delta_n}} d\lambda_2,\notag\\&
\end{align}
where in the first inequality, we used
$$(\lambda_2+p\beta_n)^{\frac{\beta(p-n+1)}{2}-1}
\leq
 (p\beta_n)^{\frac{\beta(p-n+1)}{2}-1}e^{(\frac{\beta(p-n+1)}{2}-1)\frac{\lambda_2}{p\beta_n}};$$
 in the second inequality, we used
 $$
 (\frac{\beta(p-n+1)}{2}-1)\frac{\lambda_2}{p\beta_n}-\frac{\lambda_2}{2}
 \lesssim -\frac{\lambda_2\delta_n}{2\beta_n} .$$
Since $\delta^2_n/(np^{-1})\to\infty$ by assumption, we see
$\beta(n-2)\frac{1}{p\delta_n}=o(\delta_n)$ and we have
the integral term  in \eqref{eq:2int} is $ \Theta({\delta_n}^{-1})\leq \sqrt{p/n}$ for big $p$.
Note that $A_{n-1}=\frac{c_{n-1,p-1,\beta}}{c_{n-2,p-2,\beta}}$. From  approximation \eqref{logA}, we have
\begin{align}\label{logA2}
\log A_{n-1} \sim  -\frac{\beta (p-1)}{2}\log p -\frac{\beta p}{2}\left(\log {\beta}-1\right)
-\frac{\beta n}{2}\log n
+\frac{1}{2}\log \frac{\beta p}{2}+O(n).\notag\\
\end{align}
Using \eqref{logA2}, a similar derivation as in \eqref{gx0} gives that
\begin{align}
& \log \mbox{ of  }\eqref{eq:2int}
\notag\\
\lesssim&~
 \log(n-1)+\log A_{n-1}+\beta(n-2)\log (p\delta_n)\notag\\
&~ + [\frac{\beta(p-n+1)}{2}-1]\log(p\beta_n)
-\frac{p\beta_n}{2}+\frac{1}{2}\log\frac{p}{n}
\notag\\
\sim&~ -\frac{\beta (p-1)}{2}\log p -\frac{\beta p}{2}\left(\log {\beta}-1\right)
-\frac{\beta n}{2}\log n
+\frac{1}{2}\log \frac{\beta p}{2}+O(n)
\notag\\
&~ +\beta(n-2)\log (p\delta_n)+ [\frac{\beta(p-n+1)}{2}-1]\log(p\beta_n)
-\frac{p\beta_n}{2}+\frac{1}{2}\log\frac{p}{n}
\notag\\
=&~-(1+o(1))\frac{p\delta_n^2}{4 \beta}
- \frac{\beta n}{2} \log(pn) +\beta n\log(p\delta_n) + O(n)-\Theta(\log p)
\notag\\
=&~
-(1+o(1))\frac{p\delta_n^2}{4 \beta}
 +\beta n\log(\sqrt{\frac{p}{n}}\delta_n) + O(n)-\Theta(\log p),\notag\\&
 \label{jsm1}
\end{align}
where in the third step we used $\log(\beta_n)=\log\beta+\delta_n/\beta-(1+o(1))\delta_n^2/(2\beta^2).$
The dominating term in the above display is $-(1+o(1))\frac{p\delta_n^2}{4 \beta} -\Theta(\log p)$ since $n=o(\frac{p\delta_n^2}{4 \beta})$
and $\beta n\log(\sqrt{\frac{p}{n}}\delta_n)=o(\frac{p\delta_n^2}{4 \beta})$, which follows from
$\frac{1}{2}\log(\frac{p}{n}\delta_n^2)=o({\frac{p}{n}}\delta^2_n)$ given ${\frac{p}{n}}\delta^2_n\to\infty$.
This gives
\begin{eqnarray}\label{ballpen}
\log \mbox{ of }\eqref{int2}= -(1+o(1))\frac{p\delta_n^2}{{4} \beta} -\Theta(\log p).
\end{eqnarray}

{\it The estimate of \eqref{int3}}. We have
\begin{eqnarray*}
\eqref{int3}
&\leq&
 E[e^{-(1+\frac{x-\beta}{\delta_n})\beta\sum_{i=3}^n\frac{\lambda_i-p\beta}{p(x-\beta)}
-\beta\alpha_2\sum_{i=3}^n\left(\frac{\lambda_i-p\beta}{px-p\beta}\right)^2}]\\
 &\leq&
 E[e^{-(1+\frac{x-\beta}{\delta_n})\beta\sum_{i=3}^n\frac{\lambda_i-p\beta}{p(x-\beta)}}].
\end{eqnarray*}
From \eqref{variance_changchun} in Lemma \ref{lemma1'}, $\sum_{i=3}^n \lambda_i$ in the above expectation follows distribution $\chi_{\beta (p-2) (n-2)}^2$ (note that here $\lambda_i$'s have density $g_{n-2,p-2,\beta}(\lambda_3,\cdots,\lambda_n)$). Since $Ee^{t\chi_{k}^2}=(1-2t)^{-k/2}$ for $t<\frac{1}{2}$,
 we have
\begin{eqnarray}\label{coffee_sofa}
(\ref{int3}) & \leq &
\Big(1+\frac{2\beta}{p(x-\beta)}[1+\frac{x-\beta}{\delta_n}]\Big)^{-\beta (p-2) (n-2)/2}\times e^{\frac{(n-2)\beta^2}{\delta_n}(1+\frac{\delta_n}{x-\beta})} \nonumber\\
& = & \exp\Big\{-\beta \frac{(p-2) (n-2)}{2}\Big(\frac{2\beta}{p(x-\beta)}[1+\frac{x-\beta}{\delta_n}]+O\big(\frac{1}{p^2\delta_n^2}\big)\Big)
\notag\\
&&+\frac{(n-2)\beta^2}{\delta_n}
(1+\frac{\delta_n}{x-\beta})\Big\} \nonumber\\
&= & \exp\big\{O(\frac{n}{p\delta_n}) +O(\frac{n}{p\delta_n^2}) )\big\} =O(1)
\end{eqnarray}
where in the last step we used the fact $\frac{2\beta}{p(x-\beta)}[1+\frac{x-\beta}{\delta_n}]=O(\frac{1}{p\delta_n})$ and $ (1+\epsilon)^{a}=a(\epsilon +o(\epsilon^2))$ as $\epsilon \to 0$ for any $a\in \mathbb{R}.$
Combing(\ref{estimate1}), (\ref{ballpen}) with (\ref{coffee_sofa}), we conclude that
\begin{eqnarray}\label{gx10}
&& P(\lambda_{(1)}> px, p(\beta+\delta_n)<\lambda_{(2)}< p(\beta+\delta'_n), \lambda_{(n)}>p(\beta-\delta'_n))
\notag\\
&\lesssim& P(\lambda_{(1)}> px)\times e^{-(1+o(1))\frac{p\delta_n^2}{{4} \beta} -\Theta(\log p)+o(n)}
\notag\\
&=&o(1)P(\lambda_{(1)}> px),
\end{eqnarray}
where the last step follows from $-(1+o(1))\frac{p\delta_n^2}{{4} \beta} -\Theta(\log p)+o(n)\to-\infty$ since  $\delta_n^2n^{-1}p\to\infty$ and $\Theta(\log p)>0$. This completes the proof of (\ref{song1}).

\medskip
{\it Step 3}. We prove (\ref{song2})  under the given condition on $\delta_n$ with  $\delta_n\to 0$ and $\frac{\delta_n^2}{np^{-1}}\to\infty$. Similar to {\it Step 2}, set $\delta_n'=\max\{2\delta_n,\sqrt{np^{-1}}\log (n^{-1}p)\}$.
Following the result in {\it Step 1}  and Lemma \ref{lemma20},
it suffices to show that
\begin{eqnarray*}
& & P(px<\lambda_{(1)}<p(x+\delta_n), \lambda_{(2)}<p(\beta+\delta'_n), p(\beta-\delta'_n)<\lambda_{(n)}< p(\beta-\delta_n))\\
&= & o(1)P(\lambda_{(1)}>px).
\end{eqnarray*}
Similar to the derivation of \eqref{july1}, using inequality \eqref{gx6}, we have for $0<\alpha_2<1/2$
 \begin{eqnarray*}
&&P(px<\lambda_{(1)}<p(x+\delta_n), \lambda_{(2)}<p(\beta+\delta'_n), p(\beta-\delta'_n)<\lambda_{(n)}< p(\beta-\delta_n))\\
&=&
\int_{\lambda_1>\cdots>\lambda_n,~ px<\lambda_{1}<p(x+\delta_n),\atop
~ \lambda_{2}< p(\beta+\delta'_n), p(\beta-\delta'_n)<\lambda_{n}<p(\beta-\delta_n)}
 n A_n \prod_{i=2}^n(\lambda_1-\lambda_i)^\beta
\cdot \lambda_1^{\frac{\beta(p-n+1)}{2}-1}
e^{-\frac{1}{2} \lambda_1}\\
&&\quad\times g_{n-1,p-1,\beta}(\lambda_2,\cdots,\lambda_n) d\lambda_1\cdots d\lambda_n\\
&\leq&
 nA_n (px-p\beta)^{\beta(n-1)}
\int_{px}^{p(x+\delta_n)}
e^{(\beta n+o(n))\frac{\lambda_1-px}{px-p\beta} }
\lambda_1^{\frac{\beta (p-n+1)}{2}-1}e^{-\frac{\lambda_1}{2}}d\lambda_1\\
&&\times
\int_{\lambda_2>\cdots>\lambda_n, \lambda_{2}<p(\beta+\delta'_n)\atop
p(\beta-\delta'_n)<\lambda_{n}< p(\beta-\delta_n)}
e^{-\beta\sum_{i=2}^n\frac{\lambda_i-p\beta}{p(x-\beta)}
-\beta\alpha_2\sum_{i=2}^n\left(\frac{\lambda_i-p\beta}{px-p\beta}\right)^2} \\
&&\quad\times
g_{n-1,p-1,\beta}(\lambda_2,\cdots,\lambda_n)d\lambda_2\cdots d\lambda_n\\
&\lesssim &
nA_n (px-p\beta)^{\beta(n-1)}
\int_{px}^{p(x+\delta_n)}
\lambda_1^{\frac{\beta (p-n+1)}{2}-1}
e^{(\beta n+o(n))\frac{\lambda_1-px}{px-p\beta} -\frac{\lambda_1}{2}}d\lambda_1\\
&&\times
\int_{\lambda_2>\cdots>\lambda_n, \lambda_{2}<p(\beta+\delta'_n)\atop
p(\beta-\delta'_n)<\lambda_{n}< p(\beta-\delta_n)}
 e^{-\beta\sum_{i=2}^{n-1}\frac{\lambda_i-p\beta}{p(x-\beta)}
-\beta\alpha_2\sum_{i=2}^{n-1}\left(\frac{\lambda_i-p\beta}{px-p\beta}\right)^2} \\
&&\quad \times
(n-1)A_{n-1} \prod_{i=2}^{n-1} (\lambda_i-\lambda_n)^\beta
\lambda_n^{\frac{\beta(p-n+1)}{2}-1}
e^{-\frac{\lambda_n}{2}} \\
&&\quad\times
g_{n-2,p-2,\beta}(\lambda_2,\cdots,\lambda_{n-1})d\lambda_2\cdots d\lambda_n,
\end{eqnarray*}
where the fact $\frac{\lambda_n-p\beta}{p(x-\beta)}=o(1)$ under the constraint in the integral is applied in the last step.
Note that for $\lambda_i>\lambda_n$, $i=2,\cdots,n-1$, we have the upper bound
$$(\lambda_i-\lambda_n)= (p\delta_n)
\left(1+\frac{\lambda_i-p\beta}{p\delta_n}-\frac{\lambda_n-p(\beta-\delta_n)}{p\delta_n}\right)
\leq  (p\delta_n) e^{\frac{\lambda_i-p\beta}{p\delta_n}-\frac{\lambda_n-p(\beta-\delta_n)}{p\delta_n}}.$$
Therefore, similar to the derivation for \eqref{int1}-\eqref{int3}, we have
\begin{align}
&P(px<\lambda_{(1)}<p(x+\delta_n), \lambda_{(2)}<p(\beta+\delta'_n), p(\beta-\delta'_n)<\lambda_{(n)}< p(\beta-\delta_n)) \nonumber\\
 &\lesssim~
nA_n (px-p\beta)^{\beta(n-1)}
\int_{px}^{p(x+\delta_n)}
\lambda_1^{\frac{\beta (p-n+1)}{2}-1}
e^{(\beta n+o(n))\frac{\lambda_1-px}{px-p\beta} -\frac{\lambda_1}{2}}d\lambda_1 \notag\\
\label{desk_1}&\\
&
\times
(n-1)A_{n-1}  (p\delta_n)^{\beta(n-2)}
\int_{p(\beta-\delta'_n)}^{p(\beta-\delta_n)}
\lambda_n^{\frac{\beta(p-n+1)}{2}-1}
e^{-\beta(n-2)\frac{\lambda_n-p(\beta-\delta_n)}{p\delta_n}-\frac{\lambda_n}{2}}
d\lambda_n \notag\\
\label{desk_2}&\\
&
\times
\int_{\lambda_2>\cdots>\lambda_n,\atop \lambda_{2}<p(x+\delta'_n)}
e^{(-1+\frac{x-\beta}{\delta_n})\beta\sum_{i=2}^{n-1}\frac{\lambda_i-p\beta}{p(x-\beta)}
-\beta\alpha_2\sum_{i=2}^{n-1}\left(\frac{\lambda_i-p\beta}{px-p\beta}\right)^2}
\notag\\
&\quad\quad\quad\quad\quad\quad\times g_{n-2,p-2,\beta}(\lambda_2,\cdots,\lambda_{n-1})d\lambda_2\cdots d\lambda_{n-1} \nonumber\\
& \label{desk_3}
\end{align}
From (\ref{int1}) and (\ref{estimate1}), we get
\begin{eqnarray}\label{hot_day}
(\ref{desk_1}) \leq  P(\lambda_{(1)}> px)e^{o(n)}.
\end{eqnarray}
For \eqref{desk_2}, review \eqref{eq:2int}, a change of variable from $\lambda_n$ to $p(\beta-\delta_n)-\lambda_n$ gives that
\begin{eqnarray*}
 \eqref{desk_2}
  &=&
 (n-1)A_{n-1}  (p\delta_n)^{\beta(n-2)}\\
 &&
\int_{0}^{p(\delta'_n-\delta_n)}
(p(\beta-\delta_n)-\lambda_n)^{\frac{\beta(p-n+1)}{2}-1}
e^{\beta(n-2)\frac{\lambda_n}{p\delta_n}-\frac{p(\beta-\delta_n)-\lambda_n}{2}}
d\lambda_n\\
&\leq&
 (n-1)A_{n-1}  (p\delta_n)^{\beta(n-2)}[p(\beta-\delta_n)]^{\frac{\beta(p-n+1)}{2}-1}\\
 &&\times
\int_{0}^{p(\delta'_n-\delta_n)}
e^{-[\frac{\beta(p-n+1)}{2}-1]\frac{\lambda_n}{p(\beta-\delta_n)}+
\beta(n-2)\frac{\lambda_n}{p\delta_n}-\frac{p(\beta-\delta_n)-\lambda_n}{2}}
d\lambda_n\\
&\leq&
 (n-1)A_{n-1}  (p\delta_n)^{\beta(n-2)}[p(\beta-\delta_n)]^{\frac{\beta(p-n+1)}{2}-1}
 e^{-\frac{p(\beta-\delta_n)}{2}}\\
 &&
\int_{0}^{p(\delta'_n-\delta_n)}
e^{- (1+o(1))\frac{\delta_n\lambda_n}{2\beta}}
d\lambda_n,
\end{eqnarray*}
where the last step follows from the facts $-[\frac{\beta(p-n+1)}{2}-1]\frac{\lambda_n}{p(\beta-\delta_n)}+\frac{\lambda_n}{2}\sim -\frac{\delta_n\lambda_n}{2\beta}(1+o(1))$ and $\frac{n}{p\delta_n}\lambda_n=o({\delta_n}\lambda_n)$. Noticing that the last integral is equal to $(1+o(1))\frac{2\beta}{\delta_n} \leq \sqrt{\frac{p}{n}}$, we have from \eqref{jsm1} and \eqref{ballpen} (regarding ``$\beta-\delta_n$" here by ``$\beta_n$" in \eqref{jsm1}) that
\begin{eqnarray}\label{jsm2}
\log \mbox{ of }\eqref{desk_2} \leq -(1+o(1))\frac{p\delta_n^2}{{4} \beta} -\Theta(\log p).
\end{eqnarray}
For \eqref{desk_3}, similar to \eqref{coffee_sofa}, we have
\begin{align}
\eqref{desk_3}
\leq&~
 E[e^{(-1+\frac{x-\beta}{\delta_n})\beta\sum_{i=3}^n\frac{\lambda_i-p\beta}{p(x-\beta)}}]
 \notag\\
=&~\Big(1+\frac{2\beta}{p(x-\beta)}[1-\frac{x-\beta}{\delta_n}]\Big)^{-\beta (p-2) (n-2)/2}\times
 e^{-\frac{(n-2)\beta^2}{\delta_n}(1-\frac{
\delta_n}{x-\beta})}\notag\\
=&~ O(1).
\label{jsm3}
\end{align}
Combing (\ref{hot_day}), (\ref{jsm2}) and (\ref{jsm3}), we conclude that
\begin{eqnarray}
&&P(px<\lambda_{(1)}<p(x+\delta_n), \lambda_{(2)}<p(\beta+\delta'_n), p(\beta-\delta'_n)<\lambda_{(n)}< p(\beta-\delta_n)) \nonumber\\
&\lesssim& P(\lambda_{(1)}> px)\times e^{-(1+o(1))\frac{p\delta_n^2}{4 \beta} -\Theta(\log p)+o(n)} \nonumber\\
&=&o(1)P(\lambda_{(1)}> px), \nonumber
\end{eqnarray}
where the last step follows from the same argument as in \eqref{gx10}.
The conclusion holds.
\end{proof}

\begin{lemma}\label{lemma4}
Consider display \eqref{lowerbound} in the proof of Theorem \ref{theorem1}, where $\lambda_2>\cdots>\lambda_n$ has density $g_{n-1,p-1,\beta}(\lambda_2,\cdots,\lambda_n)$ and $\delta_n>0$ such that $\delta_n\to 0$ and $\delta_n^2/(np^{-1})\to\infty$. We have
\begin{align}
& \int_{\lambda_2>\cdots>\lambda_n,\atop  \lambda_{2}< p(\beta+\delta_n), \lambda_{n}> p(\beta-\delta_n)}e^{-\beta\sum_{i=2}^n\frac{\lambda_i-p\beta}{p(x-\beta)}
-\beta\alpha_1\sum_{i=2}^n\frac{(\lambda_i-p\beta)^2}{p^2(x-\beta)^2}}\notag\\
&\times
g_{n-1,p-1,\beta}(\lambda_2,\cdots,\lambda_n)d\lambda_2\cdots d\lambda_n\notag\\
\sim&~ \int_{\lambda_2>\cdots>\lambda_n}e^{-\beta\sum_{i=2}^n\frac{\lambda_i-p\beta}{p(x-\beta)}
-\beta\alpha_1\sum_{i=2}^n\frac{(\lambda_i-p\beta)^2}{p^2(x-\beta)^2}}
\notag\\
&\times
g_{n-1,p-1,\beta}(\lambda_2,\cdots,\lambda_n)d\lambda_2\cdots d\lambda_n.\label{july2}
\end{align}
\end{lemma}
\begin{proof}[Proof]
The proof is similar to that of Lemma \ref{lemma0}. We first show that
\begin{equation}\label{gx7}
 \int_{\lambda_2>\cdots>\lambda_n,\atop  \lambda_{2}>p(\beta+\delta_n)}e^{-\beta\sum_{i=2}^n\frac{\lambda_i-p\beta}{p(x-\beta)}
-\beta\alpha_1\sum_{i=2}^n\frac{(\lambda_i-p\beta)^2}{p^2(x-\beta)^2}} \times
g_{n-1,p-1,\beta}(\lambda_2,\cdots,\lambda_n)d\lambda_2\cdots d\lambda_n
\end{equation}
is negligible compared with the integral in \eqref{july2}. By the same argument as in the derivation of \eqref{july1},  \eqref{int2} and \eqref{int3}, we have
\begin{align}
\eqref{gx7}
=&~
 \int_{\lambda_2>\cdots>\lambda_n\atop
\lambda_{2}>p(\beta+\delta_n)}
(n-1)A_{n-1} \prod_{i=3}^n (\lambda_2-\lambda_i)^\beta
\lambda_2^{\frac{\beta(p-n+1)}{2}-1}
e^{-\frac{\lambda_2}{2}}
\notag\\
&\quad \times
e^{-\beta\sum_{i=3}^n\frac{\lambda_i-p\beta}{p(x-\beta)}
-\beta\alpha_1\sum_{i=3}^n\left(\frac{\lambda_i-p\beta}{px-p\beta}\right)^2} g_{n-2,p-2,\beta}(\lambda_3,\cdots,\lambda_n)d\lambda_2\cdots d\lambda_n
\notag\\
\leq&~
 \int_{\lambda_2>\cdots>\lambda_n\atop
\lambda_{2}>p(\beta+\delta_n)}
(n-1)A_{n-1}
(p\delta_n)^{\beta(n-2)}
\lambda_2^{\frac{\beta(p-n+1)}{2}-1}
e^{\beta(n-2)\frac{\lambda_2-p(\beta+\delta_n)}{p\delta_n}-\frac{\lambda_2}{2}}
\notag\\
&
\quad \times
e^{-\beta\sum_{i=3}^n\frac{\lambda_i-p\beta}{p\delta_n}-\beta\sum_{i=3}^n\frac{\lambda_i-p\beta}{p(x-\beta)}
-\beta\alpha_1\sum_{i=3}^n\left(\frac{\lambda_i-p\beta}{px-p\beta}\right)^2}
\notag\\
&\times
g_{n-2,p-2,\beta}(\lambda_3,\cdots,\lambda_n)d\lambda_2\cdots d\lambda_n
\notag\\
\leq&~
(n-1)A_{n-1}
(p\delta_n)^{\beta(n-2)}
 \int_{\lambda_{2}>p(\beta+\delta_n)} \lambda_2^{\frac{\beta(p-n+1)}{2}-1}
e^{\beta(n-2)\frac{\lambda_2-p(\beta+\delta_n)}{p\delta_n}-\frac{\lambda_2}{2}} d\lambda_2
\notag\\
&
\times\int_{\lambda_3>\cdots>\lambda_n}
e^{-(1+\frac{x-\beta}{\delta_n})\beta\sum_{i=3}^n\frac{\lambda_i-p\beta}{p(x-\beta)}
-\beta\alpha_1\sum_{i=3}^n\left(\frac{\lambda_i-p\beta}{px-p\beta}\right)^2}
 \notag\\
&\times
g_{n-2,p-2,\beta}(\lambda_3,\cdots,\lambda_n)d\lambda_3\cdots d\lambda_n,\notag\\&
\label{gxgx}
\end{align}
where in the second step we used the upper bound
\begin{equation}
(\lambda_2-\lambda_i)= (p\delta_n)
\left(1+\frac{\lambda_2-p(\beta+\delta_n)}{p\delta_n}-\frac{\lambda_i-p\beta}{p\delta_n}\right)
\leq  (p\delta_n) e^{\frac{\lambda_2-p(\beta+\delta_n)}{p\delta_n}-\frac{\lambda_i-p\beta}{p\delta_n}}.
\label{xgj4}
\end{equation}
Note that the above two integrals from the final line of display \eqref{gxgx} take similar forms to \eqref{int2} and \eqref{int3} in {\it Step 2} of the proof of Lemma \ref{lemma0}. Then from \eqref{ballpen} and \eqref{coffee_sofa}, we have
$$\eqref{gx7}\lesssim e^{-(1+o(1))\frac{p\delta_n^2}{2 \beta} -\Theta(\log p)+O(1)}.$$
Using Lemma \ref{lemma1'}, we know the main integral
$\eqref{july2}\sim e^{O(n^2p^{-1})}$. Since under the assumption of $\delta_n$, $-(1+o(1))\frac{p\delta_n^2}{2 \beta} -\Theta(\log p)+O(n^2p^{-1})\to -\infty$, we know \eqref{gx7} is negligible compared with \eqref{july2}.
Furthermore, by the same argument as the derivation of \eqref{desk_2} and \eqref{desk_3}, we see that
\begin{align}\label{july3}
&
\int_{\lambda_2>\cdots>\lambda_n,\atop  \lambda_{n}<p(\beta-\delta_n)}e^{-\beta\sum_{i=2}^n\frac{\lambda_i-p\beta}{p(x-\beta)}
-\beta\alpha_1\sum_{i=2}^n\frac{(\lambda_i-p\beta)^2}{p^2(x-\beta)^2}}\notag\\
& \times
g_{n-1,p-1,\beta}(\lambda_2,\cdots,\lambda_n)d\lambda_2\cdots d\lambda_n
\notag\\
\lesssim&~
(n-1)A_{n-1}  (p\delta_n)^{\beta(n-2)}
\int_{\lambda_n<p(\beta-\delta_n)}
\lambda_n^{\frac{\beta(p-n+1)}{2}-1}
e^{-\beta(n-2)\frac{\lambda_n-p(\beta-\delta_n)}{p\delta_n}-\frac{\lambda_n}{2}}
d\lambda_n
\notag\\
&
\times
\int_{\lambda_2>\cdots>\lambda_n}
e^{(-1+\frac{x-\beta}{\delta_n})\beta\sum_{i=2}^{n-1}\frac{\lambda_i-p\beta}{p(x-\beta)}
-\beta\alpha_1\sum_{i=2}^{n-1}\left(\frac{\lambda_i-p\beta}{px-p\beta}\right)^2} \notag\\
&\times
g_{n-2,p-2,\beta}(\lambda_2,\cdots,\lambda_{n-1})d\lambda_2\cdots d\lambda_{n-1}
\nonumber\\
\lesssim&~
e^{-(1+o(1))\frac{p\delta_n^2}{4 \beta} -\Theta(\log p)+O(1)},\notag\\&
\end{align}
where the last step follows from the approximation results \eqref{jsm2} and \eqref{jsm3}.
This implies that \eqref{july3} is also negligible with respect to \eqref{july2}.
Combining \eqref{gx7} and \eqref{july3}, we have the desired conclusion.
\end{proof}

\begin{lemma}\label{lemma3}
Assume $p/n\to\infty.$ Then, for any $\delta_n>0$ with $\delta_n\to 0$ and $\frac{\delta_n^2}{np^{-1}}\to\infty$ as $n\to\infty$, we have
\begin{eqnarray}
E^Q\left[L_p^2;
 \lambda_{(1)}>px,\lambda_{(2)}>p(\beta+\delta_n)\right]
= o(1) P(\lambda_{(1)}>px)^2;\label{muskmellon}\\
E^Q\left[L_p^2;  \lambda_{(1)}>px,\lambda_{(n)}<p(\beta-\delta_n)\right]= o(1) P(\lambda_{(1)}>px)^2.\label{watermellon}
\end{eqnarray}
\end{lemma}

\begin{proof}[Proof] Recall the notation $L_p$ and $Q$ as Section \ref{subsubsec:Sim} (Efficient simulation method). The density of $\lambda_{(1)},\cdots,\lambda_{(n)}$ under measure $Q$ is given in \eqref{aa1}.
We prove (\ref{muskmellon}) and (\ref{watermellon}) separately.

\smallskip
{\it Proof of (\ref{muskmellon})}. We proceed  by two steps.

\smallskip
{\it Step 1}. We first  consider the case when $\lambda_{2}>px$ and show that
\begin{eqnarray*}
E^Q\left[L_p^2;
 \lambda_{(1)}>px,\lambda_{(2)}>px\right]
= o(1) P(\lambda_{(1)}>px)^2.
\end{eqnarray*}
From \eqref{Lp}, the above expectation term equals
\begin{align}
&
E^Q\biggr[\biggr(\frac{nA_n\times
\prod_{i=2}^n(\lambda_1-\lambda_i)^{\beta}
\cdot \lambda_1^{\frac{\beta(p-n+1)}{2}-1}\cdot
e^{-\frac{1}{2} \lambda_1}}
{\frac{x-\beta}{2x} e^{-\frac{x-\beta}{2x} (\lambda_1-px\vee \lambda_2)}\cdot I_{(\lambda_1>px\vee \lambda_2)}}
\biggr)^2;\biggr.
\lambda_2>px, \lambda_1>\cdots>\lambda_n\biggr]
\notag\\
&\lesssim~
 \Theta(1)n^2A_n^2 E^Q\biggr[\frac{
 \lambda_1^{\beta(p+n-1)-2}\cdot
e^{- \lambda_1}}
{e^{-(x-\beta) (\lambda_1-\lambda_2)/x}};\lambda_2>px, \lambda_1>\cdots>\lambda_n\biggr]
 \notag\\
 &\lesssim~
 \Theta(1)n^2A_n^2
 E^Q\biggr[E^Q\biggr\{\frac{\lambda_1^{\beta(p+n-1)-2}\cdot e^{- \lambda_1}}
{e^{-(x-\beta) (\lambda_1-\lambda_2)/x}};
\lambda_2>px, \lambda_1>\cdots>\lambda_n\biggr| \lambda_2,\cdots,\lambda_n\biggr\}\biggr],
\notag\\
&\label{july5}
 \end{align}
where in the first step we used
  $\prod_{i=2}^n(\lambda_1-\lambda_i)^{\beta} \lambda_1^{\frac{\beta(p-n+1)}{2}-1}\leq  \lambda_1^{\frac{\beta(p+n-1)}{2}-1}.$
Note that under the change of measure $Q$, the order statistics $\lambda_{(2)},\cdots,\lambda_{(n)}$ has the same distribution as the original measure $P$, and $\lambda_{(1)}$ under $Q$ follows exponential distribution in \eqref{july4}. Therefore, we know that the inner level expectation equals
\begin{eqnarray}
&& E^Q\left\{\frac{
 \lambda_1^{\beta(p+n-1)-2}\cdot
e^{- \lambda_1}}
{e^{-(x-\beta) (\lambda_1-\lambda_2)/x}};
\lambda_2>px, \lambda_1>\cdots>\lambda_n\biggr| \lambda_2,\cdots,\lambda_n\right\}
 \notag\\
 &=& I_{(\lambda_2>px, \lambda_2>\cdots>\lambda_n)}\int_{\lambda_2}^\infty \frac{y^{\beta(p+n-1)-2}\cdot
e^{- y}}{e^{-(x-\beta) (y-\lambda_2)/x}}\times \frac{x-\beta}{2x} e^{-\frac{x-\beta}{2x} (y-\lambda_2)}dy
\notag\\
&=&
I_{(\lambda_2>px, \lambda_2>\cdots>\lambda_n)}
\frac{(x-\beta)}{2x} e^{-\frac{(x-\beta)\lambda_2}{2x}}
\int_{\lambda_2}^\infty y^{\beta(p+n-1)-2} e^{-\frac{x+\beta}{2x}y}dy
\notag\\
&\leq &
I_{(\lambda_2>px, \lambda_2>\cdots>\lambda_n)}
(1+\Theta({n}{p}^{-1})) \lambda_2^{\beta(p+n-1)-2} e^{-\lambda_2},\notag\\
\label{xgj2}
 \end{eqnarray}
 where in the last step we used the following argument:
  for $\lambda_2>px$, we have
\begin{align}\label{xgj1}
\int_{\lambda_2}^\infty y^{\beta(p+n-1)-2} e^{-\frac{x+\beta}{2x}y}dy
&=~
 \int_{0}^\infty (y+\lambda_2)^{\beta(p+n-1)-2} e^{-\frac{x+\beta}{2x}(y+\lambda_2)}dy\notag\\
 &\leq~
 \int_{0}^\infty \lambda_2^{\beta(p+n-1)-2} e^{[\beta(p+n-1)-2]\frac{y}{\lambda_2}-\frac{x+\beta}{2x}(y+\lambda_2)}dy.
 \end{align}
 Since $[\beta(p+n-1)-2]\frac{y}{\lambda_2}-\frac{x+\beta}{2x}y\leq
 [\beta(p+n-1)-2]\frac{y}{px}-\frac{x+\beta}{2x}y
\leq [-\frac{x-\beta}{2x}+\frac{\beta n}{xp}]y,$ we have
\begin{equation}
\eqref{xgj1}
~\leq~
 \big(\frac{x-\beta}{2x}-\frac{\beta n}{xp}\big)^{-1}\lambda_2^{\beta(p+n-1)-2} e^{-\frac{x+\beta}{2x}\lambda_2},\label{xgj3}
 \end{equation}
which implies \eqref{xgj2}.
The result in \eqref{xgj2} implies that
\begin{eqnarray*}
\eqref{july5}
&\lesssim &
 \Theta(1) n^2A_n^2
E\left[
 \lambda_2^{\beta(p+n-1)-2} e^{-\lambda_2}; \lambda_2>px,\lambda_2>\cdots>\lambda_n\right],
 \end{eqnarray*}
 where $E[\cdot]$ is the expectation with respect to distribution $g_{n-1,p-1,\beta}(\lambda_2,\cdots,\lambda_n)$.
The corresponding density function of order statistics $\lambda_{(2)},\cdots,\lambda_{(n)}$ is\\ $g_{n-1,p-1,\beta}(\lambda_2,\cdots,\lambda_n)$, which is bounded above by
$$(n-1)A_{n-1}\lambda_2^{\beta (n+p-3)/2-1}e^{-\frac{\lambda_2}{2}}g_{n-2,p-2,\beta}
(\lambda_3,\cdots,\lambda_n)$$
 following \eqref{boston1} and the fact that $\lambda_{2}-\lambda_{i}<\lambda_{2}$ for $i=3,\cdots,n$.
This implies that
\begin{eqnarray*}
&&
E^Q\left[L_p^2; \lambda_{(1)}>px,\lambda_{(2)}>px\right]\\
 &\lesssim&
 \Theta(1)n^2A_n^2
(n-1)A_{n-1}\int_{px}^\infty
 \lambda_2^{\beta(p+n-1)-2} e^{-\lambda_2}
\times \lambda_2^{\beta (n+p-3)/2-1}e^{-\frac{\lambda_2}{2}}d\lambda_2\\
& & \ \ \ \ \ \ \ \ \ \ \ \ \ \ \ \ \ \ \ \ \ \ \ \ \ \ \ {\times\int_{\lambda_3>\cdots>\lambda_n}g_{n-2,p-2,\beta}(\lambda_3,\cdots,\lambda_n)\,d\lambda_3\cdots d\lambda_n}\\
 &\lesssim&
 \Theta(1)n^2A_n^2
(n-1)A_{n-1}
 (px)^{\beta(3p+3n-5)/2-3} e^{-\frac{3px}{2}}\\
  &\lesssim&
  \Theta(1)e^{-3\frac{\beta p}{2}\log p -3\frac{\beta p}{2}\left(\log {\beta}-1\right)
-3\frac{\beta n}{2}\log n
+O(n+\log p)}e^{[\beta(3p+3n-5)/2-3]\log(px)-\frac{3px}{2}}\\
&=&
 \Theta(1)e^{ 3p\left(\frac{\beta}{2}-\frac{\beta}{2}\log\beta-\frac{x}{2}+\frac{\beta}{2}\log x\right)+3\frac{\beta n}{2}\log \frac{p}{n}
+O(n+\log p)}\\
&=& o(1)P(\lambda_{(1)}>px)^2,
\end{eqnarray*}
where we used the fact that the second integral is equal to $1$ and  a similar argument to \eqref{xgj1} and \eqref{xgj3} is applied in the second step; the third step follows from {\eqref{logA}} and \eqref{logA2}, and the last step follows from the approximation result in Theorem 1 by noting that $\frac{\beta}{2}-\frac{\beta}{2}\log\beta-\frac{x}{2}+\frac{\beta}{2}\log x<0$.

\smallskip
{\it Step 2}.
Based on the result in {\it Step 1}, for the first equation, we only need to focus on the case when $\lambda_2<px$. Note that $\lambda_{(1)}$ under $Q$ follows exponential distribution in \eqref{july4}. We have from \eqref{Lp}
\begin{eqnarray}
&&E^Q\left[L_p^2;
 \lambda_{(1)}>px>\lambda_{(2)}>p(\beta+\delta_n)\right]\notag\\
&=& E^Q\left[\biggr(\frac{nA_n\times
\prod_{i=2}^n(\lambda_1-\lambda_i)^{\beta}
\cdot \lambda_1^{\frac{\beta(p-n+1)}{2}-1}\cdot
e^{-\frac{1}{2} \lambda_1}}
{\frac{x-\beta}{2x} e^{-\frac{x-\beta}{2x} (\lambda_1-px)}\cdot I_{(\lambda_1>px)}}
\biggr)^2;\right. \notag\\
&& \quad\quad\quad\quad \lambda_1>px>\lambda_2>p(\beta+\delta_n), \lambda_1>\cdots>\lambda_n\biggr]\notag\\
&=&
\Theta(1)  n^2A_n^2 \int_{ \lambda_2>\cdots>\lambda_n\atop px>\lambda_2>p(\beta+\delta_n)}
\int_{px}^\infty \frac{\prod_{i=2}^n(\lambda_1-\lambda_i)^{2\beta}
\lambda_1^{\beta(p-n+1)-2}\cdot
e^{- \lambda_1}}{e^{-(x-\beta) (\lambda_1-px)/x}} \notag\\
&&\times \frac{x-\beta}{2x} e^{-\frac{x-\beta}{2x} (\lambda_1-px)}d\lambda_1\times g_{n-1,p-1,\beta}(\lambda_2,\cdots,\lambda_n)d\lambda_2\cdots d\lambda_n\notag\\
&\lesssim&
 \Theta(1)n^2A_n^2 (px-p\beta)^{2\beta(n-1)}
e^{-\frac{x-\beta}{2x}px}
\notag\\
&& \times \int_{px}^\infty
\lambda_1^{\beta(p-n+1)-2}
e^{2\beta(n-1)\frac{\lambda_1-px}{p(x-\beta)}-\frac{x+\beta}{2x} \lambda_1}
d\lambda_1\notag\\
&&\times \int_{ \lambda_2>\cdots>\lambda_n\atop px>\lambda_2>p(\beta+\delta_n)}
e^{-2\beta\sum_{i=2}^n\frac{\lambda_i-p\beta}{p(x-\beta)}}
g_{n-1,p-1,\beta}(\lambda_2,\cdots,\lambda_n)d\lambda_2\cdots d\lambda_n,
\notag\\
\label{uppersq}
\end{eqnarray}
where in the last step we used upper bound $$(\lambda_1-\lambda_i)= (px-p\beta)
\left(1+\frac{\lambda_1-px}{p(x-\beta)}-\frac{\lambda_i-p\beta}{p(x-\beta)}\right)
\leq  (px-p\beta) e^{\frac{\lambda_1-px}{p(x-\beta)}-\frac{\lambda_i-p\beta}{p(x-\beta)}}.$$
A direct calculation as in \eqref{xgj1} and \eqref{xgj3} for the first integral in \eqref{uppersq}
 gives that
\begin{eqnarray*}
\eqref{uppersq}
&\sim&
\Theta(1)  n^2A_n^2(px-p\beta)^{2\beta(n-1)}(px)^{\beta(p-n+1)-2}e^{- px}\\
&&\times
 \int_{ \lambda_2>\cdots>\lambda_n\atop px>\lambda_2>p(\beta+\delta_n)}
e^{-2\beta\sum_{i=2}^n\frac{\lambda_i-p\beta}{p(x-\beta)}}
g_{n-1,p-1,\beta}(\lambda_3,\cdots,\lambda_n)d\lambda_2\cdots d\lambda_n\\
&\lesssim&
  \Theta(1)n^2A_n^2(px-p\beta)^{2\beta(n-1)}(px)^{\beta(p-n+1)-2}e^{- px}\\
&&\times \int_{\lambda_2>\lambda_3>\cdots>\lambda_n\atop px>\lambda_2>p(\beta+\delta_n)}
(n-1)A_{n-1} \prod_{i=3}^n (\lambda_2-\lambda_i)^\beta
\lambda_2^{\frac{\beta(p-n+1)}{2}-1}
e^{-\frac{\lambda_2}{2}-2\beta\frac{\lambda_2-p\beta}{p(x-\beta)}}\\
&&\quad \times
e^{-2\beta\sum_{i=3}^n\frac{\lambda_i-p\beta}{p(x-\beta)}}
g_{n-2,p-2,\beta}(\lambda_3,\cdots,\lambda_n)d\lambda_2\cdots d\lambda_n.
\end{eqnarray*}
Using again the upper bound \eqref{xgj4}: $(\lambda_2-\lambda_i)
\leq  p\delta_n e^{\frac{\lambda_2-p(\beta+\delta_n)}{p\delta_n}-\frac{\lambda_i-p\beta}{p\delta_n}},$
the integral term in the above display is bounded by
\begin{align}
&(n-1)A_{n-1} (p\delta_n)^{\beta(n-2)}\int^{px}_{p(\beta+\delta_n)}
\lambda_2^{\frac{\beta(p-n+1)}{2}-1}
e^{\beta(n-2)\frac{\lambda_2-p(\beta+\delta_n)}{p\delta_n}-\frac{\lambda_2}{2}-2\beta\frac{\lambda_2-p\beta}{p(x-\beta)}}d\lambda_2
\notag\\
&\times\int_{\lambda_3>\cdots>\lambda_n}
e^{-\beta\sum_{i=3}^n\frac{\lambda_i-p\beta}{p\delta_n}-2\beta\sum_{i=3}^n\frac{\lambda_i-p\beta}{p(x-\beta)}}
g_{n-2,p-2,\beta}(\lambda_3,\cdots,\lambda_n)d\lambda_3\cdots d\lambda_n
\notag\\
\leq&~
(n-1)A_{n-1} (p\delta_n)^{\beta(n-2)}e^{-\frac{p(\beta+\delta_n)}{2}}\notag\\
&~\times
\int_{0}^{p(x-\beta)}
(\lambda_2+p(\beta+\delta_n))^{\frac{\beta(p-n+1)}{2}-1}
e^{\beta(n-2)\frac{\lambda_2}{p\delta_n}-\frac{\lambda_2}{2}}d\lambda_2
\notag\\
&\times\int_{\lambda_3>\cdots>\lambda_n}
e^{-\beta\sum_{i=3}^n\frac{\lambda_i-p\beta}{p\delta_n}-2\beta\sum_{i=3}^n\frac{\lambda_i-p\beta}{p(x-\beta)}}
g_{n-2,p-2,\beta}(\lambda_3,\cdots,\lambda_n)d\lambda_3\cdots d\lambda_n
\notag\\
\sim&~e^{-(1+o(1))\frac{p\delta_n^2}{4\beta}-\Theta(\log p)},\notag\\
&
\label{xgj5}
\end{align}
where in the first step, we used $e^{-2\beta\frac{\lambda_2-p\beta}{p(x-\beta)}}\leq 1$ since $\lambda_2>p\beta$ and the change of variable for the first integral from $\lambda_2$ to $\lambda_2+p(\beta+\delta_n)$, and the last step follows from a similar argument as the approximations of \eqref{int2} and \eqref{int3} in the proof of Lemma \ref{lemma0}. This implies that
\begin{eqnarray*}
\eqref{uppersq}
&\lesssim&
   \Theta(1)n^2A_n^2(px-p\beta)^{2\beta(n-1)}(px)^{\beta(p-n+1)-2}e^{- px}
\times e^{-(1+o(1))\frac{p\delta_n^2}{4\beta}-\Theta(\log p)}.
\end{eqnarray*}
Together with the tail probability expression in Theorem \ref{theorem1}, this and the fact $p\delta_n^2/(p^{-1}n^2)\to\infty$ conclude  that
$
\eqref{uppersq}=o(1) P(\lambda_{(1)}>px)^2.
$
We then get (\ref{muskmellon}).

\smallskip
{\it Proof of (\ref{watermellon})}. From the result in {\it Step 1}, we only need to focus on the case when $px>\lambda_2$, and we have a similar  upper bound as in \eqref{uppersq}:
\begin{eqnarray}\label{uppern}
&&E^Q\left[L_p^2;
 \lambda_{(1)}>px>\lambda_{(2)}, \lambda_{(n)}<p(\beta-\delta_n)\right]\notag\\
&\lesssim&
  \Theta(1)n^2A_n^2(px-p\beta)^{2\beta(n-1)}e^{-\frac{x-\beta}{2x}px}
\notag\\
&&\times \int_{px}^\infty
\lambda_1^{\beta(p-n+1)-2}
e^{2\beta(n-1)\frac{\lambda_1-px}{p(x-\beta)}-\frac{x+\beta}{2x} \lambda_1}
d\lambda_1\notag\\
&&\times
\int_{ \lambda_2>\cdots>\lambda_n\atop \lambda_2<px, \lambda_n<p(\beta-\delta_n)}
e^{-2\beta\sum_{i=2}^n\frac{\lambda_i-p\beta}{p(x-\beta)}}
g_{n-1,p-1,\beta}(\lambda_2,\cdots,\lambda_n)d\lambda_2\cdots d\lambda_n\notag\\
&\lesssim&
  \Theta(1)n^2A_n^2(px-p\beta)^{2\beta(n-1)}(px)^{\beta(p-n+1)-2}e^{- px}\notag\\
&&\times \int_{\lambda_2>\lambda_3>\cdots>\lambda_n\atop \lambda_2<px, \lambda_n<p(\beta-\delta_n)}
(n-1)A_{n-1} \prod_{i=2}^{n-1} (\lambda_i-\lambda_n)^\beta
\lambda_n^{\frac{\beta(p-n+1)}{2}-1}
e^{-\frac{\lambda_n}{2}-2\beta\frac{\lambda_n-p\beta}{p(x-\beta)}}\notag\\
&&\quad \times
e^{-2\beta\sum_{i=2}^{n-1}\frac{\lambda_i-p\beta}{p(x-\beta)}}
g_{n-2,p-2,\beta}(\lambda_2,\cdots,\lambda_{n-1})d\lambda_2\cdots d\lambda_n.
\notag\\
\end{eqnarray}
Use the upper bound
$$(\lambda_i-\lambda_n)= (p\delta_n)
\left(1+\frac{\lambda_i-p\beta}{p\delta_n}-\frac{\lambda_n-p(\beta-\delta_n)}{p\delta_n}\right)
\leq  (p\delta_n) e^{\frac{\lambda_i-p\beta}{p\delta_n}-\frac{\lambda_n-p(\beta-\delta_n)}{p\delta_n}}$$
to get
\begin{eqnarray*}
&&\eqref{uppern}\\
&\lesssim&
\Theta(1) n^2A_n^2(px-p\beta)^{2\beta(n-1)}(px)^{\beta(p-n+1)-2}e^{- px}
(n-1)A_{n-1}(p\delta_n)^{\beta(n-2)}\\
&& \int_{\lambda_n<p(\beta-\delta_n)}
\lambda_n^{\frac{\beta(p-n+1)}{2}-1}
e^{-\beta(n-2)\frac{\lambda_n-p(\beta+\delta_n)}{p\delta_n}-\frac{\lambda_n}{2}
-2\beta\frac{\lambda_n-p\beta}{p(x-\beta)}}d\lambda_n\notag\\
&& \times
\int_{\lambda_2>\cdots>\lambda_{n-1}\atop \lambda_2<px}
e^{\beta\sum_{i=2}^{n-1}\frac{\lambda_i-p\beta}{p\delta_n}-2\beta\sum_{i=2}^{n-1}\frac{\lambda_i-p\beta}{p(x-\beta)}}\\
&&\times g_{n-2,p-2,\beta}(\lambda_2,\cdots,\lambda_{n-1})d\lambda_2\cdots d\lambda_{n-1}\\
&\lesssim&  \Theta(1)  n^2A_n^2(px-p\beta)^{2\beta(n-1)}(px)^{\beta(p-n+1)-2}e^{- px}
\times e^{-(1+o(1))\frac{p\delta_n^2}{4\beta}-\Theta(\log p)},
\end{eqnarray*}
where the last step follows from the same argument as in the proof of \eqref{xgj5}.
Together with the result in Theorem \ref{theorem1}, we obtain (\ref{watermellon}).
\end{proof}~\\

\section{Proof of Theorem 3}\label{aaax}
 \begin{proof}[Proof of Theorem \ref{Thmconstant}]
Consider the case when $p/n\to \gamma\in[1,\infty).$
{Recall the definition of $\tilde Q$ in \eqref{tildeLp} and $\tilde L_p=\frac{dP}{d\tilde{Q}}1_{(\lambda_{(1)} > px)}$.} Write
 \begin{eqnarray*}
&&E^{\tilde Q}\left[\tilde L_p^2\right]= E^{\tilde Q}\biggr[\left(\frac{dP}{d\tilde Q}\right)^2; \lambda_{(1)}>pM\biggr]
+E^{\tilde Q}\biggr[\left(\frac{dP}{d\tilde Q}\right)^2;px< \lambda_{(1)}<pM\biggr],
\end{eqnarray*}
where $M$ is some big constant. We first show that
  \begin{eqnarray}\label{Street_light}
\lim_{M\to\infty}\limsup_{n\to\infty}
\frac{1}{n}\log E^{\tilde Q}\biggr[\left(\frac{dP}{d\tilde Q}\right)^2; \lambda_{(1)}>pM\biggr]
&=&-\infty.
\end{eqnarray}
In fact, by (\ref{tildeLp}),
  \begin{eqnarray*}
&&\lim_{M\to\infty}\limsup_{n\to\infty}
\frac{1}{n}\log
E^{\tilde Q}\biggr[\left(\frac{dP}{d\tilde Q}\right)^2;  \lambda_{(1)}>pM\biggr]\\
&=&
\lim_{M\to\infty}\limsup_{n\to\infty}
\frac{1}{n}\log
E^{\tilde Q}\biggr[
\biggr(\frac{nA_n  \prod_{i=2}^n(\lambda_{(1)}-\lambda_{(i)})^{\beta}
\cdot \lambda_{(1)}^{\frac{\beta(p-n+1)}{2}-1}\cdot
e^{-\frac{1}{2} \lambda_{(1)}}}
{J_{\beta,x} e^{-J_{\beta,x} (\lambda_{(1)}-px\vee\lambda_{(2)})}\cdot I_{(\lambda_{(1)}>px\vee \lambda_{(2)})}}\biggr)^2;\\
&&\quad\quad\quad\quad\quad\quad\quad\quad\quad\quad\quad\quad\quad\quad\quad
 \lambda_{(1)}>pM\biggr]
\notag\\
&\leq&
\lim_{M\to\infty}\limsup_{n\to\infty}\frac{1}{n}\log
E_{\lambda_{(2)}}\biggr[
\int_{\lambda_1>pM,\atop \lambda_1>\lambda_{(2)}}
J_{\beta,x}^{-2}{n^2A^2_n \lambda_1^{{\beta(p+n-1)}-2}e^{-\lambda_1+2J_{\beta,x} (\lambda_1-px\vee\lambda_{(2)})}}\\
&&\quad\quad\quad\quad\quad\quad\quad\quad\quad\quad\quad\quad\quad\quad\quad
\times J_{\beta,x}e^{-J_{\beta,x} (\lambda_1-px\vee\lambda_{(2)})}
  d\lambda_1\biggr]\\
  &\leq&
\lim_{M\to\infty}\limsup_{n\to\infty}\frac{1}{n}\log \int_{\lambda_1>pM}{J_{\beta,x}^{-1}}
n^2A^2_n
 \lambda_1^{{\beta(p+n-1)}-2}\cdot
e^{-\lambda_1+J_{\beta,x} \lambda_1-J_{\beta,x} px}  d\lambda_1\\
  &\leq&
\lim_{M\to\infty}\limsup_{n\to\infty}\frac{1}{n}\log
\int_{0}^\infty
n^2A^2_n
(pM)^{{\beta(p+n-1)}-2}\\
&&\quad\quad\quad\quad\quad\quad\quad\quad\quad
\times e^{({\beta(p+n-1)}-2)\lambda_1/(pM)-(1-J_{\beta,x} ) (\lambda_1+pM)-J_{\beta,x} px}
  d\lambda_1\\
&=&
\lim_{M\to\infty}\limsup_{n\to\infty}\frac{1}{n}\log
[A^2_n
(pM)^{{\beta(p+n-1)}-2}
e^{-(1-J_{\beta,x} ) pM-J_{\beta,x} px}]
 =-\infty,
  \end{eqnarray*}
where $E_{\lambda_{(2)}}$ denotes the expectation with respect to $\lambda_{(2)}$. In particular,  in the second step we used $\lambda_{(1)}-\lambda_{(i)}<\lambda_{(1)}$ and the fact that under $\tilde Q$ the conditional density of $\lambda_{(1)}$ given $\lambda_{(2)}$  is $J_{\beta,x}e^{-J_{\beta,x} (\lambda_1-px\vee\lambda_{(2)})}$  (see Section 2.2). The third step follows from the fact $e^{-J_{\beta,x} (px\vee\lambda_{(2)})}<e^{-J_{\beta,x} px}$. In the fourth step we changed variable $\lambda_1$ to $\lambda_1+px$ for the integral. In the last step we used the approximation in \eqref{logA}, which gives, for $p/n\to\gamma$,
\begin{eqnarray}\label{logA3}
&&\quad\quad \log A_n\\
&&\sim~ -\frac{\beta p}{2}\log p -\frac{\beta p}{2}\left(\log {\beta}-1\right)
-\frac{\beta n}{2}\log n-\frac{\beta n}{2}\left(\log{\beta}-1\right)+O(\log n)
\notag\\
&&\sim~ -\frac{\beta}{2}(\gamma+1) n\log n-\frac{\beta}{2} [(\gamma+1)(\log\beta-1)+\gamma\log\gamma]\,n+O(\log n).\notag
\end{eqnarray}

From (\ref{Street_light}), we only need to focus on $E^{\tilde Q}[({dP}/{d\tilde Q})^2;px< \lambda_{(1)}<pM].$
 Recall that $\sigma_{\beta}$ denotes the equilibrium  measure for the large deviations of the empirical distribution of eigenvalues $(\lambda_1/n,\cdots,\lambda_n/n)$ under $P$ (Lemma 2.6.2 from Anderson et al 2010). Let $B(\epsilon)$ be the ball of probability measures defined on $[0,2\gamma M]$ with radius $\epsilon$ around
 $\sigma_{\beta}$ under the following metric $\rho$ that generates the weak convergence of probability measures on $\mathbb{R}$: for two probability
measures  $\mu$ and $\nu$ on $\mathbb{R}$,
\begin{eqnarray}\label{Newton}
\rho(\mu,\nu)=\sup_{\|h\|_L\leq 1}\Big| \int_{\mathbb{R}} h(x)\mu(dx)-\int_{\mathbb{R}} h(x) d\nu(dx)
\Big|
\end{eqnarray}
where $h$ is a bounded Lipschitz function defined on $\mathbb{R}$ with
$\|h\| = \sup_{x\in\mathbb{R}} |h(x)|$ and $\|h\|_L=\|h\| +\sup_{x\neq y}|h(x)-h(y)|/|x-y|.$

Let ${\cal L}_{n-1}$ be the empirical measure of  $(\lambda_2/(n-1),\cdots,\lambda_n/(n-1))$ with $\lambda_2,\cdots,\lambda_n$ being constructed as in {\it Step 1} of Algorithm \ref{algorithm1} in Section \ref{subsubsec:Sim}. {Notice $(n-1)\mbox{Supp}({\cal L}_{n-1}) \subset [0, p M]$ under the restriction $\lambda_1\leq pM$}. For any $\epsilon>0$, we first consider the following expectation
  \begin{eqnarray}
&&\limsup_{n\to\infty}
\frac{1}{n}\log E^{\tilde Q}\left[\left(\frac{dP}{d\tilde Q}\right)^2; pM>\lambda_1>px,\lambda_1>\cdots>\lambda_n,
 {\cal L}_{n-1} \notin B(\epsilon)\right]
 \notag\\
&&\leq~
\limsup_{n\to\infty}\frac{1}{n}\log E^{\tilde Q}\biggr[
\biggr(\frac{nA_n\times
\prod_{i=2}^n(pM)^{\beta}
\cdot \lambda_1^{\frac{\beta(p-n+1)}{2}-1}\cdot
e^{-\frac{1}{2} \lambda_1}}
{J_{\beta,x} e^{-J_{\beta,x} (\lambda_1-px)}}
\biggr)^2;
\notag\\
&&~~~~~~~~~~~~~~~~~~~~~~~~~~~~~~
 pM>\lambda_1>px,  {\cal L}_{n-1} \notin B(\epsilon)\biggr],
 \label{aaaa}
\end{eqnarray}
where the above inequality follows from the fact that $\lambda_1-\lambda_i<pM$ and $\lambda_{(2)}\vee px>px$.
Note that $\frac{nA_n\times
\prod_{i=2}^n(pM)^{\beta}
\cdot \lambda_1^{\frac{\beta(p-n+1)}{2}-1}\cdot
e^{-\frac{1}{2} \lambda_1}}
{J_{\beta,x} e^{-J_{\beta,x} (\lambda_1-px)}}=e^{O(n\log n)}$ under the assumption that $p/n\to\gamma$ and $\lambda_1<pM$.
We have
\begin{eqnarray*}
\eqref{aaaa}
&\leq& \limsup_{n\to\infty}\frac{1}{n}\log [e^{O(n\log n)}\tilde Q( pM>\lambda_1>px,  {\cal L}_{n-1} \notin B(\epsilon))]\\
&\leq& \limsup_{n\to\infty}\Big\{O(\log n)+\frac{1}{n}\log P({\cal L}_{n-1} \notin B(\epsilon))\Big\}.
\end{eqnarray*}
The large deviation result for ${\cal L}_{n-1}$ (Theorem 2..6.1 in Anderson et al 2010) implies that
$\limsup_{n\to\infty}\frac{1}{n^2}\log P({\cal L}_{n-1} \notin B(\epsilon))<0.$
 Thus,
\begin{eqnarray}\label{supermoon}
\eqref{aaaa}=-\infty
\end{eqnarray}
for any $\epsilon>0$. From (\ref{Street_light}) and (\ref{supermoon}), to estimate $E^{\tilde Q}[\tilde L_p^2]$, we  need to further explore the expectation under the restriction  $\Omega_n:=\{px<\lambda_{(1)}<pM\ \mbox{and} \ {\cal L}_{n-1} \in B(\epsilon)\}$.
Let $\Phi(z,\epsilon)=\sup_{\mu\in B(\epsilon)}\int \log(z-y)[\mu(dy)-\sigma_\beta(dy)].$
We have
  \begin{eqnarray*}
W_n&:=&
E^{\tilde Q}\left[\left(\frac{dP}{d\tilde Q}\right)^2;\Omega_n\right]
\notag\\
&= &
E^{\tilde Q}\left[
\biggr(\frac{nA_n  \prod_{i=2}^n(\lambda_{(1)}-\lambda_{(i)})^{\beta}
\cdot \lambda_{(1)}^{\frac{\beta(p-n+1)}{2}-1}\cdot
e^{-\frac{1}{2} \lambda_{(1)}}}
{J_{\beta,x} e^{-J_{\beta,x} (\lambda_{(1)}-px\vee\lambda_{(2)})}\cdot I_{(\lambda_{(1)}>px\vee \lambda_{(2)})}}\biggr)^2;
~ \Omega_n\right]
\notag\\
&\leq &O(1)n^2 A_n^2
E^{\tilde Q}\left[
e^{2\beta\sum_{i=2}^n\log(\lambda_{(1)}-\lambda_{(i)})}
 \lambda_{(1)}^{{\beta(p-n+1)}-2}
e^{-\lambda_{(1)} +2J_{\beta,x} (\lambda_{(1)}-px)};
~ \Omega_n\right]\\
&\leq& O(1)n^2 A_n^2 n^{2\beta n}
\notag\\
&& \times
\int_{px}^{pM}
e^{2\beta(n-1)\Phi(\frac{\lambda_1}{n-1},\epsilon)+2\beta(n-1)\int \log(\frac{\lambda_1}{n-1}-y)\sigma_\beta(dy)}\\
&& \times \lambda_1^{\beta(p-n+1)-2}
e^{-\lambda_1+J_{\beta,x} (\lambda_1-px)} d\lambda_1,
\notag
\end{eqnarray*}
where in the second step we simply used the inequality $e^{-J_{\beta,x}px\vee\lambda_2}<e^{-J_{\beta,x}px}$ and in the last step we used that for ${\cal L}_{n-1} \in B(\epsilon)$
\begin{eqnarray*}
&&\sum_{i=2}^n\log(\lambda_{(1)}-\lambda_{(i)})\\
& = & (n-1)\int_{\mathbb{R}}\log\Big(\frac{\lambda_{(1)}}{{n-1}}-y\Big)\,{\cal L}_{n-1}(dy)+(n-1)\log (n-1)\\
&\leq & (n-1) \Phi\Big(\frac{\lambda_{(1)}}{n-1},\epsilon\Big)
+(n-1)\int_\mathbb{R} \log\Big(\frac{\lambda_{(1)}}{n-1}-y\Big)\sigma_\beta(dy)+n\log n.
\end{eqnarray*}
Observe that
$$\Phi\Big(\frac{\lambda_{(1)}}{n-1},\epsilon\Big)
\leq \sup_{z\in[\frac{px}{n-1}, \frac{pM}{n-1}]}\Phi(z,\epsilon)$$
under the constraint $px<\lambda_{(1)}<pM$ and that
\begin{eqnarray*}
&& \int_\mathbb{R} \log\Big(\frac{\lambda_{(1)}}{n-1}-y\Big)\sigma_\beta(dy)\\
&=&
\int_\mathbb{R} \log(\frac{px}{n-1}-y)\sigma_\beta(dy)
+\int_\mathbb{R} \log\Big(1+\frac{\lambda_1-px}{px-(n-1)y}\Big)\sigma_\beta(dy)\\
&\leq&
\int_\mathbb{R} \log(\frac{px}{n-1}-y)\sigma_\beta(dy)
+\int_\mathbb{R}  \frac{\lambda_1-px}{px-(n-1)y}\sigma_\beta(dy).
\end{eqnarray*}
 It follows that
\begin{align}
W_n
\leq&~ O(1)n^2 A_n^2 n^{2\beta n}
e^{2\beta(n-1) \sup_{z\in[\frac{px}{n-1},\frac{pM}{n-1}]}\Phi(z,\epsilon)
+2\beta(n-1)\int \log(\frac{px}{n-1}-y)\sigma_\beta(dy)}
\notag\\
& \times \int_{px}^{pM}
e^{2\beta(n-1) \int \frac{\lambda_1-px}{px-(n-1)y}d\sigma_\beta(y)}
\cdot \lambda_1^{\beta(p-n+1)-2}\cdot
e^{-(1-J_{\beta,x} )\lambda_1-J_{\beta,x} px} d\lambda_1 \notag\\
=&~
O(1)n^2 A_n^2 n^{2\beta n}
e^{2\beta(n-1) \sup_{z\in[\frac{px}{n-1},\frac{pM}{n-1}]}\Phi(z,\epsilon)
+2\beta(n-1)\int \log(\frac{px}{n-1}-y)\sigma_\beta(dy)}
\notag\\
&~ \times \int_{0}^{p(M-x)}
e^{2\beta(n-1) \int \frac{\lambda_1}{px-(n-1)y}d\sigma_\beta(y)}
\cdot (\lambda_1+px)^{\beta(p-n+1)-2}\notag\\
&~\times
e^{-(1-J_{\beta,x} )(\lambda_1+px)-J_{\beta,x} px} d\lambda_1 \notag\\
\leq&~
 O(1) n^2 A_n^2 n^{2\beta n}
e^{2\beta(n-1)\sup_{z\in[\gamma'x, 2\gamma M]}\Phi(z,\epsilon)
+2\beta(n-1)\int \log(\frac{px}{n-1}-y)\sigma_\beta(dy)}
\notag\\
&~ \times (px)^{\beta(p-n+1)-2}e^{- px}\notag\\
&~ \times
\int_{0}^{p(M-x)}
e^{2\beta(n-1) \int \frac{\lambda_1}{p x-(n-1)y}d\sigma_\beta(y)
+(\beta(p-n+1)-2)\frac{\lambda_1}{px}
-(1-J_{\beta,x} )\lambda_1} d\lambda_1,
\label{xiaoshu}
\end{align}
where $\gamma'\in (x^*/x, \gamma)$; in the second  step we changed the variable $\lambda_1$ to $(\lambda_1+px)$ for the integral and in the last step we used $(\lambda_1+px)^{{\beta(p-n+1)-2}}\leq (px)^{\beta(p-n+1)-2}e^{(\beta(p-n+1)-2)\lambda_1/(px)}$.

Next we show that
\begin{eqnarray}\label{High_tower}
\limsup_{\epsilon\to 0}\sup_{z\in[\gamma'x,2\gamma M]}\Phi(z,\epsilon)\leq 0.
\end{eqnarray}
Recall the definition of $B(\epsilon)$ and (\ref{Newton}). For any $z\in [\gamma'x,2\gamma M]$ and $\mu\in B(\epsilon)$,
let ${\cal S}_1=\{y\in \mbox{supp}(\sigma_\beta)\cup \mbox{supp}(\mu): |z-y|> \eta\}$ and
${\cal S}_2=\{y\in \mbox{supp}(\sigma_\beta)\cup \mbox{supp}(\mu): |z-y|\leq \eta\}$,
 where $\mbox{supp}(\mu)$ is the support of measure $\mu$ and $\eta$ is a small constant such that $\eta<\min\{\gamma'x-x^*,1\}$.
 Note that $\mbox{supp}(\sigma_\beta)\subset {\cal S}_1.$
 {Given $z\in [\gamma'x,2\gamma M]$, set $f_z(y):=\log(|z-y|)$ for $y\in {\cal S}_1$.
 Then, the Lipschitz norms of the set of functions $\{f_z(\cdot);\, z\in [\gamma'x,2\gamma M]\}$ is bounded by a constant $C<\infty$. By the definition of $\rho(\cdot, \cdot)$ in (\ref{Newton}),
\begin{align*}
&\sup_{z\in [\gamma'x,2\gamma M]}\int_{\mathbb{R}} \log(|z-y|)[\mu(dy)-\sigma_\beta(dy)]\\
\leq &
\sup_{z\in [\gamma'x,2\gamma M]}\int_{{\cal S}_1} \log(|z-y|)[\mu(dy)-\sigma_\beta(dy)]
+\sup_{z\in [\gamma'x,2\gamma M]}\int_{{\cal S}_2} \log(|z-y|)\mu(dy)
\\
\leq&
\sup_{z\in [\gamma'x,2\gamma M]}\int_{{\cal S}_1} f_z(y)[\mu(dy)-\sigma_\beta(dy)]\\
 \leq & ~C \rho(\mu, \sigma_\beta)<C\epsilon
\end{align*}
for any $\mu\in B_{\epsilon}$. This implies that $\sup_{z\in[\gamma'x,2\gamma M]}\Phi(z,\epsilon)<C\epsilon$. Then (\ref{High_tower}) follows.}

From \eqref{rate}, we know that the integral term in \eqref{xiaoshu} is $\sim O(1).$
Joining this with (\ref{Street_light}),  \eqref{logA3} and (\ref{supermoon}), we conclude
\begin{eqnarray*}
&&\limsup_{n\to\infty}\frac{1}{n}\log
 E^{\tilde Q}\left[\tilde L_p^2\right]\\
&\leq& \lim_{\epsilon\to 0}\limsup_{n\to\infty}\frac{1}{n}\log\{ \mbox{display }\eqref{xiaoshu}\}\\
&=& 2\beta\int \log(\gamma x-y)\sigma_\beta(dy)-\gamma x+\beta(\gamma-1)\log(\gamma x)-2\alpha_\beta,\\
& = & -2I_{\beta}(\gamma x)
\end{eqnarray*}
where $\alpha_\beta= \frac{\beta}{2}[(\gamma+1)(\log\beta-1)+\gamma\log \gamma]$  and $I_{\beta}$ is defined as in \eqref{LDP}.
By the large deviation result in \eqref{LDP},  we have $\lim_{n\to\infty}\frac{1}{n}\log P(\lambda_{(1)}>px)=-I_{\beta}(\gamma x)$. Hence
$$\limsup_{n\to\infty}\frac{\log E^{\tilde Q}\left[\tilde L_p^2\right]}
{2\log P(\lambda_{(1)}>px)}\leq 1.$$
On the other hand, review $L_p=\frac{dP}{d\tilde Q}$, we know  $E^{\tilde Q}[\tilde L_p^2]\geq P(\lambda_{(1)}>px)^2$ by H\"{o}lder's inequality. The two facts imply the desired conclusion.
\end{proof}

\end{document}